\documentclass{amsart}
\usepackage{hyperref}

\usepackage{amsrefs,amsthm,amsmath,amsfonts,amssymb,graphicx,latexsym,color}

\newtheorem{theorem}{Theorem}

\newtheorem{lemma}[theorem]{Lemma}

\newtheorem{proposition}{Proposition}

\theoremstyle{definition}
\newtheorem*{remark}{Remark}

\newcommand{\eqdef}{\overset{\mbox{\tiny{def}}}{=}}

\newcommand{\R}{{\mathbb R}}
\newcommand{\NgE}{{K \ge \ksob}}
\newcommand{\ang}[1]{ \left< {#1} \right> }
\newcommand{\ext}[1]{ \underline{ {#1} } }
\newcommand{\ind}{ {\mathbf 1}}
\newcommand{\sph}{{\mathbb S}^{\dim-1}}
\newcommand{\nsm}{|}
\newcommand{\nl}{\left|}

\newcommand{\nr}{\right|}
\newcommand{\set}[2]{ \left\{ #1 \ \left| \ #2 \right. \right\} }
\newcommand{\ip}[2]{ \left< #1 , #2 \right>}

\newcommand{\threed}{{\mathbb R}^n}
\newcommand{\last}{{n+1}}
\renewcommand{\dim}{n}

\newcommand{\domain}{\mathbb{T}^\dim}

\newcommand{\teePLUSop}{T^{k,\ell}_{+}}
\newcommand{\teeMINUSop}{T^{k,\ell}_{-}}
\newcommand{\teeSTARop}{T^{k,\ell}_{*}}

\newcommand{\nPiece}{\mathcal{N}}
\newcommand{\kPiece}{\mathcal{K}}

\newcommand{\nullSpace}{N(L)}

\newcommand{\opGstar}{\Gamma_{*}}
\newcommand{\LopGstar}{\Gamma_{*}^\ell}

\newcommand{\TaylorP}{\zeta}

\newcommand{\CurveP}{\zeta}
\newcommand{\CurveEP}{\ext{\zeta}}
\newcommand{\starCurveP}{\zeta_*}
\newcommand{\starCurveEP}{\ext{\zeta_*}}

\newcommand{\Ctheta}{\vartheta}

\newcommand{\HARDvDER}{V}
\newcommand{\HARDxDER}{X}
\newcommand{\HARDspaceh}{H^{\HARDxDER;\HARDvDER}_{\ell}}
\newcommand{\hardNspace}{N^{s,\gamma}_{\ell,\HARDxDER;\HARDvDER}}

\newcommand{\nspace}{N^{s,\gamma}_{\ell,K}}
\newcommand{\spacen}{N^{s,\gamma}}
\newcommand{\spaceELLn}{N^{s,\gamma}_\ell}
\newcommand{\spaceh}{H^{K}_{\ell}}
\newcommand{\ksob}{K^*_n}

\newcommand{\Btheta}{\vartheta}
\newcommand{\iter}{m}

\newcommand{\spaceU}{H}

\newcommand{\macroCOE}{\lambda}

\numberwithin{equation}{section}
\numberwithin{theorem}{section}
\numberwithin{proposition}{section}

\begin{document}
\title[Global Solutions of the Boltzmann Equation without Cutoff]{Global Classical Solutions of the Boltzmann Equation without  Angular Cut-off}

\author[P. T. Gressman]{Philip T. Gressman}
\address{University of Pennsylvania, Department of Mathematics, David Rittenhouse Lab, 209 South 33rd Street, Philadelphia, PA 19104-6395, USA} 
\email{gressman at math.upenn.edu \& strain at math.upenn.edu}
\urladdr{http://www.math.upenn.edu/~gressman/ \& http://www.math.upenn.edu/~strain/}
\thanks{P.T.G. was partially supported by the NSF grant DMS-0850791.}

\author[R. M. Strain]{Robert M. Strain}
\thanks{R.M.S. was partially supported by the NSF grant DMS-0901463.}

\keywords{Kinetic Theory, Boltzmann equation, long-range interaction, non cut-off, soft potentials, hard potentials, fractional derivatives, anisotropy, Harmonic analysis. \\
\indent 2010 {\it Mathematics Subject Classification.}  35Q20, 35R11, 76P05, 82C40, 35H20, 35B65, 26A33}

\date{Revised version of (arXiv:0912.0888v1) and (arXiv:1002.3639v1): July 2010}

\begin{abstract}
This work proves the global stability of the Boltzmann equation (1872) with the physical collision kernels derived by Maxwell in 1866 for the full range of inverse-power intermolecular potentials, $r^{-(p-1)}$ with $p>2$,
for initial perturbations of the Maxwellian equilibrium states,
 as announced in \cite{gsNonCutA}.  
We  more generally cover collision kernels with parameters $s\in (0,1)$ and $\gamma$ satisfying 
$\gamma  > -n$ in arbitrary dimensions $\mathbb{T}^n \times \mathbb{R}^n$ with $n\ge 2$.  
Moreover, we prove rapid convergence as predicted by the celebrated Boltzmann $H$-theorem.
When $\gamma \ge -2s$, we have  exponential time decay to the Maxwellian equilibrium states.  When $\gamma <-2s$, our solutions decay polynomially fast in time with any rate.  
These results are completely constructive.  Additionally, we prove  sharp constructive upper and lower bounds for the linearized collision operator in terms of a geometric fractional Sobolev norm; we thus observe that a spectral gap exists only when $\gamma  \ge -2s$, as conjectured in  Mouhot-Strain \cite{MR2322149}.  It will be observed that this fundamental equation, derived by both Boltzmann  and Maxwell, grants a basic example where a range of geometric fractional derivatives occur in a physical model of the natural world.  Our methods provide a new  understanding of the grazing collisions in the Boltzmann theory. 
\end{abstract}
\maketitle

\thispagestyle{empty}

\tableofcontents  

\section{Introduction, main theorem, and historical remarks}

In 1872, Boltzmann was able to derive an equation which accurately models the dynamics of a dilute gas; it has since become a cornerstone of statistical physics \cite{MR1313028,MR1307620, MR1014927,MR0156656, MR1942465,MR1379589}.  
There are many useful mathematical theories of global solutions for the Boltzmann equation,  and we will start off by  mentioning a brief few.  In 1933, Carleman \cite{MR1555365} proved existence and uniqueness of the spatially homogeneous problem with radial initial data.  For  spatially dependent theories, it was Ukai \cite{MR0363332} in 1974  who proved the existence of global classical solutions with close-to-equilibrium initial data.  Ten years later,
 Illner-Shinbrot 
\cite{MR760333}  found unique global mild solutions with near vacuum data.
Then in 1989, the work of DiPerna-Lions \cite{MR1014927} established global renormalized weak solutions for initial data without a size restriction.  We also mention recent methods introduced in the linearized regime by Guo \cite{MR2000470} in 2003 and Liu-Yang-Yu \cite{MR2043729} in 2004.
All of these methods and their generalizations apply to hard sphere particles or soft particle interactions in which there is a non-physical cut-off of an inherently nonintegrable angular singularity.

When the physically relevant effects of these angular singularities are not cut-off, 
we recall the remarkable paper by Alexandre-Villani \cite{MR1857879} from 2002, 
which proves the existence of
 DiPerna-Lions renormalized weak solutions \cite{MR1014927}
with a non-negative defect measure.  It is illustrated therein that the mass conservation they prove would imply this defect measure was zero if the solutions were sufficiently strong.  At the moment this defect measure appears difficult to characterize \cite[Appendix]{MR1857879}.

This present work contributes to the understanding of global-in-time, close to Maxwellian equilibrium  classical solutions of the Boltzmann equation without angular cut-off, that is, for long-range interactions.  This problem has been the subject of intense investigations for some time now. 
We develop  a satisfying mathematical framework 
for these solutions  for all of the collision kernels derived from the intermolecular potentials 
 and more generally.  For the hard-sphere and cut-off collision kernels, such a framework has been well established for a long time 
\cite{MR2095473,MR882376,MR0363332,MR2013332,MR2043729,MR2000470}.  The hard-sphere kernel applies formally in the limit when $p\to\infty$.
A framework is also known for the Landau collision operator \cite{MR1946444}, which can be thought of as  the limiting model for $p=2$. 
The important case of the Boltzmann equation with the collision kernels derived by Maxwell is then the last case for the intermolecular potentials in which this framework has remained open.  
This article provides a solution to this longstanding, well-known problem.
We will discuss more of the  historical background in Section \ref{sec:reviewNON}.

The model which is the focus of this research is  the {\em Boltzmann equation}
  \begin{equation}
  \frac{\partial F}{\partial t} + v \cdot \nabla_x F = {\mathcal Q}(F,F),
  \label{BoltzFULL}
  \end{equation}
where the unknown $F(t,x,v)$ is a nonnegative function.  For each time $t\ge 0$, $F(t, \cdot, \cdot)$ represents the density  of particles in phase space, and is often 
called the empirical measure. The spatial coordinates  are $x\in\mathbb{T}^n$, and the velocities are $v\in\mathbb{R}^n$ with $n\ge 2$.
The  {\em Boltzmann collision operator} ${\mathcal Q}$ 
is a bilinear operator which acts only on the velocity variables $v$ and is local in $(t,x)$ as 
  \begin{equation*}
  {\mathcal Q} (G,F)(v) \eqdef 
  \int_{\mathbb{R}^n}  dv_* 
  \int_{\mathbb{S}^{n-1}}  d\sigma~ 
  B(v-v_*, \sigma) \, 
  \big[ G'_* F' - G_* F \big].
  \end{equation*} 
Here we are using the standard shorthand $F = F(v)$, $G_* = G(v_*)$, $F' = F(v')$, $G_*^{\prime} = G(v'_*)$. 
In this expression,  $v$, $v_*$ and $v'$, $v' _*$  are 
the velocities of a pair of particles before and after collision.  They are connected through the formulas
  \begin{equation}
  v' = \frac{v+v_*}{2} + \frac{|v-v_*|}{2} \sigma, \qquad
  v'_* = \frac{v+v_*}{2} - \frac{|v-v_*|}{2} \sigma,
  \qquad \sigma \in \mathbb{S}^{n-1}.
  \label{sigma}
  \end{equation}
This representation of ${\mathcal Q}$ results from making a choice for the parameterization of the set of solutions to the physical law of elastic collisions:
\begin{equation}
\begin{array}{rcl}
v+v_* &=& v^{\prime } + v^{\prime }_*,
\\
|v|^2+|v_*|^2 &=& |v^\prime|^2+|v_*^\prime|^2.
\end{array}
\notag
\end{equation}
This choice is not unique; we specifically utilize also Carleman-type representations.

 The {\em Boltzmann collision kernel} $B(v-v_*, \sigma)$ for a monatomic gas is, on physical grounds, a non-negative function which 
only depends on the {\em relative velocity} $|v-v_*|$ and on
the {\em deviation angle}  $\theta$ through 
$\cos \theta = \ip{ k }{ \sigma}$ where $k = (v-v_*)/|v-v_*|$ and $\langle \cdot, \cdot \rangle$ is the usual scalar product in $\threed$. 
Without loss of generality we may assume that  $B(v-v_*, \sigma)$
is supported on $\ip{ k }{ \sigma} \ge 0$, i.e. $0 \le \theta \le \frac{\pi}{2}$.
Otherwise we can  reduce to this situation 
with the following standard ``symmetrization'' \cite{MR1379589}: 
$$
\overline{B}(v-v_*, \sigma)
=
\left[ 
B(v-v_*, \sigma)
+
B(v-v_*, -\sigma)
\right]
{\bf 1}_{\ip{ k }{ \sigma} \ge 0}.
$$ 
Above and generally, ${\bf 1}_{A}$ is the usual indicator function of the set $A$.

\subsection*{The Collision Kernel}

Our assumptions are as follows:  
 \begin{itemize}
 \item 
We suppose that $B(v-v_*, \sigma)$ takes product form in its arguments as
\begin{equation}
B(v-v_*, \sigma) =\Phi( |v-v_*| ) \, b(\cos \theta).
\notag
\end{equation}
In general both $b$ and $\Phi$ are non-negative functions. 
 
  \item The angular function $t \mapsto b(t)$ is not locally integrable; for $c_b >0$ it satisfies
\begin{equation}
\frac{c_b}{\theta^{1+2s}} 
\le 
\sin^{\dim -2} \theta  ~ b(\cos \theta) 
\le 
\frac{1}{c_b\theta^{1+2s}},
\quad
s \in (0,1),
\quad
   \forall \, \theta \in \left(0,\frac{\pi}{2} \right].
   \label{kernelQ}
\end{equation}
 
  \item The kinetic factor $z \mapsto \Phi(|z|)$ satisfies for some $C_\Phi >0$
\begin{equation}
\Phi( |v-v_*| ) =  C_\Phi  |v-v_*|^\gamma , \quad \gamma  \ge -2s.
\label{kernelP}
\end{equation}
In the rest of this paper these will be called ``hard potentials.'' 

  \item Our results will also apply to the more singular situation
\begin{equation}
\Phi( |v-v_*| ) =  C_\Phi  |v-v_*|^\gamma , \quad  -2s > \gamma  > -n. 
\label{kernelPsing}
\end{equation}
These will be called ``soft potentials.''

 \end{itemize}

Our main physical motivation is derived from  particles interacting according to a spherical intermolecular repulsive potential of the form
$$
\phi(r)=r^{-(p-1)} , \quad p \in (2,+\infty).
$$
For these potentials, Maxwell \cite{Maxwell1867} in 1866 showed that 
the kernel $B$ can be computed.   In dimension $n=3$, $B$ satisfies the conditions above with 
$\gamma = (p-5)/(p-1)$ 
and $s = 1/(p-1)$; 
see for instance \cite{MR1313028,MR1307620,MR1942465}.
Thus the conditions in  \eqref{kernelQ}, \eqref{kernelP}, and \eqref{kernelPsing} include all of the potentials $p > 2$ in the physical dimension $n=3$.
Note further that the Boltzmann collision operator is not well defined for $p=2$, see \cite{MR1942465}.

We will study the linearization of \eqref{BoltzFULL} around the Maxwellian 
equilibrium state 
\begin{equation}
F(t,x,v) = \mu(v)+\sqrt{\mu(v)} f(t,x,v),
\label{maxLIN}
\end{equation}
where without loss of generality
$$
\mu(v) = (2\pi)^{-n/2}e^{-|v|^2/2}.
$$
We will also suppose without restriction that the mass, momentum, and energy conservation laws for the perturbation $f(t,x,v)$ hold for all $t\ge0$ as
\begin{equation}
\int_{\domain\times \threed} ~ dx ~ dv ~ \begin{pmatrix}
      1   \\      v  \\ |v|^2
\end{pmatrix}
~ \sqrt{\mu(v)} ~ f(t,x,v) 
=
0.
\label{conservation}
\end{equation}
This condition should be satisfied initially, and then will continue to be satisfied for a suitably strong solution.  Our main interest is in global classical solutions to the Boltzmann equation \eqref{BoltzFULL} which are perturbations of the Maxwellian equilibrium states \eqref{maxLIN} for the long-range collision kernels \eqref{kernelQ}, \eqref{kernelP} and \eqref{kernelPsing}.

Our solution to this problem rests heavily on our introduction into the Boltzmann theory of the 
following weighted geometric fractional  Sobolev space:
$$
\spacen
\eqdef 
\left\{ f \in \mathcal{S}'(\mathbb{R}^n): 
 \nsm f\nsm_{\spacen} <\infty
\right\},
$$
where we specify the anisotropic norm by
\begin{equation} 
\nsm f\nsm_{\spacen}^2 
\eqdef 
\nsm f\nsm_{L^2_{\gamma+2s}}^2 + \int_{\mathbb{R}^n} dv \int_{\mathbb{R}^n} dv' ~
(\ang{v}\ang{v'})^{\frac{\gamma+2s+1}{2}}
 \frac{(f' - f)^2}{d(v,v')^{n+2s}} 
\ind_{d(v,v') \leq 1}. \label{normdef} 
\end{equation}
This space includes the weighted $L^2_\ell$ space, for $\ell\in\mathbb{R}$, with norm given by
$$
\nsm f\nsm_{L^2_{\ell}}^2 
\eqdef
\int_{\mathbb{R}^n} dv ~ 
\ang{v}^{\ell}
~
|f(v)|^2.
$$
The weight is
$
\ang{v}
\eqdef \sqrt{1+|v|^2}.
$
The fractional differentiation effects are measured 
using the following anisotropic metric $d(v,v')$ on the ``lifted'' paraboloid:
$$
d(v,v') \eqdef \sqrt{ |v-v'|^2 + \frac{1}{4}\left( |v|^2 -  |v'|^2\right)^2}.
$$
The inclusion of the quadratic difference $|v|^2 - |v'|^2$ is essential; it is not a lower order term.
Heuristically, this metric encodes the anisotropic changes in the power of the weight, which are 
entangled with the non-local fractional differentiation effects.

The space $\spacen$ is essentially a weighted anisotropic Sobolev norm; this particular feature was conjectured in \cite{MR2322149}.  
We see precisely that if $\mathbb{R}^n$ is identified with a paraboloid in $\mathbb{R}^{n+1}$ by means of the mapping $v \mapsto (v, \frac{1}{2}|v|^2)$ and $\Delta_{P}$ is defined to be the Laplacian on the paraboloid induced by the Euclidean metric on $\mathbb{R}^{n+1}$ then
\[ 
\nsm f \nsm_{\spacen}^2 \approx \int_{\mathbb{R}^n} dv \ang{v}^{\gamma+2s} \left| (I- \Delta_{P})^{\frac{s}{2}} f (v) \right|^2. 
\]
The rest of our Sobolev spaces are defined in Section \ref{sec:FuncSp} just below.

We may now state our first main result as follows:

\begin{theorem}
\label{SG1mainGLOBAL}  
(Hard potentials)
Fix $\HARDxDER \ge \ksob$, the number of spatial derivatives, $0\le \HARDvDER \le \HARDxDER$ the number of velocity derivatives, and $\ell \ge0$.  
Suppose \eqref{kernelQ} and \eqref{kernelP}.
Choose initially $f_0(x,v) \in  \HARDspaceh(\domain \times \threed)$ in \eqref{maxLIN} which satisfies \eqref{conservation}.  
There is an $\eta_0>0$ such that if $\| f_0 \|_{\HARDspaceh} \le \eta_0$, then there exists a unique global strong solution to the Boltzmann equation \eqref{BoltzFULL}, in the form 
\eqref{maxLIN}, which satisfies
$$
f(t,x,v) \in 
L^\infty_t( [0,\infty); \HARDspaceh (\domain\times \threed ))
\cap
L^2_t( (0,\infty); \hardNspace  (\domain\times \threed )).
$$
Moreover, we have exponential decay to equilibrium.  
For some fixed $\lambda >0$,
$$
\| f(t) \|_{\HARDspaceh (\domain\times \threed )} \lesssim e^{-\lambda t} 
\| f_0 \|_{\HARDspaceh (\domain\times \threed )}.
$$
 We also have positivity, i.e. $F= \mu + \sqrt{\mu} f \ge 0$ if $F_0= \mu + \sqrt{\mu} f_0 \ge 0$.
\end{theorem}

We will now make a few comments on Theorem \ref{SG1mainGLOBAL}.   Note that $\ksob = \lfloor \frac{n}{2} +1 \rfloor$, which is the smallest integer strictly greater than $\frac{n}{2}$, is the critical number of whole derivatives required to use the Sobolev embedding theorems, such as \eqref{sobolevE}, in $n$-dimensions.  
Now in dimension three, for example, in the above theorem we only need $\HARDxDER \ge 2$ derivatives to have a unique theory of global solutions.  This is a result of \eqref{sobolev}.  This low regularity theorem improves by a whole derivative the previously-known amount of regularity needed, even in the presence of angular cut-off.  
We note, furthermore, that aside from the exponential decay of solutions, all of our results for hard potentials such as Theorem \ref{SG1mainGLOBAL}  hold for a larger range of parameters $\gamma > \max\{-n, -\frac{n}{2} - 2s\}$.
However $\gamma  \ge -2s$ is needed for a spectral gap \eqref{sharpLinearEst}.
Thus we would only obtain rapid polynomial decay when $\gamma + 2s <0$ as in Theorem \ref{mainGLOBAL}.  We also have the following result for the soft-potentials \eqref{kernelPsing}:

\begin{theorem}
\label{mainGLOBAL}  
(Soft potentials)
Fix $K \geq 2 \ksob$, the total number of derivatives, and $\ell\ge 0$ the order of the velocity weight in our Sobolev spaces.
Suppose \eqref{kernelQ} and \eqref{kernelPsing}.
Choose initially $f_0(x,v) \in \spaceh(\mathbb{T}^n \times \mathbb{R}^n)$ in \eqref{maxLIN} which satisfies \eqref{conservation}.  
There is an $\eta_0>0$ such that if $\| f_0 \|_{\spaceh} \le \eta_0$, then there exists a unique global classical solution to the Boltzmann equation \eqref{BoltzFULL}, in the form 
\eqref{maxLIN}, which satisfies
$$
f(t,x,v) \in 
L^\infty_t \spaceh  ( [0,\infty)\times\mathbb{T}^n\times \mathbb{R}^n )
\cap
L^2_t \nspace( (0,\infty)\times\mathbb{T}^n\times \mathbb{R}^n ).
$$
Since $\gamma <- 2s$, if $\| f_0 \|_{H^K_{\ell+m}}$ is sufficiently small for $\ell,m \ge 0$ then
 we have 
$$
\| f(t) \|_{\spaceh(\mathbb{T}^n \times \mathbb{R}^n)} \le C_m (1+t)^{-m}
\| f_0 \|_{H^K_{\ell+m}(\mathbb{T}^n \times \mathbb{R}^n)}.
$$
 We also have positivity, i.e. $F= \mu + \sqrt{\mu} f \ge 0$ if $F_0= \mu + \sqrt{\mu} f_0 \ge 0$.
\end{theorem}

For the derivatives needed above, in dimension $n=3$,  we have $2K_n^* = 4$.  This is a substantial reduction in comparison to the usual regularity conditions (eight derivatives are required, for instance, by Guo \cite{MR1946444,MR2013332} for the cut-off soft potential and the Landau equation), made possible by a trick appearing after Lemma \ref{NonLinEstA}.

\subsection{Function spaces}
\label{sec:FuncSp}
We will use $\langle \cdot, \cdot \rangle$ to denote the standard $L^2(\threed)$ inner product; because of the context this should not be confused with the scalar product on $\threed$.  The notation $(\cdot, \cdot )$ will refer to the corresponding $L^2(\domain \times \threed)$ inner product. 
These spaces will be sometimes called $L^2_v$ and $L^2_v L^2_x$, respectively.  
The notation $| \cdot |$ will refer to function space norms acting on $\R^n$ only.  The analogous norms on $\domain \times \threed$ will be denoted by $\| \cdot \|$, typically using the same subscript.
For example,
$$
\| h\|_{\spacen}^2
\eqdef
\left\|~ \nsm h\nsm_{\spacen} ~ \right\|_{L^2(\domain)}^2.
$$
The multi-indices $\alpha = (\alpha^1, \alpha^2, \ldots, \alpha^n)$ and $\beta = (\beta^1, \beta^2, \ldots, \beta^n)$ will be used to record spatial and velocity derivatives, respectively.  Specifically, 
$$
\partial^{\alpha}_{\beta}
\eqdef
\partial^{\alpha^1}_{x_1} 
\partial^{\alpha^2}_{x_2} \cdots
\partial^{\alpha^n}_{x_\dim}
\partial^{\beta^1}_{v_1} 
\partial^{\beta^2}_{v_2} \cdots 
\partial^{\beta^n}_{v_\dim}. 
$$
Similarly, the notation $\partial^\alpha$ will be used when $\beta = 0$, 
and likewise for $\partial_\beta$. 
If each component of $\alpha$ is not greater than that of $\alpha_1$,
we write $\alpha \le \alpha_1$.  Also $\alpha <\alpha_1$ means 
$\alpha \le \alpha_1$ and $|\alpha |<|\alpha_1|$, 
where
$|\alpha | = \alpha^1 + \alpha^2 + \cdots + \alpha^n$ as usual.
We also define the unified weight function
$$
w(v) 
\eqdef 
\left\{
\begin{array}{ccc}
\ang{v}, &\gamma + 2s \ge 0, & \text{``hard potentials'': \eqref{kernelP}}\\
\ang{v}^{-\gamma - 2s}, & \gamma + 2s < 0, & \text{``soft potentials'': \eqref{kernelPsing}.}\\
\end{array}
\right.
$$
This somewhat non-standard terminology distinguishes exactly when a spectral gap exists for the non cut-off linearized collision operator.
In both cases  the weight goes to infinity as $|v|$ goes to infinity.  The scaling chosen above will be convenient for the soft potentials in our analysis below.
For any $\ell \in  \mathbb{R}$, the space
$ H^K_\ell(\mathbb{R}^n) $ with
 $K \ge 0$ velocity derivatives is defined by
 \begin{gather*}
|h|^2_{H^K_\ell}
=
|h|^2_{H^K_{\ell} (\mathbb{R}^n)}
\eqdef 
\sum_{|\beta| \le K}
|w^{{\ell - |\beta|}}\partial_{\beta}h |^2_{L^2 (\mathbb{R}^n)}.
\end{gather*}
The norm notation 
$
|h|^2_{H^K_{\ell,\gamma+2s}}
\eqdef 
\sum_{|\beta| \le K}
|w^{{\ell - |\beta|}}\partial_{\beta}h |^2_{L^2_{\gamma+2s} (\mathbb{R}^n)}
$
will also find utility.  We also use the  space, 
$ \HARDspaceh=\HARDspaceh \left(\domain \times \threed \right)$ with
 $\HARDxDER \ge \HARDvDER$ derivatives, given by  
 \begin{gather*}
\|h\|^2_{\HARDspaceh}\eqdef 
\|h\|^2_{\HARDspaceh \left(\domain \times \threed \right)}
=
\sum_{|\beta | \le \HARDvDER} \sum_{|\alpha|  \le \HARDxDER - |\beta |} \| w^{\ell - |\beta|} \partial_\beta^\alpha h\|^2_{L^2 (\domain \times \threed)}.
\end{gather*}
Moreover, the space $\spaceh(\domain \times \threed)$ is given by
 \begin{gather*}
\|h\|^2_{\spaceh}
=
\|h\|^2_{\spaceh (\mathbb{T}^n \times \mathbb{R}^n)}
\eqdef 
\sum_{|\alpha | + |\beta| \le K}\|w^{\ell - |\beta|}\partial^{\alpha}_{\beta}h\|^2_{L^2 (\mathbb{T}^n \times \mathbb{R}^n)}.
\end{gather*}
In Section \ref{sec:deBEest}, the following unified notation will become useful
\begin{equation}
\label{unifiedHspace}
\|h\|^2_{\spaceU}
\eqdef
\left\{
\begin{array}{cc}
\|h\|^2_{\HARDspaceh}, & \text{for the hard potentials: } (\ref{kernelP})
\\
\|h\|_{\spaceh}, &  \text{for the soft potentials: }  (\ref{kernelPsing}).
\\
\end{array}
\right.
\end{equation}
We also consider the general weighted anisotropic derivative space as in \eqref{normdef} by 
 \begin{gather*}
\nsm h \nsm^2_{{\spaceELLn}}
\eqdef 
\nsm w^{\ell} h\nsm_{L^2_{\gamma+2s}}^2 + \int_{\mathbb{R}^n} dv ~ 
\ang{v}^{\gamma+2s+1} w^{2\ell}(v)
\int_{\mathbb{R}^n} dv' 
~
 \frac{(h' - h)^2}{d(v,v')^{n+2s}} 
\ind_{d(v,v') \leq 1}.
\end{gather*}
To this velocity space we associate the total spaces
$
\| h \|_{{\spaceELLn}}
\eqdef
\| \nsm h \nsm_{{\spaceELLn}} \|_{L^2(\mathbb{T}^n)}
$
and
$
\| h\|_{\spacen}^2.
$
We also  use the anisotropic space $\nspace (\mathbb{T}^n\times\mathbb{R}^n)$, given  by
 \begin{equation*}
\|h\|^2_{\nspace}
=
\|h\|^2_{\nspace(\mathbb{T}^n\times\mathbb{R}^n)}\eqdef \sum_{|\alpha | + |\beta| \le K}
\|\partial^{\alpha}_{\beta}h\|^2_{{N^{s,\gamma}_{\ell-|\beta|}}(\mathbb{T}^n\times\mathbb{R}^n)}.
\end{equation*}
Furthermore we will use a version of this norm only in the velocity variables
$$
\nsm h \nsm_{\nspace}^2
=
\nsm h \nsm_{\nspace(\threed)}^2
\eqdef 
\sum_{|\beta| \le K}
\nsm \partial_{\beta}h\nsm^2_{{N^{s,\gamma}_{\ell-|\beta|}}(\mathbb{R}^n)}.
$$
Moreover,
$
\hardNspace
=
\hardNspace  (\domain\times \threed )
$
is given by 
 \begin{equation*}
\|h\|^2_{\hardNspace}
=
\|h\|^2_{\hardNspace(\mathbb{T}^n\times\mathbb{R}^n)}\eqdef 
\sum_{|\beta | \le \HARDvDER} \sum_{|\alpha|  \le \HARDxDER - |\beta |}
\|\partial^{\alpha}_{\beta}h\|^2_{{N^{s,\gamma}_{\ell-|\beta|}}(\mathbb{T}^n\times\mathbb{R}^n)}.
\end{equation*}
We also define $B_C\subset \threed$ to be  the Euclidean ball of radius $C$ centered at the origin, then $L^2 (B_C)$ is the space $L^2$ on this ball and similarly for other spaces.  These are the spaces that we will use in our proofs of Theorem \ref{SG1mainGLOBAL} and Theorem \ref{mainGLOBAL}.

\subsection{Historical discussion} \label{sec:reviewNON}
 For early developments in the Boltzmann equation with long-range interactions, between 1952-1988,
 we mention  the work of
 Arkeryd, 
 Bobylev,
 Pao,
 Ukai,  
 and
 Wang Chang-Uhlenbeck-de Boer
 in 
 \cite{WCUh52,MR1128328,MR0636407,MR630119,MR679196,MR839310}.

Now Grad proposed \cite{MR0156656} in 1963 the angular cut-off which requires that $b(\cos \theta)$ be bounded.  Grad also pointed out that many cut-offs are possible.  
In particular,  the following less stringent $L^1(\sph)$ cut-off
 has become fashionable\footnote{
Note that the $L^1(\mathbb{S}^2)$ cut-off was already implicitly used in 1954 by Morgenstern \cite{MR0063956}.} 
$$
\int_{\sph}d\sigma~b(\ang{k, \sigma}) < \infty.
$$
These types of truncations have been widely accepted, and have now influenced several
decades of mathematical progress on the Boltzmann equation.   We refer the reader to a brief few breakthrough works of \cite{09conv,MR1379589,MR2043729,MR1313028,MR1307620,MR1014927,MR0156656,MR2259206,
MR2570766, MR2095473,MR0475532,MR1284432,MR882376,MR0363332,MR2000470,MR760333,MR2209761,MR2366140,MR2013332,MR2116276,MR1057534}; further references can be found in the review article \cite{MR1942465}.  Many of these works develop ideas and methods that are fundamental and important even without angular cut-off.  In particular the space-time estimates and general non-linear energy method developed by Guo \cite{MR1946444,MR2013332,MR2000470} and
 Strain-Guo \cite{MR2209761,MR2366140} is an important element in Section \ref{sec:deBEest} of our proof.

 These cut-off assumptions were originally believed to not change the essential nature of solutions to the equation. 
It has been argued by physicists, see \cite{MR1942465}, that the important properties of the Boltzmann equation are not particularly sensitive to the dependence of the collision kernel upon the deviation angle, $\theta$. There is on the other hand an extensive history of mathematical results which illustrates instead that solutions of the Boltzmann equation have a strong dependence upon the angular singularity.
 In particular, in the presence of these physical effects, the Boltzmann equation is well-known to experience regularizing effects.
 Results of this sort  go back to Lions \cite{MR1278244} and Desvillettes \cite{MR1324404}
 and have seen substantial developments \cite{MR1407542,
 MR1475459,MR1750040,MR1737547,MR1715411,MR2149928,MR2038147}.  
 Recently 
 Chen-Desvillettes-He \cite{MR2506070}
 and also
 Alexandre-Morimoto-Ukai-Xu-Yang \cite{MR2462585,MR2679369} have developed independent  machinery to study these general smoothing effects for kinetic equations.
 Contrast this with the case of an angular cut-off, where, as a result of works by Boudin-Desvillettes \cite{MR1798557} in 2000, and additional progress in \cite{MR2435186,MR2476678},
 we know that under
 the angular cut-off assumption small-data solutions
 can have the same
 Sobolev space regularity as the initial data.    These results illustrate that the Boltzmann equation with angular cut-off can be in some respects a very different model from the one without any angular cut-off.  
 Yet all of the inverse
 power-law potentials $p\in (2,\infty)$ dictate that
  the cross section  $B(v-v_*,\sigma)$ is non-integrable in the
 angular variable.

In 1998 Lions proved a functional inequality \cite{MR1649477} which bounds below the ``entropy dissipation'' by an isotropic Sobolev norm $H^\alpha_v$ up to lower order terms, for a certain range of $\alpha$.
Then in the work of Alexandre-Desvillettes-Villani-Wennberg
\cite{MR1765272} from 2000, this entropy dissipation smoothing estimate was obtained in the isotropic space $H^s_v(B_R)$ with the optimal exponent $s$.  This work further introduced elegant formulas, such as the cancellation lemma and  isotropic sub-elliptic coercivity estimates using the Fourier transform.
This was in several ways the starting point of the modern theory of grazing collisions.
Subsequent results of Desvillettes-Wennberg \cite{MR2038147} further demonstrated that solutions to the spatially homogeneous Boltzmann equation for regularized hard potentials enter the Schwartz space instantaneously.
And recently in
2009  Desvillettes-Mouhot \cite{MR2525118} proved the uniqueness of spatially homogeneous strong solutions  for the full range of angular singularities $s\in(0,1)$, and they have shown existence for moderate angular singularities $s\in(0,1/2)$. 
 We mention several further works which 
 developed and utilized the entropy production
 estimates for the
 collision operator
 (which grants a non-linear smoothing effect) and also the spatially homogeneous theories
 as in
 \cite{MR2149928,MR1650006,MR1484062,MR2398952,MR1750040}.
 Further references can be found in the surveys \cite{MR1942465,krmReview2009}.  In Section \ref{sec:ep} we discuss a new global anisotropic entropy production estimate with the stronger semi-norm from \eqref{normdef}.
 
 Lastly, for the most physically interesting and mathematically challenging case of the
 spatially inhomogeneous Boltzmann equation there are much fewer results.
  Here we have two results on local existence  
 \cite{MR1851391,MR2679369}, and a result \cite{MR1857879} on global existence of DiPerna-Lions renormalized weak solutions \cite{MR1014927}
with defect measure.   We also discuss some very recent work \cite{arXiv:1005.0447v2,arXiv:1007.0304v1,newNonCutAMUXY} related to our own in Section \ref{sec:comparison} after we state and explain our main estimates.

  Other  methods have been introduced to further study the Boltzmann collision operator without angular cut-off, using more involved methods from pseudodifferential operators and harmonic analysis.
In particular, some uncertainty principles in the framework of Fefferman
\cite{MR707957} 
were introduced in 2008 by  Alexandre-Morimoto-Ukai-Xu-Yang \cite{MR2462585}. 
These methods, as well as  \cite{MR2284553,MR2476686},
and the references therein, establish the hypoellipticity of the
Boltzmann operator.  They further develop methods for estimating the commutators  between the Boltzmann collision
operator and  some weighted pseudodifferential operators.  And they sharpen some of the isotropic coercivity and upper bound estimates for the Boltzmann collision operator.  

We also mention the  
linear isotropic coercivity estimates from \cite{MR2254617,MR2322149}.  In particular Mouhot-Strain \cite{MR2322149} proved the coercive lower bound for the linearized collision operator with the sharp weight, $\gamma+2s$, in the non-derivative part of our norm \eqref{normdef}.

Broadly speaking, the approaches outlined above use the Fourier transform to interpret the fractional differentiation effects in terms of isotropic Sobolev spaces.
Indeed for the Boltzmann collision operator, its essential behavior  has been widely conjectured  to be that of a fractional flat diffusion. 
 Precisely 
 $$
 F \mapsto \mathcal{Q}(g,F) \sim -(-\Delta_v)^s F + \mbox{l.o.t.}
 $$
 Here the function $g$ is thought of as a parameter.   Above ``l.o.t.'' indicates that the remaining terms will be lower order.
 The original mathematical intuition for this conjecture has been credited to Carlo Cercignani in 1969, now more than forty years ago (see for instance  Villani \cite[p.91]{MR1942465}).

By comparison, the Landau equation, derived in 1936, is maybe the closest analog to the Boltzmann collision operator for long-range interactions; however the Landau operator involves regular partial derivatives rather than
 fractional derivatives and for that reason may be somewhat more understandable at first.  Landau's equation is obtained in some sense as the limiting system when $p\to 2$ in the inverse power law potential, the Landau collision operator in three dimensions can be shown to satisfy \cite{MR1942465}:
$$
\mathcal{Q}_{\mathcal{L}}(F, F) = \sum_{i,j=1}^3 \bar{a}_{ij} \partial_{v_i}\partial_{v_j } F + 8\pi F^2,
\quad
\bar{a}_{ij}
=
\left( \frac{1}{|v|}\left[\delta_{ij} - \frac{v_i v_j}{|v|^2}  \right]  \right) * F.
$$
Let us briefly review a few results for the Landau equation.  For the spatially  homogeneous case with hard potentials (roughly, replace $1/|v|$ above with $|v|^{\gamma + 2}$ for $\gamma \ge 0$),  global existence of unique weak solutions and the instantaneous smoothing effect was shown for the first time by Desvillettes and Villani  \cite{MR1737547} for a large class of initial data in the year 2000.
Then Guo \cite{MR1946444} in 2002 established the existence of classical solutions for the spatially
dependent case with the physical Coulombian interactions ($p=2$) for smooth near Maxwellian initial data in a periodic box. 
Guo's solutions were recently shown to experience instantaneous regularization  in \cite{MR2506070}.  For further  results in these directions we refer to the references in \cite{MR2506070}.

Notice that in the Landau equation there is a metric of sorts in this case--in the $\bar{a}_{ij}$--which depends in an essential way on your unknown solution $F$.  Even in the simplest case when  your unknown is the steady state, $F = \mu(v)$, this $\bar{a}_{ij}$ weights more heavily angular derivatives \cite{MR1946444}.   Now the sharp anisotropic differentiation effects for the Dirichlet form of the linearized Landau collision operator have been studied, for example, by Guo \cite{MR1946444} in 2002, 
and Mouhot-Strain \cite{MR2322149} in 2007.   Anisotropic differentiation effects can also be observed in the Boltzmann theory using delicate calculations involving the Fourier transform; see, for example, \cite{MR0636407}, \cite{MR1763526}.    Other studies of the Fourier transform of the Boltzmann collision operator were given,  for example, in  
\cite{MR1763526,MR1128328,MR1765272,MR2052786,krmReview2009} and the references therein. 
But it has proved to be difficult to use the Fourier transform alone to prove sharp energy estimates.

For  this paper, the basic new understanding which enabled our progress was to identify that the fractional differentiation effects induced by the linearized Boltzmann collision operator are taking place on a paraboloid in $\R^{\last}$.  We prove that  the sharp linear behavior is in fact that of a fundamentally anisotropic fractional geometric Laplacian \eqref{sharpLinearEst}, the geometry being given by that of a ``lifted'' paraboloid in $\mathbb{R}^{n+1}$.  
 Our intuition for this behavior is derived from the original physics representations for the collision operator in terms of $\delta$-functions. We have also recently shown that the sharp diffusive behavior of the non-linear Boltzmann collision operator is also controlled by the diffusive semi-norm in \eqref{normdef}; see \cite{gsNonCutEst}.

Now using this new point of view during the course of the proof of our main Theorem \ref{mainGLOBAL}, we introduce a set of tools for the long-range interactions, which we believe have implications for a variety of future results both in the perturbative regime and perhaps beyond it \cite{gsNonCutEst}.  We do not use any of the major non cut-off techniques  described above, most of which are designed around the Fourier transform and estimates in terms of isotropic Sobolev spaces.  
Moreover, we do not study the Fourier transform of the collision operator at all.

From the standpoint of harmonic analysis, the estimates we make for the bilinear operator 
\eqref{gamma0} arising from our ansatz \eqref{maxLIN} fall well outside the scope of standard theorems.  The operator and its associated trilinear form may be expressed in terms of Fourier transforms as a trilinear paraproduct; such objects have been the subject of recent work of Muscalu, Pipher, Tao, and Thiele \cite{MR2320408} and are known to be very difficult to study in general.  Known results for such objects fail to apply in our case because of the loss of derivatives (meaning that two of the three functions $g$, $h$, and $f$ must belong to some Sobolev space with a positive degree of smoothness).  Moreover, routine modifications of known results (for example, composing with fractional integration to compensate for the loss of derivatives) also fail because of the presence of a fundamentally non-Euclidean geometry, namely, the geometry on the paraboloid.  This nontrivial geometry essentially renders any technique based on the Fourier transform difficult to use herein. Instead, we base our approach on the generalized Littlewood-Paley theory developed by Stein \cite{MR0252961}.  Rather than  directly using semigroup theory, however, we opt for a more geometric approach, as was taken, for example, by Klainerman and Rodnianski \cite{MR2221254}.  Since the underlying geometry we identify is explicit, we are able to make substantial simplifications over both of these earlier works by restricting attention to the particular case of interest.

\subsection{Possibilities for the future, and extensions}\label{possible}
We believe that our general methods and anisotropic point of view can be useful in making further progress on multiple fronts in the non cut-off theory.   Herein we list  some of those.

We are hopeful that the estimates we prove can play a part to resolve the existence question for the Vlasov-Maxwell-Boltzmann system without angular cut-off; notice at the moment the theory here is limited to the hard-sphere interactions \cite{MR2000470}.

For spatially homogeneous solutions, our results provide additional information for the high singularities $s\ge 1/2$ with singular kinetic factors \eqref{kernelPsing}, as in $p\in(2,3)$, in which otherwise there does not seem to be a global existence theory for strong solutions \cite{MR2476686,MR2525118}.
The methods and point of view in this paper may help to treat the high singularities  with large spatially homogeneous data.

Lastly, we think it would be important to work with the estimates herein and in 
\cite{MR1946444} to  justify rigorously the validity of Landau approximation near Maxwellian.  

\bigskip

In Section \ref{nrstat}, we linearize the Boltzmann equation 
\eqref{BoltzFULL} around the perturbation \eqref{maxLIN} and then explain the sharp space associated with the linearized collision operator.  We further define all the relevant notation and formulate and discuss the main velocity fractional derivative estimates in Section \ref{mainESTsec}.  Then we describe our resolution of a conjecture from \cite{MR2322149} in Section \ref{sec:conj}.  In Section \ref{sec:overview} we describe the several key new ideas which are used in our proof.  We will discuss some other recent related work \cite{arXiv:1005.0447v2,arXiv:1007.0304v1,newNonCutAMUXY} in Section \ref{sec:comparison}.  Then in Section \ref{sec:ep} we discuss the entropy production estimates, and finally in Section \ref{sec:outlineA}  we outline the rest of the article.

\section{Notation, reformulation, the main estimates, and our strategy}
\label{nrstat}

Throughout this paper, the notation $A \lesssim B$ will mean that a positive constant $C$ exists such that $A \leq C B$ holds uniformly over the range of parameters which are present in the inequality (and that the precise magnitude of the constant is irrelevant).  In particular, whenever either $A$ or $B$ involves a function space norm, it will be implicit that the constant is uniform over all elements of the relevant space unless explicitly stated otherwise.  The notation $B \gtrsim A$ is equivalent to $A \lesssim B$, and $A \approx B$ means that both $A \lesssim B$ and $B \lesssim A$.

The first thing to do in this section will be to reformulate the problem in terms of the equation \eqref{Boltz} for the perturbation \eqref{maxLIN}.

\subsection{Reformulation}
We linearize the Boltzmann equation \eqref{BoltzFULL} around the perturbation \eqref{maxLIN}.  This grants an equation for the perturbation $f(t,x,v)$ as
\begin{gather}
 \partial_t f + v\cdot \nabla_x f + L (f)
=
\Gamma (f,f),
\quad
f(0, x, v) = f_0(x,v),
\label{Boltz}
\end{gather}
where the {\it linearized Boltzmann operator} $L$ is given by
\begin{align*}
 L(g)
 \eqdef & 
- \mu^{-1/2}\mathcal{Q}(\mu ,\sqrt{\mu} g)- \mu^{-1/2}\mathcal{Q}(\sqrt{\mu} g,\mu) \\
  = &
  \int_{\threed}dv_{*}
  \int_{\sph} d\sigma~
  B(v-v_*,\sigma) \, 
   \left[g_{*} M + g M_{*}- g^{\prime}_{*} M^{\prime} - g^{\prime} M^{\prime}_{*}  \right] M_{*},
\end{align*}
and the bilinear operator $\Gamma$ is given by
\begin{gather}
\Gamma (g,h)
\eqdef
 \mu^{-1/2}\mathcal{Q}(\sqrt{\mu} g,\sqrt{\mu} h)
 =
 \int_{\threed} dv_* \int_{\sph} d \sigma  ~ B 
 ~M_* (g_*' h' - g_* h). 
\label{gamma0}
\end{gather}
In both definitions, we take
$$
M(v) \eqdef \sqrt{\mu(v)} = (2 \pi)^{-n/4} e^{- |v|^2/4}.
$$
When convenient, we will without loss of generality abuse notation and neglect the constant $(2 \pi)^{-n/4}$ in the definition of $M$.  Finally, we note that 
\begin{equation}
L(g) \eqdef - \Gamma(M,g) - \Gamma(g, M).
\label{LinGam} 
\end{equation}
This reformulation shows that it is fundamentally important to obtain favorable estimates for the bilinear operator $\Gamma$.

We split the main term of the linearized Boltzmann collision operator whilst preserving the cancellations as follows:
   \begin{gather*}
\Gamma(M,g)  =
 \int_{\threed}dv_{*}
  \int_{\sph} d\sigma~
  B (v-v_*,\sigma) \, 
   \left(g^{\prime}- g \right) M^{\prime}_{*}M_{*}
   -
 \tilde{\nu}(v)  ~ g(v),
  \end{gather*}
where
  $$
 \tilde{\nu}(v) 
   =
  \int_{\threed}dv_{*}
  \int_{\sph} d\sigma~
  B(v-v_*,\sigma) \, 
 ( M_{*} - M^{\prime}_{*} ) M_{*}.
 $$ 
 The first piece above contains a crucial Hilbert space structure.  
 This can be seen from the pre-post collisional change of variables  \cite{MR1942465} $(v, v_*,\sigma) \to  (v', v_*',k)$, as
   \begin{multline*}
 - \int_{\mathbb{R}^n}dv
  \int_{\mathbb{R}^n}dv_{*}
  \int_{\sph} d\sigma~
  B(v-v_*,\sigma) \, 
 (g^{\prime}-g) h M^{\prime}_{*}  M_{*}
 \\
     =
     -
     \frac{1}{2}
  \int_{\mathbb{R}^n}dv
  \int_{\mathbb{R}^n}dv_{*}
  \int_{\sph} d\sigma~
  B \,  
 (g^{\prime}-g) h M^{\prime}_{*}  M_{*}
 \\
 -
      \frac{1}{2}
  \int_{\mathbb{R}^n}dv
  \int_{\mathbb{R}^n}dv_{*}
  \int_{\sph} d\sigma~
  B \, 
 (g-g^{\prime}) h^{\prime} M_{*}  M^{\prime}_{*}
  \\
     =
     \frac{1}{2}
  \int_{\mathbb{R}^n}dv
  \int_{\mathbb{R}^n}dv_{*}
  \int_{\sph} d\sigma~
  B \, 
 (g^{\prime}-g) (h^{\prime}-h) M^{\prime}_{*}  M_{*}.
  \end{multline*}
For the weight, we will use Pao's splitting as
$$
 \tilde{\nu}(v)  =  \nu(v) + \nu_{\kPiece}(v),
$$
where under only   \eqref{kernelQ}, \eqref{kernelP}, and \eqref{kernelPsing} the following asymptotics are known:
\begin{equation}
 \nu (v)\approx
\ang{v}^{\gamma+2s},
\quad
\text{and}
\quad
\left|  \nu_{\kPiece}(v) \right|
\lesssim
\ang{v}^{\gamma}.
\notag
\end{equation}
These  estimates
were established by
 Pao in \cite[p.568 eq. (65), (66)]{MR0636407} by reducing to the known asymptotic behavior of confluent hypergeometric functions.  
They can also be established with the standard contemporary machinery.

We further decompose  $L= \nPiece  + \kPiece $.  Here $\nPiece$ is the ``norm part'' and $\kPiece$ will be seen as the ``compact part.''
The norm part is then written as
\begin{multline}
   \label{normpiece}
 \nPiece g
   \eqdef
   - \Gamma(M,g) - \nu_{\kPiece}(v) g
   \\
   =
  -\int_{\threed}dv_{*}
  \int_{\sph} d\sigma~
  B (v-v_*,\sigma) \, 
 (g^{\prime}-g) M^{\prime}_{*}  M_{*}
   + \nu(v) g(v).
\end{multline}
Then, with the previous calculations, this norm piece satisfies the following identity: 
\begin{gather*}
  \ang{\nPiece g,g} =  \frac{1}{2} \int_{\threed} dv \int_{\threed} dv_* \int_{\sph} d \sigma B  (g'-g)^2 M_*' M_* 
  + 
  \int_{\threed} dv ~ \nu(v) ~ |g(v)|^2. 
\end{gather*}
As a result, in the following we will use the anisotropic fractional semi-norm
 \begin{equation}
| g |_{B_\ell}^2 \eqdef 
\frac{1}{2} \int_{\mathbb{R}^n} dv~ w^{\ell}(v)\int_{\mathbb{R}^n} dv_*~ \int_{\sph} d \sigma~ B~ (g'-g)^2 M_*' M_*,
\quad 
\ell \in \mathbb{R}. 
\label{normexpr}
\end{equation}
We also sometimes write $| g |_{B_0}^2 = | g |_{B}^2$ to ease the notation.
For the second part of $\ang{\nPiece g,g}$ we recall the norm 
$
\nsm f\nsm_{L^2_{\gamma+2s}}
$
defined below equation \eqref{normdef}. 
These two quantities will define our designer norm, which is sharp for the linearized operator.
We also record here the definition of the ``compact piece'' $\kPiece$:
 \begin{equation}
  \kPiece g \eqdef  \nu_{\kPiece}(v) g 
- \Gamma(g, M)
=
  \nu_{\kPiece}(v) g 
-
 \int_{\threed} dv_* \int_{\sph} d \sigma B M_* (g_*' M' - g_* M).
 \label{compactpiece}
\end{equation}
This is our main splitting of the linearized operator.

\subsection{Main Estimates for the hard and soft potentials}\label{mainESTsec}
In this sub-section we will state most of the crucial long-range estimates to be used in our main results.
In the next sub-section, we discuss how these estimates in particular resolve a conjecture from Mouhot-Strain \cite{MR2322149}.

We will prove all of our estimates for functions in the Schwartz space, $\mathcal{S}(\threed)$, which is the well-known space of real valued $C^{\infty}(\threed)$ functions all of whose derivatives decay at infinity faster than the reciprocal of any polynomial.   
Note that the Schwartz functions are dense in the anisotropic spaces $\spacen$, $\hardNspace$, $\nspace$, etc, and the proof of this fact is easily reduced to the analogous one for Euclidean Sobolev spaces by means of the partition of unity as constructed, for example, in Section \ref{sec:funcN}.
Moreover, in all of our estimates, none of the constants that come up will depend on the 
regularity of the functions that we are estimating.  Thus using routine density arguments, our estimates will apply to any function in $\spacen$ or whatever the appropriate function space happens to be for a particular estimate.

All of the estimates below will hold for both the hard \eqref{kernelP} and the soft \eqref{kernelPsing} potentials unless otherwise stated.  Our essential trilinear estimate is the following:

\begin{theorem}
\label{TriLinEst}
(Main trilinear estimate)
We have the basic estimate
\begin{equation}
 |\ang{\Gamma(g,h),f}| \lesssim  \nsm g\nsm_{L^2} \nsm h\nsm_{\spacen} \nsm f\nsm_{\spacen}. 
\notag
\end{equation}
This holds in the case of the hard potentials from \eqref{kernelP}.
\end{theorem}

Theorem \ref{TriLinEst} already contains the essential idea for the rest of our trilinear estimates.   The following two trilinear estimates are the main ones we use below:

\begin{lemma}
\label{NonLinEstLOW}
(Trilinear estimate for the hard potentials)  
Suppose that $|\alpha | +|\beta|\le \HARDxDER$ and $|\beta| \le \HARDvDER$ with
$\HARDxDER \geq \ksob$ and $0\le \HARDvDER \le \HARDxDER$.
For any $\ell \ge 0$ we have 
$$
 \left| \left(  w^{2\ell-2|\beta|}\partial^\alpha_\beta \Gamma(g,h), \partial^\alpha_\beta f\right) \right| 
 \lesssim 
  \| g\|_{\HARDspaceh} \| h\|_{\hardNspace} \| \partial^{\alpha}_{\beta}f\|_{N^{s,\gamma}_{\ell - |\beta|}}. 
$$
This estimate will hold for the hard potentials \eqref{kernelP}.
\end{lemma}

Note that what we actually prove below is that Theorem \ref{TriLinEst} and Lemma \ref{NonLinEstLOW} hold more generally whenever 
$\gamma + 2s > -\frac{\dim}{2}$ and $\gamma > -\dim$.  The most general trilinear estimate of this type that we prove is precisely given in \eqref{generalUPPER}.

\begin{lemma}
\label{NonLinEstHIGH}
(Trilinear estimate for the soft potentials)  
For any $|\alpha | + |\beta |\le K$, with $K \geq 2 \ksob$ and $\ell \ge 0$,  we have
the following  estimate for \eqref{gamma0}:
\begin{equation*}
 \left| \left(w^{2\ell-2|\beta|}\partial^{\alpha}_{\beta}\Gamma(g, h),\partial^{\alpha}_{\beta}f\right) \right| 
 \lesssim 
  \| g\|_{\spaceh} \| h\|_{\nspace} 
  \| \partial^{\alpha}_{\beta}f\|_{N^{s,\gamma}_{\ell - |\beta|}}.
\end{equation*}
This estimate will hold for the soft potentials \eqref{kernelPsing}.
\end{lemma}

The estimates in Lemmas \ref{NonLinEstLOW} and \ref{NonLinEstHIGH} improve substantially over previously known-estimates of this sort,  such as, for example, \cite{MR1946444,MR2013332,MR2000470,MR2095473} which hold in the cut-off regime.   This is because we are able to have only one term in the upper bounds, rather than two or more additional lower order terms.  The bounds above are consistent with estimates for a Laplacian-type smoothing operator.  Notice furthermore that this estimate
  does not require a velocity weight which goes to infinity at infinity.
This feature is made possible by our anisotropic Littlewood-Paley adapted to the paraboloid, 
which characterizes exactly the geometric fractional differentiation effects that are induced by the linearized Boltzmann collision operator. 
 The next important inequality that we establish is for the linear operator:
 
\begin{lemma}
\label{sharpLINEAR}
Consider the linearized Boltzmann operator $L = \nPiece + \kPiece$ where $\nPiece$ is defined in \eqref{normpiece} and $\kPiece$ is defined in \eqref{compactpiece}.
We have the uniform inequalities
\begin{align}
\left| \ang{w^{2\ell} \nPiece g, g} \right| & \lesssim |g|_{\spaceELLn}^2, \label{normupper} \\
\left| \ang{w^{2\ell} \kPiece g, g} \right| & \le \eta |w^\ell g|_{L^2_{\gamma + 2s}}^2 + C_\eta 
|g|_{L^2(B_{C_\eta})}^2, \label{compactupper}
\end{align}
where $\ell \in \mathbb{R}$, $\eta>0$ is any small number, and $C_\eta>0$.  
\end{lemma}

In these estimates there are several things to observe.  First of all there are no derivatives in the ``compact estimate'' from \eqref{compactupper}, which should be contrasted with the corresponding estimate in the Landau case \cite[Lemma 5]{MR1946444} in which the upper bound  requires the inclusion of derivatives. Further  \eqref{normupper} is a  simple consequence of the main estimate \eqref{nlineq}.  
This estimate tells us that the ``norm'' piece of the linear term, given by $\ang{\nPiece g,g}$, is bounded above by a uniform constant times $\nsm g\nsm_{N^{s,\gamma}}^2$.  This means that the coercive inequality  in the next lemma 
is essentially sharp.

\begin{lemma}
\label{estNORM3}
(Main coercive inequality)
For the sharp space defined in \eqref{normdef} with \eqref{normpiece}, \eqref{kernelQ}, \eqref{kernelP} and \eqref{kernelPsing} we have
  the uniform coercive lower bound estimate:
\begin{equation*}
\ang{ w^{2\ell} \nPiece g,g}
 \gtrsim
 \nsm g \nsm_{\spaceELLn}^2
 -
  C \nsm g \nsm_{L^2(B_C)}^2,
\quad
\exists C \ge 0.
\end{equation*}
This holds   for any $\ell \in\mathbb{R}$; 
if $\ell = 0$ we may take $C = 0$.
\end{lemma}

The coercive inequality in Lemma \ref{estNORM3} and \eqref{normupper} taken together demonstrate that the ``norm piece'' \eqref{normpiece} is actually comparable to our designer norm $\spacen$, i.e.,
\[ \ang{ \nPiece g, g} \approx |g|_{\spacen}^2. \]
The upper bound \eqref{normupper} is important because it demonstrates that the anisotropic space $\spacen$ naturally arises in this near Maxwellian problem.

Lastly, we have two coercive interpolation inequalities for the linearized operator:

\begin{lemma}
\label{DerCoerIneq}
(Coercive interpolation inequalities)
For any multi-indices $\alpha,\beta$, any $\ell \ge 0$, and any small $\eta >0$ there is a positive constant $C_\eta$,  such that we have the following coercive lower bound for the linearized collision operator
\begin{equation}
\label{coerc1ineq}
\langle w^{2\ell-2|\beta|} \partial^{\alpha}_{\beta} Lg, \partial^{\alpha}_{\beta} g \rangle
\gtrsim
\nsm \partial^{\alpha}_{\beta}g\nsm_{N^{s,\gamma}_{\ell - |\beta|}}^2     
-
\eta
\sum_{|\beta_1| \le |\beta|}\nsm \partial^{\alpha}_{\beta_1} g\nsm_{N^{s,\gamma}_{\ell - |\beta_1|}}^2  
-
C_\eta     \nsm \partial^\alpha g\nsm_{L^2(B_{C_\eta})}^2.
\end{equation}
Furthermore, when no derivatives are present, for some $C>0$ we have 
\begin{equation}
\label{coerc2ineq}
\langle w^{2\ell}Lg,  g \rangle
\gtrsim
\nsm g\nsm_{\spaceELLn}^2     
-
C     \nsm  g\nsm_{L^2(B_C)}^2.
\end{equation}
\end{lemma}

This concludes our statements of the main estimates that will be used in  Section \ref{sec:deBEest} to establish our main Theorems \ref{SG1mainGLOBAL} and \ref{mainGLOBAL}.   Next, we deduce some consequences for the spectral properties of the linearized Boltzmann collision operator.

\subsection{Conjecture from Mouhot-Strain \cite{MR2322149}}\label{sec:conj}
We will now discuss  sharp constructive coercivity estimates of the linearized collision operator, $L$, away from its null space.  
More generally, from the $H$-theorem  $L$ is non-negative and for every fixed $(t,x)$ the null
space of $L$ is given by the $(n+2)$-dimensional space 
\begin{equation}
\nullSpace \eqdef {\rm span}\left\{
\sqrt{\mu},
v_1 \sqrt{\mu}, \ldots, v_n \sqrt{\mu}, 
|v|^2 \sqrt{\mu}\right\}.
 \label{null}
\end{equation}
We define the orthogonal projection from $L^2(\mathbb{R}^n)$ onto the null space $\nullSpace$ by ${\bf P}$. 
Further expand ${\bf P} g$ as a linear combination of the basis in \eqref{null}: 
\begin{equation}
{\bf P} g
\eqdef
 \left\{a^g(t,x)+\sum_{j=1}^{\dim} b_j^g(t,x)v_j+c^g(t,x)|v|^2\right\}\sqrt{\mu}.
\label{hydro}
\end{equation}
Recall for \eqref{LinGam} that $L\ge 0$ and $Lg = 0$ if and only if $g = {\bf P} g$; see e.g. \cite{MR1307620,MR1942465}.  
By combining the constructive upper bound estimates in Lemma \ref{sharpLINEAR} with the constructive coercivity estimate from Theorem \ref{lowerN} in Section \ref{sec:deBEest}, we obtain
\begin{equation}
\frac{1}{C} | \{{\bf I- P} \} g |_{N^{s,\gamma}}^2 \le \langle Lg, g \rangle \le C | \{{\bf I- P} \} g |_{N^{s,\gamma}}^2,
\label{sharpLinearEst}
\end{equation}
with a constant $C>0$ that can be tracked from the proof.
Thus a spectral gap exists if and only 
if $\gamma + 2 s \ge 0$, as conjectured in  \cite{MR2322149}.
 The main tools in our proof of these statements are the new constructive estimates from Lemma \ref{sharpLINEAR} and Lemma \ref{estNORM3} combined with the constructive but non-sharp coercive lower bound from \cite{MR2254617}  for the non-derivative part of the norm.
This may be of independent interest. 
 
Notice that for the linearized Landau collision operator, this statement already had been shown several years earlier in \cite{MR2254617,MR2322149, MR1946444} for any $\gamma \ge -3$ in dimension $\dim =3$ (and more generally).  For Landau, there is a spectral gap if and only if $\gamma +2\ge 0$; the Landau operator can be thought of as the limit case when $s=1$ and  regular (rather than fractional) derivatives are present.

\subsection{Overview of our proof}\label{sec:overview}
In this section we explain the several new ideas which contributed to the proofs in this work.
It has been known to the experts for some time that the sum total of the inequalities in Section \ref{mainESTsec} would be sufficient for global existence \cite{MR2000470,MR1946444}, 
although crucially the spaces in which these inequalities should be proved was unknown.
We have the inequality
\begin{align}
\ang{L g,g}  & \gtrsim |g|_{N^{s,\gamma}}^{2}  - \mbox{lower order terms}. \label{lowerbound}
\end{align}
This follows from Lemmas \ref{sharpLINEAR} and \ref{estNORM3}.
This coercive lower bound inequality is fundamental to global existence.
Since the the operators $\Gamma$ and $L$ are intimately connected, among other consequences, this means that if both of \eqref{lowerbound} and, for example, Theorem \ref{TriLinEst} are simultaneously true, then the Hilbert space $N^{s,\gamma}$ satisfying these inequalities is unique.   From this point of view the first major difficulty which we had to overcome was the identification of the appropriate Hilbert space.  

\subsection*{Identification of the space $\spacen$}

It turned out that the candidate Hilbert space $\spacen$ is a weighted, anisotropic fractional Sobolev space \eqref{normdef}  and \eqref{normpiece} which corresponds to fractional differentiation on the paraboloid in $\R^{\last}$.  Sharp comparisons of $\nsm \cdot \nsm_{\spacen}$ to the weighed isotropic Sobolev spaces are established by the inequalities
\begin{equation}
\nsm g \nsm_{L^2_{\gamma+2s}(\threed)}^2
+
\nsm g \nsm_{H^s_{\gamma}(\threed)}^2
\lesssim 
\nsm g\nsm_{\spacen}^2 
\lesssim \nsm g \nsm_{H^s_{\gamma+2s}(\threed)}^2.
\label{isocomp}
\end{equation}
Here 
$
H^s_{\ell}(\threed)
\eqdef
\{ g \in L^2_\ell (\threed) : \nsm f \nsm_{H^s_{\ell}(\threed)}^2 
\eqdef \int_{\threed} dv \ang{v}^\ell \left| (I - \Delta_v)^{s/2} g(v) \right|^2 <\infty
\}
$
is the standard isotropic fractional Sobolev space.
These inequalities are established using the partition of unity as constructed in Section \ref{sec:funcN}.  The anisotropy of $\spacen$ manifests itself in the fact that the spaces $H^s_{\gamma}$ and $H^s_{\gamma+2s}$ appearing in the upper and lower bounds are sharp but not equal to one another.   This illustrates that the standard isotropic Littlewood-Paley decompositions are insufficient to prove sharp anisotropic estimates.  To estimate this space, we find it convenient to use a geometric Littlewood-Paley-type decomposition, inspired by the work of Stein \cite{MR0252961}.  We do not, however, take a semigroup approach to the actual construction of our Littlewood-Paley projections as Stein did.  Instead, we use the embedding of the paraboloid in $\R^{\last}$ to our advantage.  If $d \mu$ is the Radon measure on $\R^{\last}$ corresponding to surface measure on the paraboloid, our approach is to take a renormalized version of the {$(\dim+1)$-dimensional}, {\it Euclidean} Littlewood-Paley decomposition of the {\it measure} $g d \mu$ as our anisotropic, $\dim$-dimensional, Littlewood-Paley-type decomposition for the function $g$.  Among other benefits, this approach automatically allows for a natural extension of the Littlewood-Paley projections $P_j g$ and $Q_j g$ (from Section \ref{sec:aniLP}) as smooth functions defined on $\R^{\last}$ in a neighborhood of the paraboloid.  This allows us to avoid a direct discussion of the induced metric on $\threed$ by phrasing our results in terms of the projections $P_j g$, $Q_j g$, and various Euclidean derivatives of these functions in $\R^{\last}$ instead of $\threed$.

\subsection*{The upper bound inequality}

The proof of the main trilinear estimates in Theorem \ref{TriLinEst}, and Lemmas \ref{NonLinEstLOW} and \ref{NonLinEstHIGH} are based on a dyadic decomposition of the singularity of the collision kernel $B$ in \eqref{kernelQ}, \eqref{kernelP} and \eqref{kernelPsing} as well as a Littlewood-Paley-type decomposition of the functions $h$ and $f$.  The end result is that one is led to consider a triple sum of the form
\[ 
\sum_{k=-\infty}^\infty \sum_{j' = 0}^\infty \sum_{j=0}^\infty \left| \ang{\Gamma_k(g,h_{j'}), f_{j}} \right|. 
\]
Here $\Gamma_k$ is the bilinear operator \eqref{gamma0} summed over a specific anisotropic dyadic decomposition of the singularity (see Section \ref{sec:dyadic}), and $h_{j'}$, $f_{j}$ are the functions $h$, $f$ expanded in terms of the anisotropic Littlewood-Paley decomposition described just above (and in Section \ref{sec:aniLP}).
Control over the sum rests on two important observations.
First, when considering terms for which $2^{-k}$ is large relative to $2^{-j'}$ and $2^{-j}$, a favorable estimate holds simply because the support of $B_k(v-v_*,\sigma)$ is compact and bounded away from the singularity at $\theta = 0$.  
Second, when either $2^{-j'}$ or $2^{-j}$ is large relative to $2^{-k}$, i.e., near the singularity, an improvement may be made by exploiting the inherent cancellation structure of $\Gamma_k$.  
The cost which must be paid in order to use this cancellation is that derivatives must fall on either $h_{j'}$ or $f_{j}$.  In this case, with the dual formulation (described next)  it is always possible to arrange for the derivatives to be placed on the function of our choice.  Placing the derivatives on the function of largest scale (that is, the function whose index is least) gives some extra decay that allows one to sum all the terms by comparison to a geometric series.   These types of decompositions, of course, have a long history and generally follow the long-established techniques of harmonic analysis. 
The crucial new feature distinguishing our approach is that we do not measure cancellations in the standard isotropic way; cancellations are measured instead by using the anisotropic metric on the ``lifted'' paraboloid in $\R^\last$.

It should be noted that our analysis allows us to essentially ignore the dependence of $\Gamma(g,h)$ on the function $g$; this is a great advantage, as it means that one may think of the trilinear form $\ang{\Gamma(g,h),f}$ as a family of bilinear forms in $h$ and $f$ parametrized by the function $g$.  This observation is essential, since the fully trilinear form falls well outside the scope of existing tools in harmonic analysis.

\subsection*{The dual formulation}
A key point of significant technical importance in the proof of the upper bound inequality is that we must be able to make estimates for $\ang{\Gamma(g,h),f}$ which exploit the intrinsic cancellations at the cost of placing derivatives on any one of the two functions $h$ or $f$ that we choose.  
As the presence of fractional derivatives rules out a traditional integration-by-parts, 
it is necessary to find two different, yet analogous, representations of the trilinear form $\ang{\Gamma(g,h),f}$ which clearly relate cancellation to smoothness of $h$ and $f$, respectively.  It turns out that placing derivatives on $f$ is fairly straightforward to do using existing representations for the bilinear operator $\Gamma$.  In particular, one may apply the pre-post change of variables 
to obtain the representation
\begin{align*}
\ang{\Gamma(g,h),f} = \int_{\threed} dv \int_{\threed} dv_* \int_{\sph} d \sigma ~B(v-v_*,\sigma) ~ g_* h ~ (M_*' f' - M_* f).
\end{align*}
Clearly, for each fixed $g$, there is an operator $T_g$ such that $\ang{\Gamma(g,h),f} = \ang{T_g f,h}$, and moreover, the formula above can be used to write down an explicit formula for $T_g$; see \eqref{3dualZ}. 
To place derivatives on $h$, on the other hand, it is necessary to find a new representation which involves only differences of $h'$ and $h$, i.e., no differences of $g$ or $f$.  To that end, 
we compute what we call the ``dual formulation,'' which amounts to writing down a formula for $T_g^*$.  These computations may be found in  Appendix \ref{secAPP:HSr}; the end result is that
$$
\ang{\Gamma(g,h),f} = \int_{\threed} dv   \int_{\threed} dv_* \int_{\sph} d \sigma   ~ B ~ g_* f' ~
 \left(M_*' h \vphantom{\int}   -   M_* h'  \right) + \opGstar.
$$
Here $\opGstar$ is a term which does not differentiate,  defined in \eqref{opGlabel}.
An interesting consequence of this formula is that the gain term ${\mathcal Q}^+$ is unchanged and only the loss term ${\mathcal Q}^-$ differs in these two formulas.  These two formulas also demonstrate the essentially straightforward dependence on $g$ which we use to apply traditionally bilinear methods to the trilinear form.

\subsection*{The coercive inequality}
The key to proving \eqref{lowerbound}, on the other hand, is to show the equivalence between 
\eqref{normdef} and the inner product $\ang{\nPiece f,f}$ from
\eqref{normpiece}.
We prove equivalent estimates in terms of the Littlewood-Paley projections.  This consists of two parts.  The first is rewriting \eqref{normexpr} with a Carleman representation as
\begin{equation}
| f |_B^2
=
\int_{\threed} dv \int_{\threed} dv' ~ K(v,v') ~ (f'-f)^2, \label{semiCARLEMAN}
\end{equation}
for an appropriate function $K(v,v')$, see \eqref{kernel}.   A simple pointwise estimation of this function $K$ demonstrates that
\[ K(v,v') \gtrsim (\ang{v} \ang{v'})^{\frac{\gamma+2s+1}{2}} (d(v,v'))^{-n - 2s}, \]
for a large set of pairs $(v,v')$, the exact description of which is slightly complicated.  The second part is to demonstrate that the set of pairs for which this inequality holds is large enough to conclude an integral version of this inequality, namely,
\[ 
\ang{\nPiece f,f} \gtrsim \int _{\threed} dv \int_{\threed} dv' ~ (\ang{v} \ang{v'})^{\frac{\gamma+2s+1}{2}} 
\frac{(f' - f)^2}{d(v,v')^{n+2s}}  {\mathbf 1}_{d(v,v') \leq 1}. 
\]

\subsection{Other recent results}  
\label{sec:comparison}
This present paper is a combination, simplification, and extension of two separate preprints originally posted on the arXiv as 
\cite{gsNonCut1,gsNonCut2}, the first covering the hard potential case \eqref{kernelP} and the second covering the case of soft potentials \eqref{kernelPsing}.   After \cite{gsNonCut1} had been posted, a related work \cite{newNonCutAMUXY} by Alexandre, Morimoto, Ukai, Xu, and Yang appeared.  This preprint \cite{newNonCutAMUXY} announces a proof, using different methods, of global existence and smoothness for  perturbations of the Maxwellian equilibrium states \eqref{maxLIN} in $\R^3_x$ for the  Maxwell molecules collision kernel (meaning that the kinetic factor in \eqref{kernelP} is constant) and moderate angular singularities (meaning that $0< s< 1/2$ in \eqref{kernelQ}); these assumptions  apply to the inverse power intermolecular potentials when $p=5$.
At the time of this writing there are now a series of papers from our collaboration and theirs: 
\cite{gsNonCut1}, \cite{newNonCutAMUXY}, \cite{gsNonCut2}, \cite{gsNonCutA}, \cite{arXiv:1005.0447v2}, \cite{arXiv:1007.0304v1}.

The papers \cite{arXiv:1005.0447v2,arXiv:1007.0304v1} are the first two in a series of
four which will be devoted to the whole-space $\R^3_x$ version of the problems
solved in our work herein 
on the torus $\domain_x$.  That is, they intend to prove the global existence, uniqueness, positivity, and convergence rates to equilibrium of classical solutions which are initially perturbations of the Maxwellian
solutions.   We remark that taking $x\in \mathbb{T}^n_x$ in our main Theorems \ref{SG1mainGLOBAL} and \ref{mainGLOBAL} is not a restriction.  In particular the case when $x\in \mathbb{R}^n_x$ replaces $x\in \mathbb{T}^n_x$ will follow directly from our new velocity estimates in the previous sections when combined with other known whole space cut-off methods.  
This would only require modifying Section \ref{sec:deBEest} of this paper, using instead the cut-off energy methods in the whole space such as, for example, \cite{MR2095473,MR2259206,MR1057534}.  Also we can prove the optimal convergence rates in the whole space for both the hard and the soft potentials as in \cite{sNonCutOp}; see additionally \cite{MR2209761,2010arXiv1006.3605D,arXiv:0912.1742}.   Indeed, we consider that the key contributions in this paper and also in \cite{arXiv:1005.0447v2,arXiv:1007.0304v1} are to introduce new methods on the velocity variables to study the anisotropic fractional diffusive nature of the linearized non cut-off Boltzmann collision operator.

At the time of this writing, the completed works in the series \cite{arXiv:1005.0447v2,arXiv:1007.0304v1} are
restricted to showing various estimates for the linearized collision operator \cite{arXiv:1005.0447v2} and
 global existence \cite{arXiv:1007.0304v1}  for a range of soft potentials.  
 As such, it is currently difficult to make a full comparison of the two different methods; however, there are already
several important points on which their approach differs from our own.  Additionally, since they work in $\R^3_x$,  it is hard to compare the two results outside of three dimensions.

A key difference between the analysis in this paper and that of \cite{arXiv:1005.0447v2,arXiv:1007.0304v1} is the representation for the sharp norm which is used.  Compare our geometric fractional norm \eqref{normdef} with the norm from \cite{arXiv:1005.0447v2,arXiv:1007.0304v1}:
$$
||| g |||^2 \eqdef 
\int_{\R^3} dv \int_{\R^3} dv_* \int_{\mathbb{S}^{2}} d\sigma ~
B(v-v_*,\sigma) ~ \left\{\mu_* ~ (g' - g)^2
+ g^2\left( \sqrt{\mu'} - \sqrt{\mu}\right)^2 \right\}.
$$
Notice that it is a consequence of our estimates in Section \ref{mainESTsec} and \eqref{sharpLinearEst}, that this norm is in fact equivalent to \eqref{normdef} (which was introduced in \cite{gsNonCut1}).  Moreover, by comparison to the linearized Boltzmann collision operator \eqref{LinGam}; this norm  $||| \cdot |||$ is quite simple, it is also basically the same as our representation for $\ang{\nPiece g, g}$ in e.g. \eqref{normpiece}.   The norms $||| \cdot |||$ and  \eqref{normpiece} both however do not provide much new information over the linearized collision operator \eqref{LinGam}  regarding the sharp geometric fractional diffusive nature of the of the Boltzmann collision operator.

One of the main motivating factors in our study of this problem was in fact to gain a detailed understanding of the Boltzmann equation as as a geometric fractional diffusive operator.  Our norm \eqref{normdef} provides this sharp information; it has lead to our proof of a conjecture from \cite{MR2322149}, see e.g. \eqref{sharpLinearEst}.  The geometric norm \eqref{normdef} has also lead to our recent sharp characterization of the geometric fractional diffusive nature of the fully non-linear Boltzmann collision operator, in \cite{gsNonCutEst}, and also its Entropy production estimates; see \cite{gsNonCutEst} and Section \ref{sec:ep}.  We furthermore believe that the new geometric information that is directly contained in our norm \eqref{normdef} will be quite useful to a wide variety of future work in the Boltzmann theory.  The introduction of this norm \eqref{normdef} and the establishment of sharp estimates like Theorem \ref{TriLinEst} may be of independent interest.

Another topic of comparison is the estimates.  A typical estimate of theirs \cite[Proposition 2.5]{arXiv:1005.0447v2} (in our notation) is of the form
\begin{equation}
\begin{split}
\Big|\ang{\Gamma(g,f),h}\Big| & 
\lesssim \Big\{\|f\|_{L^2_{s+\gamma/2}}|||g|||
+ \|g\|_{L^2_{s+\gamma/2}}|||f|||
\\
& \quad { + \min \Big ( \|f\|_{L^2}\|g\|_{L^2_{s+\gamma/2}} \,,\,\|f\|_{L^2_{s+\gamma/2}}\|g\|_{L^2}
\Big)}\Big\}|||h||| \, .
\end{split} \label{AMUXY}
\end{equation}
Here they allow $0<s<1$ and $\gamma>-3/2$.
A nice feature of this estimate is that it can be proved using the structure of their norm $||| \cdot |||$ in roughly eight pages, while the corresponding estimates used to establish Theorem \ref{TriLinEst} are scattered throughout roughly twenty pages of this present work. The estimate \eqref{AMUXY}, however, has the disadvantage
that it ``loses weights'' in the case of hard potentials (in the sense that there are terms on the right-hand side for which there is a growing velocity weight in every function space). 
Indeed, in order to transform this into an estimate which can be applied to the current machinery for proving global in time stability requires additional delicate and extensive commutator estimates to even out the weights.  By comparison,  our main estimate in Theorem \ref{TriLinEst} for the hard potentials \eqref{kernelP} only requires putting velocity weights in two of the three function spaces; moreover, it does not require differentiation of the 
parameter function $g$, nor does it require the addition of any lower-order terms on the right-hand side.
This is consistent with the heuristic understanding of the Boltzmann collision operator as a fractional geometric Laplacian in the energy space. 
Thus our estimates like Theorem \ref{TriLinEst} and others are substantially sharper than their analogues in \cite{newNonCutAMUXY,arXiv:1005.0447v2,arXiv:1007.0304v1}.

Regarding the main theorems, their global existence theorem for classical solutions and soft potentials, \cite[Theorem 1.3]{arXiv:1007.0304v1}, requires $N\ge 6$ derivatives and $\ell \ge N$ velocity weights; it also restricts to kinetic factors \eqref{kernelPsing} satisfying $\gamma > \max\left\{-3, -\frac{3}{2} -2s\right\}$ with $\gamma + 2s \le 0$ in three dimensions.  (In the upper bound estimate of, for example 
$
 \left| \left(  w^{2\ell-2|\beta|}\partial^\alpha_\beta \Gamma(g,h), \partial^\alpha_\beta f\right) \right| 
$
their methods use ${\ell-|\beta|}\ge 0$.)  By comparison, our main Theorems \ref{SG1mainGLOBAL} and \ref{mainGLOBAL} for global existence and uniqueness of classical solutions for the hard and soft potentials allows all $\gamma > -\dim$ in $\dim$ dimensions, any weight $\ell \ge 0$, and uses in three dimensions $K \ge 4$ derivatives (less derivatives, $\HARDxDER \ge 2$, are needed for the hard potentials).  Despite the additional length required to establish our first main estimate, there is an economy  of scale to our estimates as a whole because they are proved in a unified framework with modular reusable pieces (and additionally, for example, we do not require detailed commutator estimates), meaning that in the span of the present paper we accomplish existence, uniqueness, positivity, and decay in roughly the same space that it takes to contain the first two papers \cite{arXiv:1005.0447v2,arXiv:1007.0304v1} of their announced four-paper series.

The principal difference of philosophy between the works of \cite{arXiv:1005.0447v2,arXiv:1007.0304v1,newNonCutAMUXY} and this work appears to be the approach to the fundamental anisotropy.  In fact, the effort put forth to make direct estimates like \eqref{AMUXY} in \cite{arXiv:1005.0447v2,arXiv:1007.0304v1,newNonCutAMUXY} is limited 
in the sense that their approach ultimately relies on the comparison to the isotropic Sobolev spaces corresponding to \eqref{isocomp}.  In contrast, our methods crucially and explicitly rely on the sharp, non-isotropic structure of $\spacen$. To attack the issue of anisotropy directly and develop an explicit, geometric understanding of the problem, it was, of course, necessary to introduce several new tools to the Boltzmann theory.  At the price of a loss of weights in the main estimates and theorems, the works \cite{arXiv:1005.0447v2,arXiv:1007.0304v1,newNonCutAMUXY} have, up to this point, been able to proceed without these new tools, along essentially classical lines.  
We hope this comparison is useful to future readers of both articles.

The fourth paper in the series starting with \cite{arXiv:1005.0447v2,arXiv:1007.0304v1} is scheduled to prove also the $C^\infty_{t,x,v}$ smoothing effect for solutions.  We have not attempted this but it has recently been carried out for our solutions by Chen and He \cite{ChenHeSmoothing} using 
\cite{MR2506070}.\footnote{Note added in November 2010; We would like to mention that recently the completed series of papers starting with \cite{arXiv:1005.0447v2,arXiv:1007.0304v1} have been revised, posted, and combined into three papers. These new papers do not obtain the optimal time decay rates for the soft potentials in the whole space.}

\subsection{Entropy Production}\label{sec:ep}
Among Boltzmann's most important contributions to statistical physics was his celebrated $H$-theorem. 
We define the H-functional by
$$
{H}(t)\eqdef -\int_{\mathbb{T}^n}dx~ \int_{\R^n}dv~ F\log F.
$$
Then the Boltzmann $H$-theorem predicts that the entropy is increasing over time
$$
\frac{d{H}(t)}{dt} =  \int_{\mathbb{T}^n}dx~ D(F) \ge 0,
$$
which is a manifestation of the second law of thermodynamics.  Here the entropy production functional, which is non-negative, is defined by
\begin{multline}
\notag
D(F)\eqdef -\int_{\R^n} ~ dv ~ \mathcal{Q}(F,F)\log F
\\
=
\frac{1}{4}
\int_{\R^n}dv \int_{\R^n} dv_* \int_{\sph} d\sigma ~ B(v-v_*, \sigma) \left(F' F'_* - F F_* \right)\log \frac{F' F'_*}{F F_*}.
\end{multline}
Moreover, the entropy production functional is zero if and only if it is operating on a Maxwellian equilibrium.
These formulas formally demonstrate that Boltzmann's equation defines an irreversible dynamics and predicts convergence to Maxwellian in large time.  Of course these predictions are usually non-rigorous  because the regularity  required to perform the above formal calculations is unknown at the moment to be propagated by solutions of the Boltzmann equation in general.

In particular, even though there are many important breakthroughs in this direction, none, so far as we are aware, can be said to completely and rigorously justify the $H$-theorem for the inverse power-law intermolecular potentials with 
$p\in(2,\infty)$.
We refer to \cite{MR0158708,MR1014927,MR2116276,MR1942465,MR2209761,MR1857879,MR2366140}, and the references therein, in this regard, to mention only a brief few works.
Our results in this paper
prove rapid convergence to equilibrium for all of the inverse power collision kernels 
when one considers
 initial data which is close to some global Maxwellian. 
 This convergence is 
the essential prediction of Boltzmann's $H$-theorem.

Furthermore, many works study entropy production estimates in the non cut-off regime, as in 
for instance \cite{MR1649477,MR1765272,MR1715411,MR2052786}. 
These estimates have found widespread utility.  
Our results and anisotropic Sobolev space have the following implications in the fully non-linear context.  We have  the following new lower bound
$$
D(F) \gtrsim
\int_{\mathbb{R}^n} dv ~ \int_{\mathbb{R}^n} dv' ~ 
(\ang{v}\ang{v'})^{\frac{\gamma+2s+1}{2}}
~
 \frac{(\sqrt{F'} - \sqrt{F})^2}{d(v,v')^{n+2s}} 
\ind_{d(v,v') \leq 1} 
- \mbox{l.o.t.}
$$
Precisely, this estimate follows straightforwardly from
our new estimates in Section \ref{sec:mainCOER} when combined for instance with the decomposition \cite[(29)]{MR1715411}   in the particular case when the unknown functions in the term $A(v,v')$ from \cite[(29)]{MR1715411}
are bounded below by, for example, some positive, rapidly decaying function.  
Remarkably, this is the same semi-norm as in the linearized context, and it is a stronger anisotropic and non-local version of the local smoothing estimate from \cite{MR1765272} (stronger, that is, in terms of the weight power multiplied on the order of differentiation).   
This estimate was derived as a result of our effort to find explicit equivalence between the anisotropic norm coming out of the linearized collision operator and the anisotropic Sobolev space $N^{s,\gamma}$.
Furthermore, we have recently utilized the semi-norm part of \eqref{normdef} to prove using new methods this coercive estimate in full generality in \cite{gsNonCutEst} (assuming only, for example, the local conservation laws for the Boltzmann equation).

\subsection{Outline of the rest of the article}\label{sec:outlineA}

 The plan for the rest of the paper is as follows.  In Section \ref{physicalDECrel}, we will formulate the first major physical decomposition of the trilinear form associated with the non-linear collision operator \eqref{gamma0}.  With this we prove the main estimates on the size and support of the decomposed pieces.  We finish this section by formulating the main cancellation inequalities for the hard potentials \eqref{kernelP} using the metric on the paraboloid.
 
In Section \ref{sec2:physicalDECrel}, we prove for the soft potentials \eqref{kernelPsing} the analogous estimates on the size and support of the decomposition, as well as those exploiting cancellations.  We furthermore prove several ``compact estimates'' for 
\eqref{compactpiece}.

In Section \ref{sec:aniLP}, we develop the anisotropic Littlewood-Paley decomposition which is associated to the geometry of the paraboloid.  We further prove estimates connecting the Littlewood-Paley square functions with our norm \eqref{normdef}.

In Section \ref{sec:upTRI}, we prove the key estimates for the trilinear form in Theorem \ref{TriLinEst} and  Lemmas \ref{NonLinEstLOW} and \ref{NonLinEstHIGH}.  These estimates will rely heavily on all of the developments in the previous sections.  The ``compact estimate'' in Lemma \ref{CompactEst} will follow shortly from these developments, and also the sharp linear upper bounds from  Lemma \ref{sharpLINEAR}. 

Section \ref{sec:mainCOER} studies the main coercive inequality.  Here it is shown crucially that the main norm \eqref{normdef} is comparable to both our anisotropic Littlewood-Paley square function and also the space which is  generated by the linearized operator: $\ang{\nPiece g, g}$ below \eqref{normpiece}.  This involves several ideas, including  estimating a Carleman-type representation and what we call a ``Fourier redistribution'' argument.  We further develop useful functional analytic properties of $\spacen$.

Then we show that all of our new singular fractional anisotropic estimates on the velocity variables from the previous seven sections  can be included in the current cut-off theory in Section \ref{sec:deBEest}.
Specifically, we use the space-time  estimates and non-linear energy method that was introduced by Guo \cite{MR2000470,MR2013332,MR1946444}.
This works in particular because our new arguments for the velocity variables outlined above are, morally, fully decoupled from the arguments to handle the space-time aspects of the equation.
We further remark that our estimates above are, in general, flexible enough to adapt to other modern cut-off methods.

Lastly, Appendix \ref{secAPP:HSr} contains Carleman-type representations and a derivation of the ``dual formulation'' for the trilinear form \eqref{3dualZ} that is used in the main text.

\section{Physical decomposition and hard potential estimates: $\gamma \ge -2s$}
\label{physicalDECrel}

In this section we introduce the first major decomposition and prove several estimates which will play a central role in establishing the main inequality for the non-linear term $\Gamma$ from \eqref{gamma0} and the norm $| \cdot |_{N^{s,\gamma}}$.  This first decomposition is a decomposition of the singularity of the collision kernel.  For various reasons, it turns out to be useful to decompose $b ( \cos \theta)$ from \eqref{kernelQ} to regions where $\theta \approx 2^{-k} |v-v_*|^{-1}$, rather than a simpler dyadic decomposition not involving $|v-v_*|$.  The principal benefit of doing so is that this extra factor makes it easier to prove estimates on the space $L^2_{\gamma+2s}(\threed)$  because the weight $\Phi(|v-v_*|)$ from  \eqref{kernelP}  is already present in the kernel and the extra weight $|v-v_*|^{2s}$ falls out automatically from our decomposition.

The estimates to be proved fall into two main categories:  the first are various $L^2$- and weighted $L^2$-inequalities which follow directly from the size and support conditions on our decomposed pieces (such estimates are typically 
called ``trivial'' estimates).  The second type of estimate will assume some sort of smoothness and obtain better estimates than the ``trivial'' estimates by exploiting the cancellation structure of the non-linear term $\Gamma$ from \eqref{gamma0}.  

\subsection{Dyadic decomposition of the singularity}\label{sec:dyadic}
Let $\{ \chi_k \}_{k=-\infty}^\infty$ be a partition of unity on $(0,\infty)$ such that $\nsm \chi_k\nsm_{L^\infty} \leq 1$ and 
$\mbox{supp}\left( \chi_k \right)\subset [2^{-k-1},2^{-k}]$.  
For each $k$:
\[
B_k = B_k(v-v_*,\sigma) \eqdef \Phi(|v-v_*|) ~b \left( \left< \frac{v-v_*}{|v-v_*|}, \sigma \right> \right) \chi_k (|v - v'|). 
\]
Note that
\[ 
|v-v'|^2 = \frac{|v-v_*|^2}{2} \left( 1 - \left< \frac{v-v_*}{|v-v_*|}, \sigma \right>  \right)
= |v-v_*|^2 \sin^2 \frac{\theta}{2}.
\]
Hence, the condition $|v-v'| \approx 2^{-k}$ is equivalent to the condition that the angle between $\sigma$ and $\frac{v-v_*}{|v-v_*|}$ is comparable to $2^{-k} |v-v_*|^{-1}$.  
With this partition, we define
\begin{equation}
\begin{split}
\teePLUSop(g,h,f)  & \eqdef \int_{\mathbb{R}^n} dv \int_{\mathbb{R}^n} dv_* \int_{\sph} d \sigma  ~B_k(v-v_*, \sigma) ~ g_* h  M_\beta(v_*') ~ f'   w^{2\ell}(v'),  \\ 
\teeMINUSop(g,h,f)  & \eqdef \int_{\mathbb{R}^n} dv \int_{\mathbb{R}^n} dv_* \int_{\sph} d \sigma ~B_k(v-v_*, \sigma) ~ g_* h  M_\beta(v_*) ~ f  w^{2\ell}(v). 
\end{split}
\label{defTKL}
\end{equation}
We use the notation
$
M_\beta (v)
=
\partial_{\beta} M = p_\beta(v) M(v)$ where $p_\beta(v)$ is the appropriate polynomial of degree $|\beta|$.  It suffices to use the inequality $|\partial_\beta M| \lesssim M^{1/2}$.

It turns out that we will also need to express the collision operator \eqref{gamma0} using its ``dual formulation.''  With the variant of Carleman's representation coming from Proposition \ref{carlemanA} and the notation 
$
M_*' = M(v+v_* - v'),
$
we record here the following alternative representation when $\beta = 0$ (with cancellations on $h$) 
\begin{multline}
\ang{ w^{2\ell} \Gamma(g,h),  f }
\eqdef
\int_{\threed} dv'   \int_{\threed} dv_* \int_{E_{v_*}^{v'}} d \pi_{v}  
~\tilde{B}  g_*  f'  w^{2\ell}(v')  ~ \left( M_*' h - M_* h' \right)
\\
+
\LopGstar (g,f,h),
\label{dualOPdef}
\end{multline}
using \eqref{3dualZ},
where the kernel $\tilde{B}$ is given by
\begin{equation}
\tilde{B}
\eqdef
2^{n-1}
\frac{B\left(v-v_*, \frac{2v' - v- v_*}{|2v' - v- v_*|}\right) }{  |v'-v_*| ~ |v-v_*|^{n-2}}
\label{kernelTILDE}
\end{equation}
and the operator $\LopGstar=\LopGstar(g,h,f)$ above  does not differentiate at all:
\begin{equation}
\LopGstar
 \eqdef 
 \int_{\threed} dv'  f'  h'  ~ w^{2\ell}(v')   \int_{\threed} dv_*  ~  g_*  (\partial_\beta M_* ) \int_{E_{v_*}^{v'}} d \pi_{v} ~  \tilde{B} 
\left(1 -  \frac{  |v'-v_*|^{n+\gamma} }{ |v-v_*|^{n+\gamma}} \right).
 \label{opGlabel}
\end{equation}
Note that in \eqref{dualOPdef} we use $\LopGstar$ with $\beta =0$, however below it will be useful to consider arbitrary multi-indicies $\beta$.
We also use the notation $\Gamma_{*}^0 = \opGstar$.  With these developments, 
we record here the following alternative representation for $\teePLUSop$ as well 
as the definition of a third trilinear operator $\teeSTARop$ (based on the calculation \eqref{3dualZ} with, recall, $v_*' = v+v_* - v'$):
\begin{equation}
\begin{split}
\teePLUSop(g,h,f)  & 
=  \int_{\threed} dv' ~ w^{2\ell}(v')  \int_{\threed} dv_* \int_{E_{v_*}^{v'}} d \pi_{v} ~ 
\tilde{B}_k ~  g_*  f' M_\beta(v_*') h, 
\\ 
\teeSTARop(g,h,f) & 
\eqdef
 \int_{\threed} dv' ~ w^{2\ell}(v')  \int_{\threed} dv_* \int_{E_{v_*}^{v'}} d \pi_{v} ~ 
\tilde{B}_k ~ g_* f'   M_\beta(v_*) h',
\end{split}
\label{defTKLcarl}
\end{equation}
where we use the notation 
$$
\tilde{B}_k
\eqdef
2^{n-1}
\frac{B\left(v-v_*, \frac{2v' - v- v_*}{|2v' - v- v_*|}\right) }{  |v'-v_*| ~ |v-v_*|^{n-2}} \chi_k (|v - v'|).
$$
In these integrals above $d\pi_{v}$ is Lebesgue measure on the $(\dim - 1)$-dimensional plane $E_{v_*}^{v'}$ passing through $v'$ with normal $v' - v_*$, and $v$ is the variable of integration.

Now recall the collision operator \eqref{gamma0} and the collisional variables \eqref{sigma}.  With the change of variables $u = v_* - v$, and $u^\pm = (u \pm |u| \sigma)/2$, we may write \eqref{gamma0} as
\begin{gather*}
\Gamma (g,h)
 =
 \int_{\mathbb{R}^n} du \int_{\sph} d \sigma ~ B(u, \sigma)  M(u+v)  \left\{g(v+u^+) h(v+u^-) - g(v+u) h(v)\right\}. 
\end{gather*}
Differentiating this formula and  applying the inverse coordinate change allows us to express derivatives of the bilinear collision operator $\Gamma$ as
\begin{gather}
\label{DerivEstG}
\partial^{\alpha}_{\beta}\Gamma (g,h)
 =
 \sum_{\beta_1 + \beta_2  = \beta} 
\sum_{ \alpha_1 \le \alpha} ~ C^{\beta, \beta_1, \beta_2}_{\alpha, \alpha_1} ~ 
\Gamma_{\beta_2} (\partial^{\alpha - \alpha_1}_{\beta - \beta_1}g,\partial^{\alpha_1}_{\beta_1}h).
\end{gather}
Here $C^{\beta, \beta_1, \beta_2}_{\alpha, \alpha_1}$ is a non-negative constant which is derived from the Leibniz rule.
Also, $\Gamma_{\beta}$ is the bilinear operator with derivatives on the Maxwellian $M$ given by
$$
\Gamma_{\beta} (g,h)
=
 \int_{\mathbb{R}^n} dv_* \int_{\sph} d \sigma ~ B(v-v_*,\sigma) ~ M_\beta (v_*) ~ (g_*' h' - g_* h).
$$
Since an expression analogous to \eqref{dualOPdef} holds for $\Gamma_{\beta}$, we can partition these using the dyadic decomposition of the singularity above.
Then for $f, g, h \in \mathcal{S}(\mathbb{R}^n)$, the pre-post collisional change of variables, the dual representation, and the previous calculations guarantee that 
\begin{align*}
 \left< w^{2\ell} \Gamma_\beta (g,h), f \right> & = 
\sum_{k=-\infty}^\infty  \left\{ \teePLUSop(g,h,f) - \teeMINUSop(g,h,f) \right\}
\\
& = \LopGstar (g,h,f)
+
\sum_{k=-\infty}^\infty  \left\{ \teePLUSop(g,h,f) - \teeSTARop(g,h,f) \right\} . 
\end{align*}
These will be the general quantities that we estimate in the following two sections.
The first step is to estimate each of  $\teePLUSop$, $\teeMINUSop$, and $\teeSTARop$ using only the known constraints on the size and support of $B_k$.

\subsection{``Trivial'' analysis of the decomposed pieces}

We will now prove several size and support estimates for the decomposed pieces of the Boltzmann collision operator.  The reader should note that all estimates proved here (and in Section \ref{sec:cancelHARD}) hold under the assumptions that $\gamma > -n$ and $\gamma + 2s > -\frac{n}{2}$.

We begin with the following:

\begin{proposition}\label{prop1}
For any integer $k$, $\ell \in \R$ and $m\ge 0$, we have the uniform estimate:
\begin{equation}
 \left| \teeMINUSop(g,h,f) \right|   \lesssim 2^{2sk} \nsm g\nsm_{L^2_{-m}} \nsm w^{\ell} h\nsm_{L^2_{\gamma+2s}}
 \nsm w^{\ell} f\nsm_{L^2_{\gamma+2s}}.
 \label{tminusgh} 
\end{equation}
\end{proposition}

\begin{proof}
Given the size estimates for $b(\cos \theta)$ in \eqref{kernelQ} and the support of $\chi_k$, clearly
\begin{equation}
\int_{\sph} d \sigma ~ B_k \lesssim
\Phi(|v-v_*|)
\int_{2^{-k-1} |v - v_*|^{-1}}^{2^{-k} |v - v_*|^{-1}}  d \theta ~ \theta^{-1-2s} 
\lesssim 2^{2sk} |v-v_*|^{\gamma+2s}.
\label{bjEST}
\end{equation}
Thus
\[ 
\left| \teeMINUSop(g,h,f) \right| \lesssim 2^{2sk} \int_{\threed} dv \int_{\threed} dv_* \sqrt{M_*} |v-v_*|^{\gamma+2s} |g_*| |h f| 
~ w^{2\ell}(v). 
\]
With Cauchy-Schwartz this is 
\begin{multline}
\left| \teeMINUSop(g,h,f) \right| \lesssim 2^{2sk}
\left( \int_{\threed} dv \int_{\threed} dv_* \sqrt{M_*}  \ang{v}^{\gamma+2s} |g_*|^2 |h|^2 w^{2\ell}(v) \right)^{1/2}
\\
\times 
\left( \int_{\threed} dv ~  w^{2\ell}(v)  |f|^2 \ang{v}^{-\gamma-2s} \int_{\threed} dv_* ~ \sqrt{M_*}  |v-v_*|^{2(\gamma+2s)}  \right)^{1/2}
\\
\lesssim 2^{2sk} \nsm g\nsm_{L^2_{-m}} \nsm w^{\ell} h\nsm_{L^2_{\gamma+2s}}
 \nsm w^{\ell} f\nsm_{L^2_{\gamma+2s}},
\label{cauchySCHminus}
\end{multline}
where we have used the inequality $\int_{\threed} dv_* ~ \sqrt{M_*}  |v-v_*|^{2(\gamma+2s)} \lesssim \ang{v}^{2(\gamma+2s)}$, which holds whenever $2(\gamma+2s) > -n$. This completes the proof of \eqref{tminusgh}.
\end{proof}

\begin{proposition}\label{propSTAR}
For any integer $k$, $\ell\in\R$ and $m\ge 0$, 
the inequality is uniform: 
\begin{equation}
 \left| \teeSTARop(g,h,f) \right| 
  \lesssim 
  2^{2sk}  \nsm g\nsm_{L^2_{-m}} \nsm w^{\ell} h\nsm_{L^2_{\gamma+2s}}
 \nsm w^{\ell} f\nsm_{L^2_{\gamma+2s}}.
  \label{tstargh} 
\end{equation}
\end{proposition}

\begin{proof}
As in Proposition \ref{prop1}, the key to these inequalities is the symmetry between $h$ and $f$ coupled with Cauchy-Schwartz.  The difference is that, this time, the Carleman representation will be used and the main integrals will be over $v_*$ and $v'$.  In this case, the quantity of interest is
\[ 
\int_{E_{v_*}^{v'}}  d \pi_{v} ~ b \left( \frac{|v'-v_*|^2 - |v - v'|^2}{|v' - v_*|^2 + |v - v'|^2} \right) \frac{ \chi_k(|v-v'|) }{|v'-v_*| |v-v_*|^{n-2}}. 
\]
The support condition yields $|v-v'| \approx 2^{-k}$.  Moreover, since $b(\cos \theta)$ vanishes for $\theta\in [\pi/2,\pi]$, we have
$|v' - v_*| \ge |v' - v|$. 
Consequently, the condition \eqref{kernelQ} gives 
\begin{equation}
b \left( \frac{|v'-v_*|^2 - |v - v'|^2}{|v' - v_*|^2 + |v - v'|^2} \right)
\lesssim
\left(\frac{ |v - v'|^2}{|v' - v_*|^2} \right)^{-\frac{n-1}{2}-s}.
\notag
\end{equation}
Since $|v-v_*|^2=|v' - v_*|^2 + |v - v'|^2$ on $E_{v_*}^{v'}$, the integral is bounded by a uniform constant times the following quantity
\[ 
\int_{E_{v_*}^{v'}}  d \pi_{v} ~\frac{|v'-v_*|^{n-1+2s}}{|v-v'|^{n-1+2s}} |v'-v_*|^{-n+1} \chi_k(|v-v'|) \lesssim 2^{2sk} |v'-v_*|^{2s}. 
\]
Note that the condition $|v' - v_*| \geq |v - v'|$ implies $|v - v_*| \approx |v' - v_*|$.  With these estimates it follows that
\begin{equation}  \label{bjCARLest}
\int_{E_{v_*}^{v'}} d \pi_{v}  ~\tilde{B}_k  \lesssim  2^{2sk} |v'-v_*|^{\gamma+2s}.
\end{equation}
As a result,
\begin{equation}
\left| \teeSTARop(g,h,f) \right| \lesssim 2^{2sk} \int_{\mathbb{R}^n} dv' \int_{\mathbb{R}^n} dv_*  ~ \sqrt{M_*}  ~ |v' - v_*|^{\gamma+2s} |g_* h' f'| ~ w^{2\ell}(v'). 
\label{AAeqREF}
\end{equation}
Now the relevant estimate can be established as in Proposition \ref{prop1} and \eqref{cauchySCHminus}.
\end{proof}

We will now estimate the operator $\teePLUSop$, which is more difficult and more technical because it contains the post-collisional velocities \eqref{sigma}.
A key problem here is to be able to distribute negative decaying velocity weights among the functions $g$, $h$, and $f$.  In the previous propositions this distribution could be accomplished more easily. 
Because the weight is in the $v$ variable and the unknown functions are of the variables $v'$ and $v'_*$ we have to work harder to distribute the negative weights.  Our methods in the proof of Proposition \ref{referLATERprop} enable us to only obtain one term in the upper bound of \eqref{tplussmallK}; this is a substantial improvement over, e.g., \cite{MR2013332,MR1946444,MR2000470}.

\begin{proposition}
\label{referLATERprop}
Fix an integer $k$ and $\ell^+, \ell^- \ge 0$,
with $\ell = \ell^+ - \ell^-$.  For any  $0\le \ell' \le \ell^-$,
with $\ell + \ell' = \ell^+ - (\ell^- - \ell')$,
  we   have the uniform estimate:
\begin{equation} 
\begin{split}
\left|  \teePLUSop(g,h,f)  \right| \lesssim   & ~
2^{2sk} 
 \nsm  w^{\ell^+ - \ell'} g\nsm_{L^2} 
\nsm w^{\ell + \ell'} h\nsm_{L^2_{\gamma + 2s}} 
\nsm w^{\ell} f\nsm_{L^2_{\gamma + 2s}}.
\end{split}
 \label{tplussmallK} 
\end{equation}
This estimate holds, in fact, whenever $\gamma + 2s > -\frac{n}{2}$.
\end{proposition}

\begin{proof}
We plan to estimate $\teePLUSop(g,h,f)$ from above as
\begin{equation}
\left| \teePLUSop(g,h,f)  \right| \lesssim \int_{\mathbb{R}^n} dv \int_{\mathbb{R}^n} dv_* \int_{\sph} d \sigma  ~B_k ~ w^{2\ell}(v') ~  \left| g_* h \right|     \left| f' \right| \sqrt{M_*'}.
\label{tBabove}
\end{equation}
Consider first the situation when $\gamma + 2s \leq 0$:
By Cauchy Schwartz, the right-hand side is bounded above by a uniform constant times the product
\begin{equation} 
\begin{split}
\left(  \int_{\mathbb{R}^n} dv \int_{\mathbb{R}^n} dv_* \int_{\sph} d \sigma ~ 
\frac{B_k(v-v_*, \sigma)}{|v-v_*|^{\gamma+2s} } ~ |g_*|^2 |h|^2 \ang{v'}^{\gamma+2s} w^{2\ell}(v')
\sqrt{M_*'}   \right)^{\frac 12}
\\
\times \left( \int_{\mathbb{R}^n} dv \int_{\mathbb{R}^n} dv_* \int_{\sph} d \sigma ~ 
\frac{B_k(v-v_*, \sigma)}{|v-v_*|^{-\gamma-2s} } \ang{v'}^{-\gamma-2s} w^{2\ell}(v') |f'|^2 \sqrt{M_*'}  \right)^{\frac 12}.
\end{split} \label{tpe}
\end{equation}
After a pre-post change of variables, the second factor equals $2^{sk}$ times a term identical to the corresponding factor of \eqref{cauchySCHminus} by virtue of \eqref{bjEST}; in particular, it is bounded above by a uniform constant times $2^{sk} |w^\ell f|_{L^2_{\gamma+2s}}$.  Note that the inequality
\[ 
\int_{\threed} dv_* \sqrt{M_*} |v-v_*|^{2(\gamma+2s)} \lesssim \ang{v}^{2(\gamma+2s)}, 
\]
we implicitly use here is the only place in this proposition where there are any constraints on $\gamma$ and $s$ (since the singularity has been ``removed'' in the first factor of \eqref{tpe} by Cauchy-Schwartz).

The first factor, on the other hand, requires more cleverness.  First of all, note that, if $|v'|^2 \leq \frac{1}{2} (|v|^2 + |v_*|^2)$, then from the collisional conservation laws $M_*' \leq \sqrt{M M_*}$; consequently on this region of $v'$, one has $\ang{v'}^{\gamma+2s} w^{2\ell}(v') \sqrt{M_*'} \lesssim \ang{v}^{-m} \ang{v_*}^{-m}$ for any fixed, positive $m$.
On the region $|v'|^2 \geq \frac{1}{2} (|v|^2 + |v_*|^2)$, it follows from the collisional geometry that $|v'|^2 \approx |v|^2 + |v_*|^2$.  For any $\ell \in \R$, we have
\[ 
w^{2 \ell}(v')  \approx w^{2 \ell}(|v| + |v_*|) \lesssim w^{2(\ell^+ - \ell')}(v_*) w^{2(\ell + \ell')}(v),
 \]
since $2 \ell \leq 2(\ell^+ - \ell')$, $2 \ell \leq 2(\ell + \ell')$, and $2(\ell^+ - \ell') + 2(\ell + \ell') \geq 2 \ell$.  Similarly, since $\gamma + 2s \leq 0$, we have
$
\ang{v'}^{\gamma+2s} \lesssim \ang{v}^{\gamma+2s}.
$
Thus
\[\ang{v'}^{\gamma + 2s} w^{2\ell}(v') \sqrt{M_*'} \lesssim w^{2(\ell^+ - \ell')}(v_*) w^{2(\ell + \ell')}(v) \ang{v}^{\gamma+2s}. \]
Substituting this into the first factor of \eqref{tpe} and using \eqref{bjEST} as before clearly establishes \eqref{tplussmallK}.  

Now consider the case when $\gamma + 2s \geq 0$.  This time, instead of \eqref{tpe}, we estimate $\left| \teePLUSop(g,h,f)  \right|$ from above by
\begin{equation} 
\begin{split}
\left(  \int_{\mathbb{R}^n} \! \! \! dv \int_{\mathbb{R}^n} \! \! \! dv_* \int_{\sph} \!\! \! d \sigma ~ 
\frac{B_k(v-v_*, \sigma)}{|v'-v_*|^{\gamma+2s}} ~ |g_*|^2 |f'|^2  \ang{v}^{\gamma+2s} \! \frac{w^{4\ell}(v')}{w^{2 (\ell + \ell')}(v)}
\sqrt{M_*'}   \right)^{\frac{1}{2}}
\\
\times \left( \int_{\mathbb{R}^n} \! \! \! dv \int_{\mathbb{R}^n} \! \! \! dv_* \int_{\sph} \! \! \! d \sigma ~ 
 \frac{B_k(v-v_*, \sigma)}{|v'-v_*|^{-\gamma-2s}} \ang{v}^{-\gamma-2s} w^{2(\ell + \ell')}(v) |h|^2 \sqrt{M_*'}  \right)^{\frac{1}{2}}.
\end{split} \label{tpe2}
\end{equation}
As in the previous case, the second factor is readily estimated.  This time, after a pre-post change of variables, the Carleman representation in Proposition \ref{carlemanA} is used along with \eqref{bjCARLest} for the exponent $2( \gamma + 2s)$ to conclude that the second factor is at most a fixed constant times 
$2^{sk} \nsm w^{\ell + \ell'} h\nsm_{L^2_{\gamma + 2s}}$.

Regarding the first factor, just as in the previous case, we have the bound
$w^{2 \ell} (v') \sqrt{M_*'} \lesssim w^{2(\ell^+-\ell')}(v_*) w^{2 (\ell + \ell')}(v)$.
The region without rapid decay in all variables is $|v'|^2 \approx |v|^2 + |v_*|^2$.  Thus, if $\gamma + 2s \geq 0$, we have $\ang{v}^{\gamma+2s} \lesssim \ang{v'}^{\gamma+2s}$ (which will also be true when $|v-v'| \leq 1$).  Thus, the same estimate \eqref{bjCARLest} establishes that the first term is bounded uniformly above by $2^{sk} \nsm  w^{\ell^+ - \ell'} g\nsm_{L^2} \nsm w^{\ell} f\nsm_{L^2_{\gamma + 2s}}$.
\end{proof}

\begin{proposition}\label{opGstarEST}
We have the following uniform estimate for \eqref{opGlabel} when $\gamma > -\frac{n}{2}$:
\begin{equation}
\left| \LopGstar(g,h,f)   \right|  \lesssim
      \nsm g\nsm_{L^2_{-m}} 
  \nsm w^\ell h\nsm_{L^2_{\gamma}} 
 \nsm w^\ell f\nsm_{L^2_{\gamma}}.
\label{opGstarINEQ1}
\end{equation}
If $\gamma + 2s > -\frac{n}{2}$, then we have the alternate uniform inequality for some $\delta > 0$
\begin{equation}
\left| \LopGstar(g,h,f)   \right|  \lesssim
      \nsm g M^\delta \nsm_{L^2} 
  \nsm  h M^\delta \nsm_{H^{s+\epsilon}} 
 \nsm  f M^\delta \nsm_{H^{s-\epsilon}} 
 + 
 \nsm g \nsm_{L^2_{-m}} 
  \nsm w^\ell h\nsm_{L^2_{\gamma}} 
 \nsm w^\ell f\nsm_{L^2_{\gamma}}.
\label{opGstarINEQ2}
\end{equation}
These inequalities hold for all $\ell\in\mathbb{R}$, $m\ge 0$ and $0 \leq \epsilon \leq \min \{s,1-s\}$.
\end{proposition}

\begin{proof}  The key quantity to estimate is the integral on $E_{v_*}^{v'}$
of
$\tilde{B}_k 
\left( 1- A \right),
$
where
$$
A \eqdef
\frac{ \Phi(v'-v_*) |v'-v_*|^{n} }{\Phi(v-v_*) |v-v_*|^{n}} =
 \left( \frac{|v' - v_*|^2}{|v-v' |^2+|v' - v_*|^2}\right)^{\frac{n+\gamma}{2}}.
$$
Now for any fixed $\alpha > 0$, one has
$
\left| c^\alpha - 1 \right|
\lesssim
\left| c - 1 \right|
$
uniformly for $0\le c \le 1$; thus
\begin{align*}
\int_{E_{v_*}^{v'}} d \pi_{v} &  ~\tilde{B}_k 
\left|  A    -   1 \right|
\lesssim
\int_{E_{v_*}^{v'}} d \pi_{v}  ~\tilde{B}_k ~
\frac{|v-v' |^2}{|v-v' |^2+|v' - v_*|^2} \lesssim
\int_{E_{v_*}^{v'}} d \pi_{v}  ~\tilde{B}_k ~
\frac{|v-v' |^2}{|v' - v_*|^2} 
\\ 
&
\lesssim
2^{(2s-2) k} |v' - v_*|^{\gamma+2s-2},
\end{align*}
where we use the inequality $|v - v'| \leq |v' - v_*|$ as well as \eqref{bjCARLest}.
We now sum over $k$, noting that, for fixed distance $|v' - v_*|$, the terms for which $2^k |v' - v_*| \leq \frac{1}{4}$ will vanish identically since $2^k |v - v'| \geq \frac{1}{2}$ and $|v' - v_*| \ge | v' - v|$. 
Thus
\begin{multline*} 
\int_{E_{v_*}^{v'}}  d \pi_v ~ \tilde{B} ~ |A-1| 
\lesssim 
\sum_{k:~2^k |v' - v_*| > 1} 2^{(2s-2) k} |v' - v_*|^{\gamma+2s-2} 
 \lesssim   |v' - v_*|^{\gamma}. 
 \end{multline*} 
Now we complete the estimate  
\begin{equation}
 |\LopGstar(g,h,f)| \lesssim \int_{\R^n} dv' \int_{\R^n} dv_* |f' h'| w^{2 \ell}(v') |g_*| \sqrt{M_*} |v' - v_*|^{\gamma} \label{hls0}
\end{equation}
with Cauchy-Schwartz as in \eqref{cauchySCHminus}, which follows as usual when $\gamma > - \frac{n}{2}$.  Moreover, if the integral on the right-hand side of \eqref{hls0} is restricted to the region 
$|v' - v_*| \ge 1$, the condition $\gamma > - \frac{n}{2}$ is unnecessary since the singularity is avoided.  Thus, we will be able to establish \eqref{opGstarINEQ2} immediately after the next proposition, taking $b_1 = s + \epsilon$, $b_2 = s- \epsilon$ and $0 \leq \epsilon \leq \min\{s,1-s\}$.
\end{proof}

\begin{proposition}
Let $\rho > -n$ and $0 \leq b_1,b_2 \leq \frac{n}{2}$ satisfy $\rho + b_1 + b_2 > - \frac{n}{2}$, and let
\[ {\mathrm{HLS}}(g,h,f) \eqdef \int_{\R^n} dv_* \int_{\R^n}  dv |g_*| \sqrt{M_*} |v-v_*|^{\rho} w^{2 \ell}(v) |h||f| {\mathbf 1}_{|v-v_*| \leq 1}. \]
Then there is some $\delta > 0$ such that
\begin{equation} \mathrm{HLS}(g,h,f) \lesssim 
  |g M^\delta|_{L^2}|h M^{\delta}|_{H^{b_1}}|f M^{\delta}|_{H^{b_2}}. \label{hlsineq}
\end{equation}
\end{proposition}
\begin{proof}
We begin by observing that, for all positive $\delta$ sufficiently small, we have $w^{2 \ell}(v) {\mathbf 1}_{|v-v_*| \leq 1}\sqrt{M_*} \lesssim M_*^{\delta} M^{2 \delta}$.  Thus
\begin{align} 
{\mathrm{HLS}}(g,h,f) \lesssim
\int_{\R^n} dv_* \int_{\R^n} dv |g_*| M_*^\delta |v-v_*|^{\rho} M^{2 \delta} |h||f|. \label{hls1}
\end{align}
Notice that the inequality \eqref{hlsineq} follows exactly as in \eqref{cauchySCHminus} if $\rho > - \frac{n}{2}$ for $b_1 = b_2 = 0$.  We therefore assume that $\rho \leq - \frac{n}{2}$.
The right-hand side of \eqref{hls1} is equal to
\[\int_{\R^n} dv \ I_{n+\rho}(|g|M^\delta)  M^{2 \delta} |h||f|,\]
where $I_{n+\rho}$ 
is the classical fractional integral operator of order $n+\rho$.  Thus, by the Hardy-Littlewood-Sobolev inequality (see, e.g., Stein \cite{MR0290095}, p. 119) applied to $g$ as well as the usual $L^2$-Sobolev embedding theorem applied to $h$ and $f$, we have:
\[
 | h M^{\delta}|_{L^{q_1}} \lesssim |h M^{\delta}|_{H^{b_1}},  \ \ | f M^{\delta}|_{L^{q_2}} \lesssim |f M^{\delta}|_{H^{b_2}}, \ \ |I_{\alpha}(|g|M^\delta)|_{L^{q_3}} \lesssim |g M^\delta|_{L^2},
 \]
where the exponents $q_i$ satisfy $\frac{1}{q_i} > \frac{1}{2} - \frac{b_i}{n}$ for $i=1,2$ and $\frac{1}{q_3} > \frac{1}{2} - \frac{n+\rho}{n}$.  Note that it is necessary in each case that $\frac{b_i}{n} \leq \frac{1}{2}$ as well as $\frac{n+\rho}{n} \leq \frac{1}{2}$.  By H\"{o}lder's inequality, we will have the second term bounded above by $|g M^\delta|_{L^2}|h M^{\delta}|_{H^{b_1}}|f M^{\delta}|_{H^{b_2}}$ as long as we may find exponents $q_i$ such that $\frac{1}{q_1} + \frac{1}{q_2} + \frac{1}{q_3} \leq 1$, which will be possible exactly when $\frac{\rho + b_1 + b_2}{n} > -\frac{1}{2}$.
\end{proof}

\subsection{Cancellations with hard potentials: $\gamma \ge -2s$}\label{sec:cancelHARD}
In this section, we seek to establish estimates for the differences $\teePLUSop  - \teeMINUSop$ and $\teePLUSop - \teeSTARop$. We wish the estimates to have good dependence on $k$ (in particular, we would like the norm to be a negative power of $2^k$), but this improved norm will be paid for by assuming differentiability of one of the functions $h$ or $f$.  The key obstacle to overcome in making these estimates is that the magnitude of the gradients of $h$ and $f$ must be measured in some anisotropic way; this is a point of fundamental importance, as the scaling is imposed upon us by the structure of the ``norm piece'' $\ang{\nPiece f,f}$.

The scaling dictated by the problem is that of the paraboloid: namely, that the function $f(v)$ should be thought of as the restriction of some function $F$ of $\last$ variables to the paraboloid $(v, \frac{1}{2} |v|^2)$.  Consequently, the correct metric to use in measuring the length of vectors in $\threed$ will be the metric on the paraboloid in $\R^{\last}$ induced by the $(\last)$-dimensional Euclidean metric.  To simplify the calculations, we will work directly with the function $F$ rather than $f$ and take its $(\last)$-dimensional derivatives in the usual Euclidean metric.  This will be sufficient for our purposes since our Littlewood-Paley-type decomposition will give us a natural way to extend the projections $Q_j f$ into 
$\last$ dimensions while preserving the relevant differentiability properties of the $\dim$-dimensional restriction to the paraboloid.

To begin, it is necessary to find a suitable formula relating differences of $F$ at nearby points on the paraboloid to the various derivatives of $F$ as a function of $n+1$ variables.  To this end, fix any two $v,v' \in \mathbb{R}^n$, and consider $\CurveP : [0,1] \rightarrow \mathbb{R}^n$ and $\CurveEP : [0,1] \rightarrow \mathbb{R}^{n+1}$ given by
\[ 
\CurveP(\Ctheta) \eqdef \Ctheta v' + (1-\Ctheta) v,
\quad 
\mbox{ and } 
\quad 
\CurveEP(\Ctheta) \eqdef  \left(\Ctheta v' + (1-\Ctheta)v, \frac{1}{2} \left| \Ctheta v' + (1-\Ctheta) v \right|^2 \right). 
\]
Now $\CurveEP$ lies on the paraboloid 
$\set{(v_1,\ldots,v_{n+1}) \in \mathbb{R}^{n+1}}{ v_{n+1} = \frac{1}{2} (v_1^2 + \cdots + v_n^2)}$, and $\CurveP(0) = v$ and $\CurveP(1) = v'$.  
Also consider the starred analogs defined by
\[ 
\starCurveP(\Ctheta) \eqdef \Ctheta v'_* + (1-\Ctheta) v_*,
\quad 
\mbox{ and } 
\quad 
\starCurveEP(\Ctheta) \eqdef  \left(\Ctheta v'_* + (1-\Ctheta)v_*, \frac{1}{2} \left| \Ctheta v'_* + (1-\Ctheta) v_* \right|^2 \right). 
\]
Elementary calculations, and \eqref{sigma}, show that $\CurveP(\Ctheta) + \starCurveP(\Ctheta) = v + v_*$ and 
\begin{align*}
\frac{d \CurveEP}{d \Ctheta} = \left( v' - v, \ang{\CurveP(\Ctheta), v' - v} \right) 
\quad 
& \mbox{ and }
\quad 
 \frac{d^2 \CurveEP}{d \Ctheta^2} = (0, |v'-v|^2), \\
\frac{d \starCurveEP}{d \Ctheta} = - \left( v' - v, \ang{\starCurveP(\Ctheta), v' - v} \right)
\quad 
& \mbox{ and }
\quad 
 \frac{d^2 \starCurveEP}{d \Ctheta^2} = (0, |v'-v|^2).
 \end{align*}
Now we use the standard trick of writing the difference of $F$ at two different points in terms of an integral of a derivative (in this case the integral is along the path $\CurveP$):
\begin{align} 
F\left(v',\frac{|v'|^2}{2} \right) - F\left(v, \frac{|v|^2}{2} \right) & = \int_0^1 d \Ctheta ~ \frac{d}{d \Ctheta} F(\CurveEP(\Ctheta)) \nonumber \\
& = \int_0^1 d \Ctheta  \left( \frac{d \CurveEP}{d \Ctheta} \cdot  (\tilde{\nabla} F) (\CurveEP(\Ctheta)) \right), 
\label{paraboladiff}
\end{align}
where the dot product on the right-hand side is the usual Euclidean inner-product on $\mathbb{R}^{n+1}$ and $\tilde{\nabla}$ is the $(n+1)$-dimensional gradient of $F$. 
For convenience we define
\[ 
|\tilde{\nabla}|^i F(v_1,\ldots,v_{n+1}) \eqdef 
\max_{0\le j \leq i}\sup_{|\xi| \leq 1} \left| \left(\xi \cdot \tilde{\nabla} \right)^j F(v_1,\ldots,v_{n+1}) \right|, 
\quad
  i=0,1,2,
\]
where $\xi \in \mathbb{R}^{n+1}$ and $|\xi|$ is the usual Euclidean length. In particular, note that we have defined $|\tilde{\nabla}|^0 F = |F|$.

If $v$ and $v'$ are related by the collision geometry \eqref{sigma}, then $\ang{v-v',v'-v_*} = 0$, which yields that 
\begin{align*}
\ang{\CurveP(\Ctheta),v'-v} & = \ang{v_*,v'-v} - (1-\Ctheta) |v-v'|^2, \\
\ang{\starCurveP(\Ctheta),v'-v} & = \ang{v_*,v'-v} - \Ctheta |v-v'|^2.
\end{align*}
Thus, whenever $|v-v'| \leq 1$, which holds near the singularity, we have
\begin{equation}
\left| \frac{d \starCurveEP}{d \Ctheta} \right|
+
\left| \frac{d \CurveEP}{d \Ctheta} \right| \lesssim |v-v'| \ang{v_*}. \label{derivEST}
\end{equation}
 Throughout this section we suppose that $|v-v'| \leq 1$ since this is the situation where our cancellation inequalities will be used.
 In particular, we have the following inequality for differences related by the collisional geometry:
\begin{align}
 \left| F\left(v',\frac{|v'|^2}{2}\right) - F\left(v, \frac{|v|^2}{2}\right)\right| 
 &
  \lesssim 
  \ang{v_*} |v-v'|  \int_0^1  d \Ctheta ~ |\tilde{\nabla}| F (\CurveEP(\Ctheta)). 
 \label{paradiff1} 
\end{align}
Furthermore,
by subtracting the linear term from both sides of \eqref{paraboladiff} and using the integration trick iteratively on the integrand of the integral on the right-hand side of \eqref{paraboladiff}, we obtain
\begin{multline}
\left| F\left(v',\frac{|v'|^2}{2} \right) -  F\left(v, \frac{|v|^2}{2}\right) -  \frac{d \CurveEP}{d \Ctheta}(0) \cdot \tilde{\nabla} F(v) \right| 
\\
 \lesssim
\ang{v_*}^2 |v-v'|^2 \int_0^1 d \Ctheta  ~ | \tilde{\nabla}|^2 F (\CurveEP(\Ctheta)). 
\label{paradiff2}
\end{multline}
We note that, by symmetry, the same result holds when the roles of $v$ and $v'$ are reversed (which only changes the curve $\CurveEP$ by reversing the parametrization: $\CurveEP(\Ctheta)$ becomes $\CurveEP(1-\Ctheta)$).  
It is also trivially true that the corresponding starred version of \eqref{paradiff2} holds as well.
We will use these two basic cancellation inequalities to prove the  cancellation estimates for the trilinear form in the following propositions.

\begin{proposition}
Suppose $f$ is a Schwartz function on $\threed$ given by the restriction of some Schwartz function $F$ 
on $\R^{\last}$ to the paraboloid $(v, \frac{1}{2} |v|^2)$.  For each $i$, let $|\tilde{\nabla}|^i f$ be the restriction of $|\tilde{\nabla}|^i F$ to the same paraboloid.  Then, for any $k \geq 0$, 
\begin{align}
 \left| (\teePLUSop- \teeMINUSop)(g,h,f) \right| \ & \lesssim 2^{(2s-i) k} \nsm g\nsm_{L^2_{-m}} \nsm w^{\ell} h\nsm_{L^2_{\gamma+2s}} \! \! \nl w^\ell  |\tilde{\nabla}|^i f \nr_{L^2_{\gamma+2s}}. \label{cancelf}
\end{align}
Here $m\ge 0$; when $s\in (0, 1/2)$ in \eqref{kernelQ} then $i=1$, otherwise 
$i = 2$.
\end{proposition}

\begin{proof}
We write out the relevant difference from \eqref{defTKL} into several terms
\begin{multline}
\label{differenceA}
M_\beta(v_*') w^{2\ell}(v')  f' -  M_\beta(v_*) w^{2\ell}(v) f 
=
w^{2\ell}(v) ~ f
~ \left( \frac{d \starCurveEP}{d \Ctheta}(0) \cdot \tilde{\nabla} M_\beta(v_*)  \right)
\\
+
\left( w^{2\ell}(v') ~ f' - w^{2\ell}(v) ~  f   \right)
~ \left( \frac{d \starCurveEP}{d \Ctheta}(0) \cdot \tilde{\nabla} M_\beta(v_*)  \right)
\\
+
w^{2\ell}(v') ~ f' ~ 
 \left( M_\beta(v_*')  - M_\beta(v_*) -  \frac{d \starCurveEP}{d \Ctheta}(0) \cdot \tilde{\nabla} M_\beta(v_*)    \right)
 \\
+
M_\beta(v_*) 
\left( w^{2\ell}(v') ~ f' - w^{2\ell}(v) ~  f  -  \frac{d \CurveEP}{d \Ctheta}(0) \cdot \tilde{\nabla} (w^{2\ell}  f )(v) \right)
 \\
 +
M_\beta(v_*) ~ \left(  \frac{d \CurveEP}{d \Ctheta}(0) \cdot \tilde{\nabla} (w^{2\ell}  f )(v)   \right)
 = \mathbb{I} + \mathbb{II}
 +
  \mathbb{III} + \mathbb{IV}+ \mathbb{V}.
\end{multline}
We split 
$
(\teePLUSop- \teeMINUSop)(g,h,f)
=
T^{\mathbb{I}}+ T^{\mathbb{II}}
+
T^{\mathbb{III}}+ T^{\mathbb{IV}}+ T^{\mathbb{V}},
$
where $T^{\mathbb{I}}$ corresponds to the first term in the splitting above, etc. Suppose initially that $s\in [1/2, 1)$.

We begin by considering the first and last terms, that is
\begin{equation}
\begin{split}
T^{\mathbb{I}}
=
\int_{\mathbb{R}^n} dv \int_{\mathbb{R}^n} dv_* \int_{\sph} d \sigma ~B_k ~ g_* h 
~ \left( \frac{d \starCurveEP}{d \Ctheta}(0) \cdot \tilde{\nabla} M_\beta(v_*)  \right)
~ w^{2\ell}(v) ~ f,
\\
T^{\mathbb{V}}
=
\int_{\mathbb{R}^n} dv \int_{\mathbb{R}^n} dv_* \int_{\sph} d \sigma ~B_k ~ g_* h 
~ M_\beta(v_*) ~ \left(  \frac{d \CurveEP}{d \Ctheta}(0) \cdot \tilde{\nabla} (w^{2\ell}  f )(v)   \right).
\end{split}
\label{firstLAST}
\end{equation}
We may estimate both of these terms in exactly the same way.  

First we define the extension.  If $f$ extends to $\mathbb{R}^{n+1}$, then 
\begin{equation}
(w^{2\ell} f) (v_1,\ldots,v_{n+1}) = w^{2\ell}(v) f(v_1,\ldots,v_{n+1}).
\label{extensionF}
\end{equation}
Here we think of $w^{2\ell}(v)$ as being constant in the $(n+1)$-st coordinate.
With this extension it follows that for $i=0,1,2$ we have
\begin{equation}
\label{nablaD}
|\tilde{\nabla}|^i (w^{2\ell}  f) 
\lesssim 
 w^{2\ell} |\tilde{\nabla}|^i f.
\end{equation}
We will use this basic estimate several times below.

For $T^{\mathbb{V}}$, notice that $\frac{d \CurveEP}{d \Ctheta}(0)$
is linear in $v'-v$ and has no other dependence on $v'$.  Thus the symmetry of $B_k$ with respect to $\sigma$ around the direction $\frac{v-v_*}{|v-v_*|}$ forces all components of $v'-v$ to vanish except the component in the symmetry direction.
Thus, one may replace $v-v'$ with $\frac{v-v_*}{|v-v_*|} \ang{v-v', \frac{v-v_*}{|v-v_*|}}$ in the expression for 
$\frac{d \CurveEP}{d \Ctheta}(0)$.  Since $\ang{v-v',v'-v_*} = 0$, the vector further reduces to $\frac{v-v_*}{|v-v_*|}\frac{|v-v'|^2}{|v-v_*|}$.  Hence
$$
\left| \frac{v-v_*}{|v-v_*|}\frac{|v-v'|^2}{|v-v_*|} \right| \leq 2^{-2 k} |v-v_*|^{-1}. 
$$
The last coordinate direction of $\frac{d \CurveEP}{d \Ctheta}(0)$ is given by $\ang{v,v'-v}$ which
 reduces to 
\begin{equation}
\left| \ang{v, \frac{v-v_*}{|v-v_*|}\frac{|v-v'|^2}{|v-v_*|}} \right| \lesssim 
 2^{-2k} \left( 1 + |v-v_*|^{-1} \right) \ang{v_*}. 
\label{symmetry1}
\end{equation}
These bounds allow us to employ similar methods to those employed in the previous section.  
To be precise, we must control the following integral
\begin{equation} 
2^{-2k} \int_{\mathbb{R}^n} dv \int_{\mathbb{R}^n} dv_* \int_{S^2} d \sigma ~  B_k  M_*^{1-\epsilon}  w^{2\ell}(v)~ |g_*| |h| ( |\tilde{\nabla}| f)
\left( 1 + |v-v_*|^{-1} \right).
\label{T5bound}
\end{equation}
Here we absorb any powers of $\ang{v_*}$  by $M_*^{-\epsilon}$ for any small $\epsilon >0$.
If $\gamma+2s -1 > - \frac{n}{2}$ or if $|v-v_*| \geq 1$, the estimate 
\eqref{cancelf}
for $T^{\mathbb{V}}$ then follows exactly as in 
\eqref{bjEST} 
with 
 \eqref{cauchySCHminus},
the only difference being the extra factor $|v-v_*|^{-1}$ giving the weight $w^\ell(v) \ang{v}^{\gamma+2s-1}$ on both $h$ and $|\tilde \nabla| f$ (away from the singularity, one need not fear destroying the local integrability of any of the singularities in 
\eqref{cauchySCHminus}).  When $|v-v_*| \leq 1$, 
the inequality \eqref{hlsineq} implies that the relevant portion of \eqref{T5bound} is bounded above by $2^{(2s-2)k} |g|_{L^2_{-m}} |h|_{L^2_{-m}} ||\tilde \nabla| f M^{\delta}|_{H^1}$.  Furthermore, $||\tilde \nabla| f M^{\delta}|_{H^1} \lesssim ||\tilde \nabla|^2 f|_{L^2_{-m}}$ since the isotropic derivatives from the $H^1$ 
space differ from the anisotropic derivatives $|\tilde \nabla|$ by at most a power of the velocity (which is controlled by $M^\delta$).
Thus $T^{\mathbb{V}}$ is controlled by the right-hand side \eqref{cancelf} with $i=2$ provided that $\gamma + 2s > -\frac{n}{2}$.

The estimation of $T^{\mathbb{I}}$ can be handled in exactly the same way because $\frac{d \starCurveEP}{d \Ctheta}(0)$ also depends linearly on $v-v'$ and has no other $v'$ dependence.  
Here we have exactly the same estimates for $\frac{d \starCurveEP}{d \Ctheta}(0)$ as just previously obtained for $\frac{d \CurveEP}{d \Ctheta}(0)$ by symmetry.  Then 
\eqref{cancelf} for $T^{\mathbb{I}}$ follows again exactly as for $T^{\mathbb V}$. 
 As in \eqref{nablaD}, we use
 \begin{equation}
\label{asin39}
|\tilde{\nabla}|^i M_\beta
\lesssim 
\sqrt{M},
\quad
i = 0,1,2,
\end{equation}
where the extension is defined as
$
M_\beta(v) = (2 \pi)^{-n/2}~p_\beta (v_1, \ldots, v_n) ~ e^{-v_{n+1}/2}.
$

We now turn to the estimation of the term $T^{\mathbb{III}}$, which can be written as 
\begin{multline}
T^{\mathbb{III}}
=
\int_{\mathbb{R}^n} dv \int_{\mathbb{R}^n} dv_* \int_{\sph} d \sigma ~B_k(v-v_*, \sigma) ~ g_* h
~ w^{2\ell}(v') ~ f' 
\\
\times
 \left( M_\beta(v_*')  - M_\beta(v_*) -  \frac{d \starCurveEP}{d \Ctheta}(0) \cdot \tilde{\nabla} M_\beta(v_*)    \right). 
\label{thirdTERM}
\end{multline}
We apply the starred analog of \eqref{paradiff2} to obtain
\begin{equation}
\left| T^{\mathbb{III}}  \right|
\lesssim
2^{-2k}
\int_{\mathbb{R}^n} dv \int_{\mathbb{R}^n} dv_* \int_{\sph} d \sigma ~ B_k  \left| g_* h \right|
 w^{2\ell}(v')  \left| f'  \right|
(M_* M'_*)^{\epsilon},
 \label{starred35}
\end{equation}
where we have just used the estimate
$
\sqrt{M(\starCurveEP(\Ctheta))} \le (M_* M_*')^{\epsilon},
$
(valid for all sufficiently small $\epsilon > 0$) which follows directly from $\ang{v_*} \lesssim \ang{\starCurveP(\Ctheta)}\lesssim \ang{v_*}$
since
$
v_* = \starCurveP(\Ctheta) + \Ctheta(v_*-v_*')
$
and $|v_* - v_*'| \le 1$.
At this point, the proof proceeds exactly as in Proposition \ref{referLATERprop} using the Cauchy-Schwartz estimate analogous to \eqref{tpe}, choosing $\ell' = 0$; note that the only difference is the presence of an additional $M_*^{\epsilon}$ which will give rapid decay of the weight applied to $g$.

It remains only to prove \eqref{cancelf} for $T^{\mathbb{II}}$ and $T^{\mathbb{IV}}$, both of which involve differences in $w^{2\ell} f$.  We use the difference estimates
\eqref{paradiff1} for $T^{\mathbb{II}}$
and
\eqref{paradiff2} for $T^{\mathbb{IV}}$ to obtain that
 for any fixed $\epsilon > 0$ both  $|T^{\mathbb{II}}|$ and $|T^{\mathbb{IV}}|$
 are controlled by
\begin{equation}
   2^{-2k} \int_0^1 d \Ctheta \int_{\mathbb{R}^n} dv \int_{\mathbb{R}^n} dv_* \int_{\sph} d \sigma ~ B_k ~ M_*^{1-\epsilon} ~ 
  w^{2\ell}(v)~ |g_* h| 
  ~ |\tilde{\nabla}|^2 f (\CurveP(\Ctheta) ).
  \label{controlT24}
\end{equation}
 The loss of $\epsilon$ comes from the factor $\ang{v_*}$ in \eqref{paradiff1} and  \eqref{paradiff2}
 in addition to the bound \eqref{asin39} and the estimate \eqref{derivEST}.  These also account for the $2^{-2k}$. Note, though, that the factor $2^{-2k}$ comes directly from \eqref{paradiff2}, but \eqref{paradiff1} only furnishes a factor of $2^{-k}$. In this case, there is an additional factor of $2^{-k}$ available in the estimate for $T^{\mathbb{II}}$ arising exactly from the derivative estimate \eqref{derivEST}.
Finally, note that
$\ang{v} \approx \ang{\CurveP(\Ctheta)}$ (which accounts for the replacement of $w^{2 \ell}(\CurveP(\Ctheta))$ by $w^{2 \ell}(v)$).

With that last estimate above and an application of Cauchy-Schwartz exactly as was done in \eqref{tpe}, it suffices to show
\begin{multline}
 \int_0^1 d \Ctheta \int_{\mathbb{R}^n} dv \int_{\mathbb{R}^n} dv_* \int_{\sph} d \sigma  
 ~  \frac{B_k ~  M_*^{1-\epsilon} ~ w^{2\ell}(\CurveP(\Ctheta))}{ |v-v_*|^{-\gamma-2s} \ang{\CurveP(\Ctheta)}^{\gamma+2s}  } ~  
 \left|  |\tilde{\nabla}|^2 f(\CurveP(\Ctheta))\right|^2  
 \\
 \lesssim 
 2^{2sk} \nsm w^\ell  |\tilde{\nabla}|^2 f \nsm_{L^2_{\gamma+2s}}^2. 
 \label{covterm}
\end{multline}
This uniform bound follows from the  change of variables $u = \CurveP(\Ctheta) = \Ctheta v' + (1-\Ctheta)v$, which is a transformation from  $v$ to $u$.  In view of the collisional variables \eqref{sigma}, we see  (with $\delta_{ij}$  the usual Kronecker delta) that 
$$
\frac{d u_i}{ dv_j} = (1-\Ctheta) \delta_{ij} + \Ctheta \frac{d v'_i}{ dv_j}
 = \left(1-\frac{\Ctheta}{2}\right)  \delta_{ij} + \frac{\Ctheta}{2} k_{j} \sigma_{i},
$$
with the unit vector $k = (v-v_*)/|v-v_*|$.  Thus the Jacobian is 
\begin{equation}
\left| \frac{d u_i}{ dv_j}  \right| 
= 
\left(1-\frac{\Ctheta}{2}\right)^2\left\{
\left(1-\frac{\Ctheta}{2}\right) + 
\frac{\Ctheta}{2} \ang{k, \sigma}
\right\}.
\label{jacobianCH}
\end{equation}
Since  $b(\ang{k, \sigma}) = 0$ when $\ang{k ,\sigma} \leq 0$ from \eqref{kernelQ}, and  $\Ctheta \in [0,1]$, 
it follows that the Jacobian is bounded from below on the support of the integral \eqref{covterm}.   But after this change of variable the old pole $k=(v-v_*)/|v-v_*|$ moves with the angle $\sigma$.  However  when one takes $\tilde k = (u-v_*)/|u-v_*|$, 
then
$
1- \ang{ k,\sigma } \approx  1 - \left< \right. \! \tilde{k}, \sigma \! \left. \right>, 
$
meaning that the angle to the pole is comparable to the angle to $\tilde{k}$ (which does not vary with $\sigma$).
Thus the estimate analogous to \eqref{bjEST} will continue to hold after the change of variables, which is used in the usual manner to give precisely the estimate in \eqref{covterm}.

It remains to prove \eqref{cancelf} for $s\in (0, 1/2)$.  
This estimate is exactly the same as the above except that the cancellation terms
$
\frac{d \CurveEP}{d \Ctheta}(0) \cdot \tilde{\nabla} (w^{2\ell}  f )(v)
$
and
$
\frac{d \starCurveEP}{d \Ctheta}(0) \cdot \tilde{\nabla} M_\beta(v_*)
$
can be removed from each of the expressions ${\mathbb I}$ through ${\mathbb V}$ in the splitting \eqref{differenceA} (leaving only the corresponding parts of $T^{\mathbb{III}}$ and $T^{\mathbb{IV}}$).  In this case we may use \eqref{paradiff1} instead of \eqref{paradiff2} which allows us to take $i =1$.
\end{proof}

This completes our proof of the cancellations for the $\sigma$ representation as in \eqref{defTKL}.  In the following, we estimate the cancellations on $h$ instead of putting them on $f$, and for this we use the Carleman representation as in \eqref{defTKLcarl}.

\begin{proposition}
As in the previous proposition, suppose $h$ is a Schwartz function on $\threed$ which is given by the restriction of some Schwartz function in $\R^{\last}$ to the paraboloid $(v, \frac{1}{2} |v|^2)$ and define $|\tilde{\nabla}|^i h$ analogously. For any $k \geq 0$, we have
\begin{align}
| ( \teePLUSop   - \teeSTARop)(g,h,f)|  & \lesssim 2^{(2s-i)k} \nsm g\nsm_{L^2_{-m}} 
\nl  w^\ell |\tilde{\nabla}|^i h \nr_{L^2_{\gamma+2s}}
\nsm w^\ell f\nsm_{L^2_{\gamma+2s}}.  
\label{cancelh}
\end{align}
Again $m\ge 0$; when $s\in (0, 1/2)$ in \eqref{kernelQ} then $i=1$ and for $s \in [1/2, 1)$  $i = 2$.
\end{proposition}

\begin{proof}
This proof follows the pattern that is now established.  The new feature in \eqref{cancelh} is that, from \eqref{defTKLcarl}, 
the pointwise differences to examine are
\begin{multline}
\label{differenceB}
M_\beta(v_*') h 
-  M_\beta(v_*) h' 
=
 \left( \frac{d \starCurveEP}{d \Ctheta}(0) \cdot \tilde{\nabla} M_\beta(v_*)    \right)
 h' 
\\
+
 \left( \frac{d \starCurveEP}{d \Ctheta}(0) \cdot \tilde{\nabla} M_\beta(v_*)    \right)
 \left( h 
 -   
 h' 
 \right)
\\
+
 \left( M_\beta(v_*')  - M_\beta(v_*) -  \frac{d \starCurveEP}{d \Ctheta}(0) \cdot \tilde{\nabla} M_\beta(v_*)    \right)
 h 
 \\
 +
 M_\beta(v_*) 
 \left( h 
 -   h' 
 -  \frac{d \CurveEP}{d \Ctheta}(1) \cdot \left(\tilde{\nabla} h \right) (\CurveEP(1))
 \right)
 \\
  +
 M_\beta(v_*) 
 \left(  \frac{d \CurveEP}{d \Ctheta}(1) \cdot \left(\tilde{\nabla} h \right)(\CurveEP(1))
 \right)
 = \mathbb{I} + \mathbb{II}
 +
  \mathbb{III} + \mathbb{IV}+ \mathbb{V}.
\end{multline}
We again 
split 
$
(\teePLUSop- \teeSTARop)(g,h,f)
=
T^{\mathbb{I}}_*+ T^{\mathbb{II}}_*
+
T^{\mathbb{III}}_*+ T^{\mathbb{IV}}_* + T^{\mathbb{V}}_*,
$
where $T^{\mathbb{I}}_*$ corresponds to the first term in the splitting above, etc.
For the last term $T^{\mathbb{V}}_*$, we have by symmetry that
\begin{equation}
T^{\mathbb{V}}_*
\eqdef
 \int_{\mathbb{R}^n} dv' ~ w^{2\ell}(v') ~
 \int_{\mathbb{R}^n} dv_*   \int_{E_{v_*}^{v'}} d \pi_v   ~ 
 \tilde{B}_k ~
 M_\beta(v_*)  g_* f' ~
  \frac{d \CurveEP}{d \Ctheta}(1) \cdot \tilde{\nabla} h' = 0.
  \label{zeroCORRECTOR}
\end{equation}
In this integral as $v$ varies on circles of constant distance to $v'$, the entire integrand is constant except for $\frac{d \CurveEP}{d \Ctheta}(1)$.  If we write $\frac{d \CurveEP}{d \Ctheta}(1)$ as a sum of two vectors, one lying in the span of the first $n$ directions and the second pointing in the last direction, it follows that we may replace the former vector by its projection onto the direction determined by $v' - v_*$.  But since the original vector points in the direction $v-v'$, the projection vanishes.  Since the last direction of $\frac{d \CurveEP}{d \Ctheta}(1)$  is exactly $\ang{v',v'-v}$, the corresponding integral of this over $v$ also vanishes by symmetry.  
Similarly, the first term $T^{\mathbb{I}}_*$ also vanishes,
\begin{gather*}
T^{\mathbb{I}}_*
\eqdef
 \int_{\mathbb{R}^n} dv'  w^{2\ell}(v')  \int_{\mathbb{R}^n} dv_*   \int_{E_{v_*}^{v'}} d \pi_v    
 \tilde{B}_k  M_\beta(v_*)  g_* f'  h' 
    \left( \frac{d \starCurveEP}{d \Ctheta}(0) \cdot \tilde{\nabla} M_\beta(v_*)    \right)
   = 0.
\end{gather*}
Here the explanation is the same as in the previous case.

The remaining terms incorporate cancellations.  
In terms of $\mathbb{II}$ and $\mathbb{IV}$, our operator  from \eqref{defTKLcarl} takes the form
\begin{equation}
T^{\mathbb{II}}_*
+
T^{\mathbb{IV}}_* = \int_{\mathbb{R}^n} dv' ~ w^{2\ell}(v') ~\int_{\mathbb{R}^n} dv_* \int_{E_{v_*}^{v'}} d \pi_v  
~ \tilde{B}_k ~ g_*  f'   ~\left( \mathbb{II}~ + \mathbb{IV} \right). 
\label{form24}
\end{equation}
Applying \eqref{paradiff1}, \eqref{derivEST}, and \eqref{paradiff2} gives the estimate:
 \begin{equation}
\label{24est}
| \mathbb{II} | + | \mathbb{IV} |  \lesssim   
|v-v'|^2
 M_*^{1/2} \int_0^1 d \Ctheta ~  |\tilde{\nabla}|^2 h(\CurveEP(\Ctheta)).  
\end{equation}
Again,  $|v-v'| \lesssim 2^{-k}$. 
With all of that we may  estimate the terms $\left| T^{\mathbb{II}}_* \right|$
and
$\left| T^{\mathbb{IV}}_* \right|$
 above by the following single term:
\begin{equation}
2^{-2k}
\int_0^1 d \Ctheta \int_{\mathbb{R}^n} dv' \int_{\mathbb{R}^n} dv_* \int_{E_{v_*}^{v'}}  d \pi_{v}   
~ \tilde{B}_k ~
  w^{2\ell}(v') ~
  M_*^{1/2} |g_* f'| ~ |\tilde{\nabla}|^2 h(\CurveP(\Ctheta)).
  \label{singleTERMest}
\end{equation}
 The estimates required now for the term above are completely analogous to a corresponding estimate from Proposition \ref{referLATERprop}.   First we change back to the $\sigma$-representation.
At that point, we can use the corresponding Cauchy-Schwartz argument \eqref{tpe2} (and, since $\ang{v} \approx \ang{v'}$ when $k\ge 0$, the assumption $\gamma+2s \geq 0$ is not even necessary in this case for the estimation to proceed).  Using the same change of variables employed for \eqref{covterm} (with Jacobian given by \eqref{jacobianCH}), the estimate proceeds along the usual lines.

Regarding the estimate for $T^{\mathbb{III}}_* $, we notice that the estimate of \eqref{paradiff2}
holds if $v, v'$ are replaced by $v_*, v_*'$ and we replace $\CurveEP$ with $\starCurveEP$.
Using \eqref{paradiff2} in this case, we have
\begin{equation}
\left| T^{\mathbb{III}}_* \right|
\lesssim
2^{-2k}\int_0^1 d \theta \int_{\mathbb{R}^n} dv' \int_{\mathbb{R}^n} dv_* \int_{E_{v_*}^{v'}}  d \pi_{v} 
~  \tilde{B}_k ~   w^{2\ell}(v') ~ M_*^{1/2} |g_* f'| ~ | h |.
\label{T3est}
\end{equation}
Now this estimate can be handled as in the previous case for $T^{\mathbb{II}}_* $. This case is easier because there is no $\CurveP(\theta)$ here, which required the use of the change-of-variables used in \eqref{covterm} as well.

Notice that the estimates above hold for any $s\in (0, 1)$, but we obtain $i=2$ in each case 
in \eqref{cancelh}.
Yet recall that the terms $T^{\mathbb{I}}_*$ and $T^{\mathbb{V}}_*$ vanish by symmetry.  
Thus the estimate with $i=1$ when $s\in (0, 1/2)$ can be proved following the same procedure as above using \eqref{paradiff1} instead of \eqref{paradiff2} in each case after removing the cancellation terms 
$
\frac{d \starCurveEP}{d \Ctheta}(0) \cdot \tilde{\nabla} M_\beta(v_*),   
$
and
$
\frac{d \CurveEP}{d \Ctheta}(1) \cdot \left(\tilde{\nabla} h \right) (\CurveEP(1))
$
from this analysis.
\end{proof}

\section{Derivative estimates for soft potentials: $-2s> \gamma > -\dim$}
\label{sec2:physicalDECrel}

The estimates from Section \ref{physicalDECrel} apply under, for example, the hard potential hypothesis \eqref{kernelP} (or, more generally, $\gamma +2s > -\frac{n}{2}$ combined with $\gamma > -\dim$).  
In this section, we want to prove estimates under more general assumptions, including for the very singular soft potentials \eqref{kernelPsing}.  To do this, we use derivatives in the upper bounds.  Note that all the estimates in this section will apply under both \eqref{kernelP} and \eqref{kernelPsing}.  In Section \ref{sec:SSE} we estimate each of  $\teePLUSop$, $\teeMINUSop$, and $\teeSTARop$ using only the constraints on the size and support of $B_k$.  Then in Section \ref{sec:SSChigh} we estimate the collision operator by exploiting the cancellation properties of $\Gamma_\beta$.
Finally,  we prove the ``compact estimates'' in Section \ref{sec:SSCE}; these are used to prove the constructive lower bound for the linearized collision operator in Theorem \ref{lowerN}.

\subsection{``Trivial'' estimates of the decomposed pieces with derivatives}
\label{sec:SSE}
We will now prove several size and support estimates for the decomposed pieces of the Boltzmann collision operator.
It will be useful to let $\phi(v)$ denote an arbitrary smooth function which satisfies for some positive constants $C_{\phi}$ and $c$ that
\begin{equation}
\left| 
\phi(v)
\right|
\le C_\phi e^{- c |v|^2}.
\label{rapidDECAYfcn}
\end{equation}
We use generic functions satisfying  \eqref{rapidDECAYfcn} often in what follows.  

\begin{proposition}\label{prop11}
For any integer $k$, any $m \ge 0$ and $\ell \in \mathbb{R}$, we have
\begin{gather}
 \left| \teeMINUSop(g,h,f) \right|   \lesssim 2^{2sk} \nsm g\nsm_{H^{\ksob}_{-m}} \nsm w^\ell h\nsm_{L^2_{\gamma + 2s}} 
 \nsm w^\ell f\nsm_{L^2_{\gamma + 2s}},
 \label{tminusg} 
 \\
 \left| \teeMINUSop(g,h,f) \right|   \lesssim 2^{2sk} 
 \nsm g\nsm_{L^2_{-m}}  
| h |_{H^{\ksob}_{\ell,\gamma + 2s}} 
| w^\ell f |_{L^2_{\gamma + 2s}}.
 \label{tminush}
\end{gather}
Furthermore, for $\phi$ 
defined as in \eqref{rapidDECAYfcn}, we have
\begin{align}
 \left| \teeMINUSop(g,\phi,f) \right| 
 +
  \left| \teeMINUSop(g,f,\phi) \right| 
 & 
 \lesssim 
  C_\phi ~ 
 2^{2sk} ~ 
 \nsm g\nsm_{L^2_{-m}}  \nsm f\nsm_{L^2_{-m}}.
 \label{tminushRAP}
\end{align}
These estimates hold uniformly.  
\end{proposition}

We record here the following Sobolev embedding theorem, with $\ksob = \lfloor \frac{n}{2} +1 \rfloor$,
\begin{equation}
|w^{-m-\ksob} f|_{L^\infty}  \lesssim |f|_{H^{\ksob}_{-m}},
\label{sobolevE}
\end{equation}
which holds for any $m\in \mathbb{R}$.
We will also use the following immediate implications
\begin{equation}
|w^{-m} f|_{L^\infty} + |w^\ell f|_{L^2_{\gamma+2s}} \lesssim |f|_{H^{\ksob}_{\ell,\gamma+2s}},
\quad
|w^{-m} f|_{L^\infty} + |w^\ell f|_{L^2} \lesssim |f|_{H^{\ksob}_{\ell}},
\notag
\end{equation}
where $m$ is sufficiently large, depending on $\ell\in\mathbb{R}$, $\gamma+2s$, and $n\ge 2$.  
Furthermore, we are using the abbreviated notation $L^\infty = L^\infty(\mathbb{R}^n)$.

\begin{proof}
Restricting to the region where $|v-v_*| \geq 1$, the same argument in Proposition \ref{prop1} leading to \eqref{cauchySCHminus} is still valid.  Thus, to establish \eqref{tminusg} and \eqref{tminush}, it suffices to restrict attention to the region where $|v-v_*| \leq 1$.  In this case, $\sqrt{M_*} w^{2 \ell}(v) \lesssim (M_* M)^{\delta}$ for some $\delta>0$.  Thus, after applying \eqref{bjEST}, it suffices to assume that 
\[ 
\left| \teeMINUSop(g,h,f) \right| \lesssim 2^{2sk} \int_{\mathbb{R}^n} dv \int_{\mathbb{R}^n} dv_* ~  |v-v_*|^{\gamma+2s} |g_* h f|~ (M_* M)^{\delta}.
 \]
 To obtain 
 \eqref{tminusg}, we
first take the $L^\infty$ norm of $g_* M_*^{\delta/2}$ and use \eqref{sobolevE}, then apply the elementary inequality
 $
\int_{\mathbb{R}^n} dv_* ~ M_*^{\delta/2} ~  |v-v_*|^{\gamma+2s}
\lesssim
\langle v \rangle^{\gamma + 2s}
$, and use Cauchy-Schwartz putting $h$ in one term, $f$ in the other, and the square root of the weight $ M^\delta \ang{v}^{\gamma+2s}$ in both.  To obtain \eqref{tminush}, on the other hand, we may take the $L^\infty$ norm of $|h| \ang{v}^{-m-\ksob}$, use \eqref{sobolevE}, and apply Cauchy-Schwartz again.

For \eqref{tminushRAP} notice that if either of the second two functions in $\teeMINUSop$  has rapid decay such as \eqref{rapidDECAYfcn}, then we have rapid decay in both $v$ and $v_*$ simultaneously.  Thus we can use Cauchy-Schwartz, putting $g$ in one term, $f$ in the other term, and spreading the rapid decay across both terms.  
\end{proof}

\begin{proposition}\label{starPROP}
For all $\ell\in\mathbb{R}$, $m\ge 0$ and integers $k$, we have
\begin{gather}
  \left| \teeSTARop(g,h,f) \right|   \lesssim 
  2^{2sk}    \nsm g\nsm_{H^{\ksob}_{-m}} 
  \nsm w^\ell h\nsm_{L^2_{\gamma + 2s}} 
 \nsm w^\ell f\nsm_{L^2_{\gamma + 2s}},
 \label{tstarg} 
 \\
  \left| \teeSTARop(g,h,f) \right|   \lesssim 2^{2sk} 
 \nsm g\nsm_{L^2_{-m}} 
| w^\ell h |_{H^{\ksob}_{\ell,\gamma + 2s}} 
| w^\ell f |_{L^2_{\gamma + 2s}}.
 \label{tstarh}
\end{gather}
These inequalities are uniform.  Moreover,
\begin{align}
 \left| \teeSTARop(g,\phi,f) \right| 
 +
  \left| \teeSTARop(g,f,\phi) \right| 
 &  \lesssim 2^{2sk} 
 \nsm g\nsm_{L^2_{-m}}  \nsm f\nsm_{L^2_{-m}},
 \label{tstarhRAP}
\end{align}
where as usual $\phi$ is defined as in \eqref{rapidDECAYfcn}.  
\end{proposition}

\begin{proof}
Recall \eqref{bjCARLest}.  With that we obtain \eqref{AAeqREF}.  Now the relevant estimates can be established as in Proposition \ref{prop11}.
\end{proof}

\begin{proposition}
\label{referLATERprop2}
Fix an integer $k$ and $\ell^+, \ell^- \ge 0$,
with $\ell = \ell^+ - \ell^-$.  For any  $0\le \ell' \le \ell^-$,
with $\ell + \ell' = \ell^+ - (\ell^- - \ell')$,
  we   have the uniform estimates 
\begin{equation} 
\begin{split}
\left|  \teePLUSop(g,h,f)  \right| \lesssim   & ~
2^{2sk} 
 \nsm  g\nsm_{H^{\ksob}_{{\ell^+} - \ell'}} 
\nsm w^{\ell + \ell'} h\nsm_{L^2_{\gamma + 2s}} 
\nsm w^{\ell} f\nsm_{L^2_{\gamma + 2s}}.
\end{split}
 \label{tplussmallN}
\end{equation}
We also have a similar 
estimate with the roles of $g$ and $h$ reversed
\begin{equation} 
\begin{split}
\left|  \teePLUSop(g,h,f)  \right| \lesssim   & ~
2^{2sk} 
\nsm w^{\ell^+ - \ell'} g\nsm_{L^2} 
 \nsm  h\nsm_{H^{\ksob}_{\ell + \ell', \gamma + 2s} }
\nsm w^{\ell} f\nsm_{L^2_{\gamma + 2s}}.
\end{split}
 \label{tplussmall2N}
\end{equation}
\end{proposition}

\begin{proof}
We plan to estimate $\teePLUSop(g,h,f)$ in \eqref{defTKL} from above as in \eqref{tBabove}.
On the region $|v-v_*| \geq 1$, the singularity of $|v-v_*|^{\gamma+2s}$ is avoided; consequently performing the same Cauchy-Schwartz argument leading to \eqref{tpe} establishes the estimate analogous \eqref{tplussmallK} on this region.  
This leaves only the region $|v-v_*| \leq 1$, where $w^{2 \ell}(v') \sqrt{M_*'} \lesssim (M_* M)^{\delta}$ for some $\delta > 0$, so the estimate is reduced to the point where we may assume
$$
\left| \teePLUSop(g,h,f)  \right| \lesssim \int_{\mathbb{R}^n} dv \int_{\mathbb{R}^n} dv_* \int_{\sph} d \sigma  ~B_k(v-v_*, \sigma) ~  \left| g_* h \right| \left| f' \right|(M_* M M' M_*')^{\delta/2}.
$$
In this case, the correct application of Cauchy-Schwartz gives the upper bound
\begin{multline*}
\left(  \int_{\mathbb{R}^n} dv \int_{\mathbb{R}^n} dv_* \int_{\sph} d \sigma  ~B_k(v-v_*, \sigma) ~  \left| g_* h \right|^2 (M_* M)^{2\delta} \right) \\ 
\times 
\left( \int_{\mathbb{R}^n} dv \int_{\mathbb{R}^n} dv_* \int_{\sph} d \sigma  ~B_k(v-v_*, \sigma) ~  \left| f' \right|^2 (M' M_*')^{2\delta} \right).
\end{multline*}
Now the second factor is clearly controlled by $2^{sk} |f|_{L^2_{-m}}$ for any $m$.  In the first term, estimating $g M^{\delta}$ by the norm $H^{\ksob}_{-m}$ gives that the entire factor is controlled by $2^{sk} |g|_{H^{\ksob}_{-m}}|h|_{L^2_{-m}}$, establishing \eqref{tplussmallN}.  As for \eqref{tplussmall2N}, it is achieved in the same manner by estimating $h M^{\delta}$ in terms of $|h|_{H^{\ksob}_{-m}}$.
\end{proof}

\begin{proposition}\label{opGstarESTderiv}
We have the following uniform estimate for \eqref{opGlabel}:
\begin{gather}
\left| \LopGstar (g,h,f)   \right|  \lesssim
      \nsm g\nsm_{H^{\ksob}_{-m}} 
  \nsm w^\ell h\nsm_{L^2_{\gamma}} 
 \nsm w^\ell f\nsm_{L^2_{\gamma}},
 \label{tstarCf} 
 \\
  \left| \LopGstar(g,h,f) \right|   \lesssim  
 \nsm g\nsm_{L^2_{-m}} 
|  h |_{H^{\ksob}_{\ell,\gamma}} | w^\ell f |_{L^2_{\gamma}}.
 \label{tstarCh}
\end{gather}
These inequalities hold for all $\ell\in\mathbb{R}$, $m\ge 0$, and $\gamma > -\dim$.
\end{proposition}

\begin{proof}  
Given the definition of $\LopGstar$ in \eqref{opGlabel}, Proposition \ref{opGstarEST} establishes in \eqref{hls0} that
\[ \left| \LopGstar(g,h,f) \right|   \lesssim \int_{\R^n} dv_* \int_{\R^n} dv' |g_*| |f'| |h'| w^{2 \ell}(v') |v' - v_*|^{\gamma} \sqrt{M_*}. \]
The methods of Proposition \ref{prop11} apply directly here and yield Proposition \ref{opGstarESTderiv}.
\end{proof}

This concludes our  size and support estimates.  In the next sub-section we will prove estimates which incorporate the essential cancellation properties of the Boltzmann collision operator in the appropriate geometric framework.

\subsection{Cancellations with soft potentials: $-2s > \gamma > -n$}
\label{sec:SSChigh}
We recall the notation from Section \ref{sec:cancelHARD}.  
The estimates in this section apply under either \eqref{kernelP} or \eqref{kernelPsing}.

\begin{proposition}
\label{cancelFprop}
Suppose $f$ is a Schwartz function on $\mathbb{R}^n$ given by the restriction of some Schwartz function $F$ 
on $\mathbb{R}^{n+1}$ to the paraboloid $(v, \frac{1}{2} |v|^2)$.  Let $|\tilde{\nabla}|^i f$ 
be the restriction of $|\tilde{\nabla}|^i F$ to the same paraboloid $(i=1,2)$.  Then, for any $k \geq 0$,
\begin{gather}
\label{cancelf21}
 \left| (\teePLUSop- \teeMINUSop)(g,h,f) \right|   
 \lesssim 2^{(2s-i)k} 
\nsm g\nsm_{L^2_{-m}}
\nsm  h\nsm_{H^{\ksob}_{\ell,\gamma+2s}}
\nl w^\ell |\tilde{\nabla}|^i f \nr_{L^2_{\gamma+2s}},
\\
 \left| (\teePLUSop- \teeMINUSop)(g,h,f) \right|   
 \lesssim 2^{(2s-i)k} 
\nsm  g\nsm_{H^{\ksob}_{-m}}
  \nsm w^\ell h\nsm_{L^2_{\gamma+2s}}
 \nl w^\ell |\tilde{\nabla}|^i f \nr_{L^2_{\gamma+2s}}. 
\label{cancelf2}
\end{gather}
Each of these inequalities hold for any $m \ge 0$ and any $\ell \in \mathbb{R}$.    Here when $s\in (0, 1/2)$ in \eqref{kernelQ} then $i=1$ and when $s \in [1/2, 1)$  we have $i = 2$.
\end{proposition}

\begin{proof}
This estimate will follow the proof of \eqref{cancelf}.
We expand difference into several terms as in \eqref{differenceA}.  We split 
$
(\teePLUSop- \teeMINUSop)(g,h,f)
=
T^{\mathbb{I}}+ T^{\mathbb{II}}
+
T^{\mathbb{III}}+ T^{\mathbb{IV}}+ T^{\mathbb{V}},
$
where $T^{\mathbb{I}}$ corresponds to the first term in the splitting above, etc. Suppose initially that $s\in [1/2, 1)$.
We begin by considering the first and last terms.   
Just as before, a symmetry argument establishes the bound \eqref{T5bound} for $|T^{\mathbb I}| + |T^{\mathbb V}|$ (where in the case of $T^{\mathbb I}$ we use the estimate \eqref{derivEST} to get a full factor $2^{-2k}$).  To establish \eqref{cancelf21} and \eqref{cancelf2} for these terms, one uses Sobolev embedding exactly as was done for \eqref{tminush} and \eqref{tminusg}, respectively.  

We now turn to the estimation of the term $T^{\mathbb{III}}$, which can be written as \eqref{thirdTERM}.
With the starred analog of \eqref{paradiff2} we obtain the estimate \eqref{starred35}.
At this point, \eqref{starred35} is estimated exactly as in the proof of \eqref{tplussmallN} and \eqref{tplussmall2N}. 

It remains to only prove \eqref{cancelf21}
and
\eqref{cancelf2} for $T^{\mathbb{II}}$ and $T^{\mathbb{IV}}$, both of which involve differences in $w^{2\ell} f$.  We use the difference estimates
\eqref{paradiff1} for $T^{\mathbb{II}}$
and
\eqref{paradiff2} for $T^{\mathbb{IV}}$ to obtain that
  both  $|T^{\mathbb{II}}|$ and $|T^{\mathbb{IV}}|$
 are controlled by \eqref{controlT24} (for some fixed $\epsilon > 0$).  With \eqref{controlT24} and Cauchy-Schwartz, it suffices to use the estimate corresponding to 
 \eqref{tpe} (omitting the extra factors $|v-v_*|^{\pm(\gamma+2s)} \ang{v'}^{\mp (\gamma+2s)}$); after a pre-post change of variables, it suffices to show
\begin{multline}
 \left( \int_0^1 d \Ctheta \int_{\mathbb{R}^n} dv \int_{\mathbb{R}^n} dv_* \int_{\sph} d \sigma  B_k M_*^{1-\epsilon} 
 w^{2\ell}(\CurveP(\Ctheta))
 \left|  |\tilde{\nabla}|^2 f(\CurveP(\Ctheta))\right|^2 \right)^\frac{1}{2}  
 \\
 \lesssim 
 2^{sk} \nsm w^\ell  |\tilde{\nabla}|^2 f \nsm_{L^2_{\gamma+2s}}. 
 \label{covterm2}
\end{multline}
This uniform bound follows from the  change of variables $u = \Ctheta v' + (1-\Ctheta)v$, which is a transformation from  $v$ to $u$ with uniformly positive Jacobian \eqref{jacobianCH}.  See the discussion surrounding \eqref{covterm} for the full details.

To prove \eqref{cancelf21} and \eqref{cancelf2} for $s\in (0, 1/2)$, we use
 exactly the same estimates as above except that the terms
$
\frac{d \CurveEP}{d \Ctheta}(0) \cdot \tilde{\nabla} (w^{2\ell}  f )(v)
$
and
$
\frac{d \starCurveEP}{d \Ctheta}(0) \cdot \tilde{\nabla} M_\beta(v_*)
$
are unnecessary and we can use \eqref{paradiff1} instead of \eqref{paradiff2} which allows us to take $i =1$.
\end{proof}


\begin{proposition}
\label{cancelHprop}
As in the prior proposition, suppose $h$ is a Schwartz function on $\mathbb{R}^n$ which is given by the restriction of some Schwartz function in $\mathbb{R}^{n+1}$ to the paraboloid $(v, \frac{1}{2} |v|^2)$ and define $|\tilde{\nabla}|^i h$ analogously for $i=1,2$. 
We have 
\begin{gather}
\label{cancelh2g1}
 \left| (\teePLUSop- \teeSTARop)(g,h,f) \right| 
   \lesssim 2^{(2s-i)k} 
    \nsm g\nsm_{L^2_{-m}} 
\nl  |\tilde{\nabla}|^i h \nr_{H^{\ksob}_{\ell,\gamma+2s}} 
  \nsm w^\ell f\nsm_{L^2_{\gamma+2s}},
\\
 \left| (\teePLUSop- \teeSTARop)(g,h,f) \right| 
   \lesssim 2^{(2s-i)k} 
\nsm  g\nsm_{H^{\ksob}_{-m}}
  \nl w^\ell |\tilde{\nabla}|^i h \nr_{L^2_{\gamma+2s}}
  \nsm w^\ell f\nsm_{L^2_{\gamma+2s}}.
  \label{cancelh2g2}
\end{gather}
The above inequalities hold uniformly in $k \geq 0$,  $m\ge 0$, and $\ell \in \mathbb{R}$.  
Again when $s\in (0, 1/2)$ in \eqref{kernelQ} then $i=1$ and when $s \in [1/2, 1)$  we have $i = 2$.
\end{proposition}

\begin{proof}
We recall the splitting \eqref{differenceB} used to establish \eqref{cancelh} (and we begin with the case $s \geq \frac{1}{2}$).  
For the last term and the first term, we have $T^{\mathbb{V}}_* = T^{\mathbb{I}}_* =0$ by symmetry.  The explanation for both of these is the same as in \eqref{zeroCORRECTOR}.
In terms of $\mathbb{II}$ and $\mathbb{IV}$, the operators take the form \eqref{form24} and are uniformly bounded above by a constant times \eqref{singleTERMest}.  By following the same procedures used for \eqref{tplussmallN} and \eqref{tplussmall2N}, we arrive at the desired inequality.  In the process, we change to the $\sigma$-representation and use the same change-of-variables from \eqref{covterm} to establish the desired estimates.  Regarding $T^{\mathbb{III}}_* $ by \eqref{T3est}, the estimates follow in even closer analogy to \eqref{tplussmallN} and \eqref{tplussmall2N} since
there is no $\CurveP(\theta)$ here, which required the change of variables.

Finally, note that the situation when $s \leq \frac{1}{2}$ follows with $i=1$ as quickly as usual once the gradient terms are removed from the splitting \eqref{differenceB}.
\end{proof}

This concludes our cancellation estimates for the differences involving three arbitrary smooth functions.  We will also need cancellation estimates when we have a more specific smooth function satisfying
the following estimate
\begin{equation}
\label{derivESTa}
|\tilde{\nabla}|^2 \phi \le C_\phi e^{-c|v|^2},
\quad
C_\phi \ge 0, \quad 
c> 0.
\end{equation}
Above $\phi$ is any  Schwartz function on $\mathbb{R}^n$ which is given by the restriction of some Schwartz function in $\mathbb{R}^{n+1}$ to the paraboloid $(v, \frac{1}{2} |v|^2)$ and $|\tilde{\nabla}|^2 \phi$ is defined analogously as usual. 
With this in mind, we have the next estimates:

\begin{proposition}  
\label{compactCANCELe}
As in the prior propositions, 
with \eqref{derivESTa},
for any $k \geq 0$ we have
\begin{gather}
\notag
 \left| (\teePLUSop- \teeMINUSop )(g,h,\phi) \right| 
 +
 \left| (\teePLUSop- \teeSTARop)(g,\phi,h) \right| 
   \lesssim 
   C_\phi ~ 2^{(2s-2)k} 
    \nsm g\nsm_{L^2_{-m}}     \nsm h\nsm_{L^2_{-m}} .
\end{gather}
The above inequalities hold uniformly for any $m\ge 0$ and $\ell \in \mathbb{R}$.
\end{proposition}

Notice that the proof of this Proposition \ref{compactCANCELe} follows exactly the proofs of the previous two Propositions \ref{cancelFprop} and \ref{cancelHprop}.  The main difference now is that when going through the last two proofs above, as a result of \eqref{derivESTa}, we will always in every estimate have strong exponential decay in both variables $v$ and $v_*$.  This strong decay allows us to easily obtain Proposition \ref{compactCANCELe}
using exactly the techniques developed in this section and the last one.  We omit repeating these details again.

\subsection{Compact estimates}  
\label{sec:SSCE}
In this sub-section we prove several useful estimates for the ``compact part'' of the linearized collision operator \eqref{compactpiece}.  Here we use the integer index $j$ instead of  $k$ to contrast with  the kernel $\kappa$ below.
Our first step is to notice that
the Carleman representation \eqref{defTKLcarl}
 of $T^{j,\ell}_{+}(g,M_{\beta_1},f)$ grants
$$
T^{j,\ell}_{+}(g,M_{\beta_1},f) = 
\int_{\mathbb{R}^n} dv' ~ w^{2\ell}(v')~ f(v') ~ 
\int_{\mathbb{R}^n} dv_*  ~ g(v_*) ~ \kappa_j^{\gamma+2s}(v',v_*),
$$
where for any multi-indices $\beta$ and $\beta_1$ we have the kernel
$$
\kappa_j^{\gamma+2s}(v',v_*)
\eqdef
\int_{E_{v_*}^{v'}} d \pi_{v}  ~
\tilde{B}_j ~   M_{\beta}(v_*') M_{\beta_1}(v). 
$$
Recall the domain 
  $
  E^{v'}_{v_*} = \left\{ v\in \threed : \ang{v_*-v', v - v'} =0 \right\},
  $
  and that $d\pi_{v} $ denotes the Lebesgue measure on this hyperplane.  More generally,  
  suppose that 
  $$
\kappa_{j}^{\psi, \phi}(v',v_*)
\eqdef
\int_{E_{v_*}^{v'}} d \pi_{v}  ~
\tilde{B}_j
~   \psi(v_*') \phi(v),
$$
where 
$\phi$ satisfies 
$\left| \phi(v) \right| \le C_\phi e^{-c |v|^2}$, and similarly
$\left| \psi(v) \right| \le C_\psi e^{-c |v|^2}$
 for any positive constants
$C_\phi$, $C_\psi$, and $c$ as in \eqref{rapidDECAYfcn}.
Then we have the main compact estimate:

 \begin{lemma}\label{prop:Grad}
For any $\ell \in \R$, $\rho \ge 0$,  $\kappa_{j}^{\psi, \phi}$ satisfies the uniform in $j \le 0$ estimate
 $$
  \forall \, v_* \in \R^n, \quad \int_{\R^n} dv'  ~ \left| \kappa_{j}^{\psi, \phi}(v',v_*) \right| \, 
  w^\ell(v')
   \lesssim
       C_\psi C_\phi ~  2^{2sj}\ang{v_*}^{-\rho-(n-1)}
         w^\ell(v_*).
 $$
Furthermore, 
the same estimate holds if the variables are reversed
 $$
  \forall \, v' \in \R^n, \quad \int_{\R^n} dv_*  ~ \left| \kappa_{j}^{\psi, \phi}(v',v_*) \right| \, w^\ell(v_*)
   \lesssim
       C_\psi C_\phi ~  2^{2sj}\ang{v'}^{-\rho-(n-1)}
       w^\ell(v').
 $$
 \end{lemma}

The crucial gain of the weight $-(n-1)$ has been known in the cut-off regime; see Grad \cite{MR0156656}, and also \cite{MR2322149}.  Here, in the regime $j \leq 0$, we further gain an arbitrarily large exponent without angular cut-off using our decomposition of the singularity.

\begin{proof}[Proof of Lemma~\ref{prop:Grad}]
The first upper bound follows easily from the second after noticing that the roles of $v'$ and $v_*$ can generally be reversed in the upper bound for 
$ \left| \kappa_{j}^{\psi, \phi}(v',v_*) \right|$ given below in \eqref{symmetricUPPER}.
We begin by observing that
\[ |v|^2 + |v_*'|^2 = \frac{| v + v_*'|^2}{2} + \frac{|v-v_*'|^2}{2} = 2 \left| v + \frac{v_* - v'}{2} \right|^2 + \frac{|v'-v_*|^2}{2}. \]
This uses that $v_*' = v + v_* - v'$.  Furthermore
$$
\left| \kappa_{j}^{\psi, \phi}(v',v_*) \right|
\lesssim C_\psi C_\phi 
\int_{E_{v_*}^{v'}} d\pi_{v} 
\tilde B_j ~ e^{-2c \left| v + \frac{v_* - v'}{2} \right|^2 - c\frac{|v'-v_*|^2}{2}}.
$$
Recall the estimate $\tilde B_j \lesssim |v-v'|^{-(n-1)-2s} |v' - v_*|^{\gamma+2s} \chi_j(|v-v'|)$.  Now for any  $\rho \ge 0$, we have $1 \lesssim |v'-v_*|^{\rho} |v - v'|^{-\rho}$ since $|v - v'|\ge 1$ when $j\le 0$ and $|v - v'| \leq |v' - v_*|$ on the support of $b(\cos\theta)$.  These conclusions were deduced in the proof of Proposition \ref{propSTAR}.  Thus it follows that 
  $
  \tilde B_j \lesssim 2^{2sj} |v-v'|^{-(n-1)-\rho} |v' - v_*|^{\gamma + 2s +\rho} \chi_j(|v-v'|)
  $ 
  on the support of the integral, and we quickly arrive at the
 upper bound
\[ C_\psi C_\phi 2^{2sj} |v'-v_*|^{\gamma+2s + \rho} e^{- c\frac{|v'-v_*|^2}{2}}
\int_{E_{v_*}^{v'}} d\pi_{v} 
 \frac{\chi_j(|v-v'|)}{|v-v'|^{\rho + n-1}} e^{-2c \left| v + \frac{v_* - v'}{2} \right|^2} \]
for $\left| \kappa_{j}^{\psi, \phi}(v',v_*) \right|$ (modulo the usual constants).  Letting $v = v' + z$ for $z \in E^0_{v'-v_*}$, 
where
$
E_{v'-v_*}^{0}\eqdef \{z\in\threed: \ang{v'-v_*,z}=0\}
$
we have that the integrand must equal
\[ \int_{E^0_{v'-v_*}} d\pi_{z} 
 \frac{\chi_j(|z|)}{|z|^{\rho + n-1}} e^{-2c \left| z + w_1 \right|^2 - 2 c |w_2|^2}
 \lesssim
 \ang{w_1}^{-\rho-(n-1)} e^{-2c |w_2|^2}
  \]
for some vectors $w_1, w_2$ such that $\ang{w_1,w_2} = 0$, $w_1 \in E^0_{v'-v_*}$, and $w_1 + w_2 = \frac{v'+v_*}{2}$. Here we use $j \leq 0$.
Collecting these estimates gives that
\[ \left| \kappa_{j}^{\psi, \phi}(v',v_*) \right| \lesssim C_\psi C_\phi 2^{2sj} |v'-v_*|^{\gamma+2s+\rho} \ang{w_1}^{-\rho-(n-1)} e^{-2c |w_2|^2} e^{- c\frac{|v'-v_*|^2}{2}}.\]
We furthermore have the inequality  $\ang{w_1}^{-\rho-(n-1)} e^{-2c |w_2|^2} \lesssim \ang{ v' + v_*}^{-\rho -(n-1)}$ for any fixed $\rho$, as well as the fact that $|v' -v_*|^{\gamma+2s + \rho} \lesssim \exp( \frac{c}{4} |v' - v_*|^{2})$ since $j\le0$.  Thus we may arrive at the final conclusion
\begin{equation}
 \left| \kappa_{j}^{\psi, \phi}(v',v_*) \right| \lesssim 2^{2sj} \ang{v' + v_*}^{- \rho - (n-1)} e^{ - \frac{c}{4} |v' - v_*|^2}. \label{symmetricUPPER}
\end{equation}
From here, we conclude that
$$
\int_{\threed} dv_*  \  w^{\ell}(v_*) \ang{v' + v_*}^{- \rho - (n-1)} e^{ - \frac{c}{4} |v' - v_*|^2}  \ \lesssim w^{\ell}(v') \ang{v'}^{-\rho - (n-1)},
$$
which finishes the lemma.
\end{proof}

With the compact estimate for the kernel from Lemma \ref{prop:Grad} in hand, we can automatically prove the main estimate for the compact term \eqref{tminushRAPplus} below.  Using similar methods, but without Lemma \ref{prop:Grad}, we will also prove \eqref{tminushRAPplus2} and \eqref{posUPPer3norm}.  

 \begin{proposition}\label{prop:GradRap}
 Consider $\phi$ satisfying \eqref{rapidDECAYfcn}.  We have the uniform estimates
\begin{equation}
 \left| \teePLUSop(g,\phi,f) \right| 
 \lesssim 
  C_\phi ~ 
 2^{2sk} ~ 
 \nsm w^\ell g\nsm_{L^2_{\gamma+2s-(n-1)}}  \nsm w^\ell f\nsm_{L^2_{\gamma+2s-(n-1)}}.
 \label{tminushRAPplus}
\end{equation}
These hold for any $k \leq 0$, $\ell \in \mathbb{R}$.  
Additionally for any $m\ge 0$ and any $k$ we obtain
\begin{align}
  \left| \teePLUSop(g,f,\phi) \right| 
 \lesssim 
  C_\phi ~ 
 2^{2sk} ~ 
 \nsm g\nsm_{L^2_{-m}}  \nsm f\nsm_{L^2_{-m}}.
 \label{tminushRAPplus2}
\end{align}
Furthermore 
\begin{equation} 
\left|  \teePLUSop (\phi,h,f)\right| 
\lesssim  
C_\phi ~   2^{2sk} 
 \nsm w^{\ell} h\nsm_{L^2_{\gamma + 2s}}   \nsm w^{\ell} f\nsm_{L^2_{\gamma + 2s}}.
 \label{posUPPer3norm}
\end{equation}
Each of these estimates hold for all parameters
 in \eqref{kernelQ}, \eqref{kernelP}, and \eqref{kernelPsing}. 
 \end{proposition}

\begin{proof}
For \eqref{tminushRAPplus}, using the formula in \eqref{defTKLcarl} and Cauchy-Schwartz we have directly
\begin{multline*}
 \left| \teePLUSop(g,\phi,f) \right| 
 \lesssim
 \int_{\mathbb{R}^n} dv' ~ w^{2\ell}(v')~ \left| f(v') \right| ~ 
\int_{\mathbb{R}^n} dv_*  ~ \left| g(v_*) \right| ~ \left| \kappa_{k}^{\psi, \phi}(v',v_*) \right|
\\
 \lesssim
  \left(
 \int_{\mathbb{R}^n} dv' ~ w^{2\ell}(v')~ |f(v')|^2 ~ 
\int_{\mathbb{R}^n} dv_*  ~ \left|  \kappa_{k}^{\psi, \phi}(v',v_*) \right|
\right)^{1/2}
\\
\times
 \left(\int_{\mathbb{R}^n} dv_* ~ |g(v_*)|^2 ~ 
\int_{\mathbb{R}^n} dv'  ~ w^{2\ell}(v')
~ \left|  \kappa_{k}^{\psi, \phi}(v',v_*) \right|
\right)^{1/2}.
\end{multline*}
Above we consider the kernel $\kappa_{k}^{\psi, \phi}(v',v_*)$ with $\psi = M_{\beta}$ and $\phi$ as in \eqref{rapidDECAYfcn}.  Now we observe that 
Lemma \ref{prop:Grad} immediately implies \eqref{tminushRAPplus}.

To prove
\eqref{tminushRAPplus2}, notice that the bound \eqref{rapidDECAYfcn} implies there is strong exponential decay in both variables $v$ and $v_*$.  The estimate \eqref{tminushRAPplus2} then follows from the Schur test for integral operators or simply using Cauchy-Schwartz as above.

The proof of \eqref{posUPPer3norm} follows from a similar application of Cauchy-Schwartz.    The main difference is that this time we have strong exponential decay in both $v'_*$ and $v_*$ which displays the appropriate symmetry under the needed pre-post collisional change of variable.
\end{proof}

This completes our basic compact estimates.  In the next section, we develop the geometric Littlewood-Paley theory adapted to the paraboloid in $n+1$ dimensions.

\section{The $n$-dimensional anisotropic Littlewood-Paley decomposition}
\label{sec:aniLP}

In this section, we develop the anisotropic Littlewood-Paley decomposition which is adapted to the paraboloid geometry.   It allows us to make sharp estimates of the linearized Boltzmann operator and to {\it explicitly} characterize the geometry that underlies it.  We also note that this geometry is a feature of the Boltzmann collision operator itself, and not an artifact of the linearization process \cite{gsNonCutEst}.  
Rather than work directly on the paraboloid, though, it turns out to be somewhat simpler to think of the Littlewood-Paley decomposition we use as being a $(n+1)$-dimensional Euclidean decomposition restricted to the paraboloid. 
This trick will allow us to use the estimates from the previous section without any additional explicit consideration of the deeper geometric aspects of our anisotropic construction (which contrasts with the approach of Klainerman and Rodnianski \cite{MR2221254}).
Throughout this section, we will use the variables $v$ and $v^\prime$ to refer to independent points in $\R^n$, meaning that we will not assume in this section that they are related by the collisional geometry.  The reason we choose to use these variable names is that they give a hint about where the Littlewood-Paley projections will be later applied in situations which do involve the collisional geometry explicitly.

\subsection{Definitions and comparison to the anisotropic norm}

Let us first consider the lifting of vectors $v \in \threed$ to the paraboloid in $\R^\last$.  Specifically, for $v \in \threed$, let $\ext{v} \eqdef (v,\frac{1}{2} |v|^2) \in \R^\last$, and consider the mappings $\tau_v : \threed \rightarrow \threed$  and $\ext{\tau}_v : \threed \rightarrow \R^\last$ given by
\[ \tau_v u \eqdef u - (1- \ang{v}^{-1}) \ang{v,u} |v|^{-2} v,
\quad \mbox{ and } 
\quad \ext{\tau}_v u \eqdef (\tau_v u,  \ang{v}^{-1} \ang{v,u}). \]
These mappings should be thought of as sending the hyperplane $v_\last = 0$ 
to the hyperplane tangent to the paraboloid $(v, \frac{1}{2} |v|^2)$ at the point $\ext{v}$.  
It is routine to check that $\ang{v, \tau_v u} = \ang{v}^{-1} \ang{v,u}$ and $|\tau_v u|^2 = |u|^2 - \ang{v}^{-2} \ang{v,u}^2$, which implies
$$
|\ext{\tau}_v u|^2
=
|\tau_v u|^2 + \ang{v, \tau_v u}^2 = |u|^2,
$$ 
meaning that $\ext{\tau}_v$ is an isometry from one hyperplane to the other. 
Moreover, it is not a hard calculation to check that
\[ \ext{v + \tau_v u} = \ext{v} + \ext{\tau}_v u + \frac{1}{2} |u|^2 e_\last, \]
where  $e_\last \eqdef (0,\ldots,0,1)$.  Next, fix any $C^\infty$ function $\varphi$ supported on the unit ball of $\R^\last$ and consider the generalized Littlewood-Paley projections given by
\begin{align}
P_j f(v) & \eqdef \int_{\threed} dv' 2^{\dim j} \varphi(2^j(\ext{v} - \ext{v'})) \ang{v'} f(v') \label{lppj} \\
Q_j f(v) & \eqdef P_j f(v) - P_{j-1} f(v), \qquad j \geq 1 \nonumber \\
 & = \int_{\threed} dv' 2^{\dim j} \psi(2^j(\ext{v} - \ext{v'})) \ang{v'} f(v'), \label{lpqj}
\end{align}
where $\psi(w) \eqdef \varphi(w) - 2^{-\dim} \varphi(w/2)$. 
Here $P_j$ corresponds to the usual projection onto frequencies at most $2^{j}$ and $Q_j$ corresponds to the usual projection onto frequencies comparable to $2^{j}$ (recall that the frequency $2^{j}$ corresponds to the scale $2^{-j}$ in physical space).  We also define $Q_0 \eqdef P_0$ to simplify notation.  In order for these projections to be generally useful, the function $\varphi$ must be chosen to satisfy various cancellation conditions which we will discuss later.   For now, we note that, as long as the integral of $\varphi$ over any $\dim$-dimensional hyperplane through the origin equals $1$ (which will be the case for any suitably normalized radial function), we have that
$P_j f(v) \rightarrow f(v)$ as $j \rightarrow \infty$ for all sufficiently smooth $f$ and that
\begin{equation} \left( \int_{\threed} dv \left| P_j f(v) \right|^p \ang{v}^\rho \right)^{\frac{1}{p}}  
\lesssim 
\left( \int_{\threed} dv \left| f(v)\right|^p \ang{v}^{\rho} \right)^{\frac{1}{p}}, \label{boundedlp}
\end{equation}
uniformly in $j \geq 0$ for any fixed $\rho \in \mathbb{R}$ and any $p \in [1,\infty)$ (with suitable variants of this inequality also holding for $p = \infty$ as well as for the operators $Q_j$).  The standard Calder\'{o}n-Zygmund theory also guarantees that $P_j f \rightarrow f$ with convergence in norm (since the metric $d(v,v')$ is the restriction of a Euclidean metric, the paraboloid is easily checked to be a space of homogeneous type when equipped with this metric).

The principal reason for defining our Littlewood-Paley projections in this way is that the particular choice of paraboloid geometry allows us to control the associated square functions by our anisotropic norm.  This informal idea is made precise in the following proposition.
\begin{proposition}
Suppose that $|Q_j (1)(v)| \lesssim 2^{-2j}$ \label{compareprop}
holds uniformly for all $v \in \threed$ and all $j \geq 0$. Then for any $s \in (0,1)$ and any real $\rho$, the following inequality holds:
\begin{equation}
\begin{split}
 \sum_{j=0}^\infty 2^{2sj} & \int_{\threed} dv ~ |Q_j f(v)|^2 \ang{v}^\rho \lesssim \\
& |f|^2_{L^2_\rho} + \int_{\threed} dv \int_{\threed} dv' 
\left( \ang{v}\ang{v'} \right)^{\frac{\rho+1}{2}}
\frac{(f(v) - f(v'))^2}{d(v,v')^{\dim+2s}} {\mathbf 1}_{d(v,v') \leq 1} . 
\end{split}
\label{compareton}
\end{equation}
This is true uniformly for all smooth  $f$.
\end{proposition}

In addition to the inequality \eqref{compareton}, it will also be necessary to establish a similar inequality when the $Q_j$'s are replaced by anisotropic derivatives $2^{-kj} \tilde \nabla Q_j$.  For the soft potentials, it is also necessary to consider commutators  of the geometric Littlewood-Paley projections and the {\it isotropic} velocity derivatives $\partial_{\beta}$.  These additional estimates are necessary because of the higher derivatives present in the norm
$\nspace$, which are ultimately necessary because the singularity of the kinetic factor is strong enough that many of the $L^2$ estimates for the hard potential case must be replaced with $L^\infty$ estimates.  These $L^\infty$ estimates, in turn, are related back to $L^2$ based spaces via Sobolev embedding.

\subsection{Littlewood-Paley commutator estimates}
For convenience, let us abbreviate 
$2^{\dim j} \varphi(2^j w) \eqdef \varphi_j(w)$ and likewise for $\psi_j$.  We seek at this point to relate the corresponding Littlewood-Paley square function to the norms 
$\spacen$ and $\nspace$ in arbitrary dimensions $n\ge 2$.  
Let us first consider the commutators of $Q_j$ with {\it isotropic} derivatives.
The commutators themselves, $ [ \frac{\partial}{\partial v_i}, Q_j ] f(v)
 \eqdef  \frac{\partial}{\partial v_i} (Q_jf ) (v)-  Q_j ( \frac{\partial}{\partial v_i}  f)(v)$, are easy to calculate:
\begin{align*}
 \left[ \frac{\partial}{\partial v_i}, Q_j \right] f(v) = & \int_{\threed} dv' 2^j \left( {\tilde \nabla}_i \psi + v_i {\tilde \nabla}_\last \psi \right)_j(\ext{v} - \ext{v'}) \ang{v'} f(v') \\
& - \int_{\threed} dv' \psi_j (\ext{v} - \ext{v'}) \ang{v'} \frac{\partial f}{\partial v_i'} (v').
\end{align*}
Above ${\tilde \nabla}_i$ is the $i$-th component of ${\tilde \nabla}$.
After an integration by parts, 
\begin{align*}
 \left[ \frac{\partial}{\partial v_i}, Q_j \right] f(v) 
 = & \int_{\threed} dv' 2^j \left({\tilde \nabla}_i \psi + v_i {\tilde \nabla}_\last \psi \right)_j(\ext{v} - \ext{v'}) \ang{v'} f(v') \\
& + \int_{\threed} d v' \frac{\partial}{\partial v_i'} \left[ \psi_j (\ext{v} - \ext{v'}) \ang{v'} \right] f(v') \\
= & \int_{\threed} dv' 2^j (v_i - v_i') \left( {\tilde \nabla}_\last \psi \right)_j(\ext{v} - \ext{v'}) \ang{v'} f(v') \\
& +  \int_{\threed} d v' \psi_j(\ext{v} - \ext{v'}) v_i' \ang{v'}^{-1} f(v').
\end{align*}
In particular, this commutator may be written as
\[ \left[ \frac{\partial}{\partial v_i}, Q_j \right] f(v) = \tilde Q_j f(v) + Q_j \tilde f(v), \]
where $\tilde Q_j$ is given by replacing $\psi(w)$ with $w_i {\tilde \nabla}_\last \psi(w)$ in the integral \eqref{lpqj} and $\tilde f(v') \eqdef v'_i \ang{v'}^{-2} f(v')$.  The end result is that, after taking $\beta$ derivatives and studying these commutators, we may always  write $\partial_{\beta} 2^{-|\alpha|j} {\tilde \nabla}^{\alpha} Q_j$ as a finite sum
$$ 
\partial_{\beta} 2^{-|\alpha|j} {\tilde \nabla}^{\alpha} Q_j f(v) 
= 
\sum_{|\beta_1| \leq |\beta|}
\sum_{k} c_{k,\beta_1}^{\alpha,\beta} Q^{k}_j ( \omega_k \partial_{\beta_1} f)(v), 
$$
where $c_{k,\beta_1}^{\alpha,\beta}\in \R$, each $Q^{k}_j$ has a form the same as \eqref{lpqj} for some $\psi^k$, the derivatives $\beta_1$ satisfy 
$|\beta_1| \leq |\beta|$, and the weights $\omega_k(v')$ are either identically one or are polynomials times powers of $\ang{v'}$ which together tend to zero at infinity.  Above ${\tilde \nabla}$ represents the $(n+1)$-dimensional restriction of the gradient of functions defined on a neighborhood of the paraboloid $(v, \frac{1}{2}|v|^2)$.

We are able to compare these weighted anisotropic Littlewood-Paley projections to the anisotropic norm using the same method which we will use to prove  Proposition \ref{compareprop},  that is,
 by completing the square.  We bound above the expression
\[ 
\int_{\threed} \! dv \int_{\threed} \! dv' \int_{\threed} dz \ \psi^k_j(\ext{v} - \ext{z}) \psi^{k}_j(\ext{v'} - \ext{z}) (f(v) - f(v'))^2 \ang{z}^{\rho} \ang{v} \ang{v'} \omega_k(v) \omega_k(v'), 
\]
by integrating over $z$ and comparing this to the semi-norm piece of $\nspace$, then we show that this term is also equal to
\[ - 2 \int_{\threed} dv ~ |Q_j^{k} (\omega_k f)(v)|^2 \ang{v}^{\rho} + 2 \int_{\threed} dv ~
Q_j^{k}(\omega_k)(v) Q_j^{k}(\omega_k f^2)(v) \ang{v}^{\rho},
 \]
(in both expressions, we suppressed the $\partial_{\beta_1}$ acting on $f$).  As long as one has the following uniform inequality
 for all $j \geq 0$ and all $v$,
 \begin{equation}
 |Q_j^{k}(\omega_k)(v)| \leq 2^{-2j},
 \label{uniformQineq}
\end{equation}
 these estimates may be multiplied by $2^{2sj}$ 
 and summed over $j$ to obtain
\begin{multline*}
 \sum_{j=0}^\infty 2^{2sj} \int_{\threed}  dv ~ |Q_j^{k} (\omega_k f)(v)|^2 \ang{v}^{\rho} \lesssim 
  |f|_{L^2_{\rho}}^2
 \\
 + \int_{\threed} dv \int_{\threed} dv' \frac{(f'-f)^2}{d(v,v')^{\dim+2s}} \left(\ang{v}\ang{v'} \right)^{\frac{\rho+1}{2}} {\mathbf 1}_{d(v,v') \leq 1},
 \quad 
 s \in (0,1).
\end{multline*}
The derivation of this inequality is similar to that of Proposition \ref{compareprop} (which has the advantage of simpler notation), so we give the reader of its proof now:
\begin{proof}[Proof of Proposition \ref{compareprop}]
For any $j \geq 1$, one has the equality
\begin{align*}
 \frac{1}{2} \int_{\threed}  \! \! \! dv & \! \int_{\threed}  \! \! \! dv' \! \int_{\threed} \! \! \! dz ~  (f(v) - f(v'))^2 \psi_j(\ext{z} - \ext{v}) \psi_j(\ext{z} - \ext{v'}) \ang{v} \ang{v'} \ang{z}^\rho  \\ 
 & =   - \int_{\threed} \! \! dv  ~
 ([Q_j f] (v))^2 \ang{v}^\rho + \int_{\threed}  \! \! \! dv \!  \int_{\threed} \! \! \! dz   (f(v))^2 \psi_j(\ext{z} - \ext{v}) Q_j(1)(z) \ang{v} \ang{z}^\rho,
\end{align*}
simply by expanding the square $(f(v) - f(v'))^2$ and exploiting the symmetry of the integral in $v$ and $v'$ (note also that the corresponding statement holds true for $Q_0$ when $\psi$ is replaced by $\varphi$).
By our assumption on the projections $Q_j$, namely $|Q_j(1)(z)| \lesssim 2^{-2j}$, we may control the second term on the right-hand side by
\begin{align*}
2^{-2j} \int_{\threed} dv ~ (f(v))^2 \ang{v}^{\rho}.
\end{align*}
This  bound follows from the change of variables 
$z \mapsto v + 2^{-j} \tau_v u$, which 
is a change the variable from $z$ to $u$ and has Jacobian $\ang{v}^{-1}$, so that
\begin{align*}
 \int_{\threed} \! \! dz |\psi_j(\ext{z} - \ext{v})| \ang{z}^{\rho} & =
\int_{\threed} \! \!du ~ |\psi( \ext{\tau}_v u + 2^{-j-1} |u|^2 e_\last)| \frac{\ang{v + 2^{-j} \tau_v u}^\rho}{\ang{v}}  \lesssim \ang{v}^{\rho-1}.
\end{align*}
We have used that $\ang{v} \approx \ang{z}$ on the support of $|\psi_j(\ext{z} - \ext{v})|$ since $j$ is non-negative.
Moreover, the same change of variables can be used to show that
\[ 
\int_{\threed} dz ~ |\psi_j(\ext{z} - \ext{v})| |\psi_j(\ext{z} - \ext{v'})| \ang{z}^\rho \lesssim 2^{\dim j} (\ang{v} \ang{v'})^{\frac{\rho+1}{2}}. 
\]
Here we use the inequality $|\psi_j(\ext{z} - \ext{v'})| \lesssim 2^{\dim j}$ and observe that $\ang{v} \approx \ang{z} \approx \ang{v'}$ on the support of the original integral.
Moreover, the triangle inequality guarantees that the integral is only nonzero when $d(v,v') \leq 2^{-j+1}$, so that we have
\[ 
2^{2sj} \int_{\threed} dz |\psi_j(\ext{z} - \ext{v})| |\psi_j(\ext{z} - \ext{v'})| \ang{z}^\rho 
\lesssim 
(\ang{v} \ang{v'})^{\frac{\rho+1}{2}} 
~ 2^{2sj} 2^{nj} {\mathbf 1}_{d(v,v') \leq 2^{-j+1}}. 
\]
These estimates may  be summed over $j \geq 1$ (when $v \ne v'$) because the sum terminates after some index $j_0$ with $2^{-j_0} < d(v,v') \leq 2^{-j_0+1}$, 
 yielding \eqref{compareton}, since 
\[ 
\sum_{j=1}^\infty 2^{2sj} 2^{nj} {\mathbf 1}_{d(v,v') \leq 2^{-j+1}} 
=
\sum_{j=1}^{j_0} 2^{2sj} 2^{nj} {\mathbf 1}_{d(v,v') \leq 2^{-j+1}} 
\lesssim
d(v,v')^{-\dim-2s} {\mathbf 1}_{d(v,v') \leq 1}. 
\]
The last inequality follows from $2^{(2s+n){j_0}} \lesssim d(v,v')^{-\dim-2s}$.  
 The remaining term $j=0$ is already bounded above by $|f|_{L^2_{\rho}}^2$.
\end{proof}

Thus, following the proof of Proposition \ref{compareprop},
subject only to the establishment of the 
 decay condition \eqref{uniformQineq} similar to 
 Proposition \ref{compareprop}, it follows that
\begin{equation}
 \sum_{j=0}^\infty 2^{2(s-|\alpha|)j} \int_{\threed} dv ~ |{\tilde \nabla}^\alpha Q_j f (v)|^2 
 \ang{v}^{\gamma + 2s} w^{2\ell}(v)
\lesssim  |f|_{\spaceELLn}^2.
\label{lpsobolev00}
\end{equation}
We also have 
\begin{equation}
 \sum_{|\beta| \leq K} \sum_{j=0}^\infty 2^{2(s-|\alpha|)j} \int_{\threed} dv ~ |\partial_{\beta} {\tilde \nabla}^\alpha Q_j f (v)|^2 
 \ang{v}^{\gamma + 2s} w^{2\ell - 2 |\beta|}(v)
\lesssim  |f|_{\nspace}^2.
\label{lpsobolev0}
\end{equation}
These will hold for any multi-index $\alpha$ of derivatives  on $\R^\last$ and any fixed $K$ and $\ell \in \R$.  
The catch, of course, is that the functions $\psi^k$ become increasingly difficult to control when either $|\alpha|$ or $K$ becomes large.  This means that the uniform estimate \eqref{uniformQineq} is increasingly difficult to obtain.  In the next section, we will establish the desired inequality contingent on the following cancellation condition; that $\psi^k$ satisfies
\begin{equation*}
 \int_{\threed} du ~ p(u) ({\tilde \nabla}^\alpha \psi^k)(\ext{\tau}_v u) = 0,
\end{equation*}
for all polynomials $p$ of degree $1$
 and all multi-indices $|\alpha| \leq 1$.  Since $\psi^k$ is itself related to the original $\psi$ by taking a sequence of ${\tilde \nabla}$-derivatives and multiplying by a polynomial, it is sufficient to choose the original $\varphi$  so that $\psi$ satisfies
\begin{equation}
 \int_{\threed} du ~ p(u) ({\tilde \nabla}^\alpha \psi)(\ext{\tau}_v u) = 0, \label{lpcancel}
\end{equation}
for all polynomials of degree at most $M$ and all $|\alpha| \leq M$, where $M$ is any fixed but arbitrary natural number.

\subsection{Selection of $\varphi$ and inequalities for smooth functions}

Let $D$ be the dilation on functions in $\R^\last$ given by $D \varphi(w) \eqdef \varphi \left(\frac{w}{2} \right)$. 
We choose a radial function $\varphi_0 \in C^\infty(\R^\last)$  which is supported on the ball $|w| \leq R$, with $R>0$, and satisfies
\begin{equation} 
\int_{\threed} du \ \varphi_0(\ext{\tau}_v u) = 1, \qquad \forall v \in \threed. 
\label{lpnormalize}
\end{equation}
Since $\varphi_0$ is radial this equality will be true for all $v$ if it is true for any single $v$.  If $p$ is any homogeneous polynomial on $\threed$, a simple scaling argument shows that
\[ 
\int_{\threed} du ~ p(u) ({\tilde \nabla}^\alpha \circ D \varphi_0) (\ext{\tau}_v u) = 2^{\dim + \deg p - |\alpha|} \int_{\threed} du ~ p(u) ({\tilde \nabla}^\alpha \varphi_0)( \ext{\tau}_v u). 
\]
Next, fix some large integer $M$ and consider 
the $C^\infty$ function $\varphi$ which is supported on the ball of radius $2^{2M} R$ and is given by 
\[ \varphi \eqdef \left( \prod_{|k| \leq M, k \neq 0} \frac{I - 2^{-\dim + k} D}{1-2^k} \right) \varphi_0.  \]
By induction on $M$ and the scaling argument just mentioned, it follows that $\varphi$ satisfies exactly the same normalization condition as $\varphi_0$, namely \eqref{lpnormalize}.  Assuming that $M$ is fixed, the radius $R$ may be chosen so that $2^{2M} R \leq 1$.  For this fixed $M$, we take the corresponding function $\varphi$ to be the basic building block of our Littlewood-Paley projections, i.e., we use this $\varphi$ in the definition \eqref{lppj}, and use $\psi \eqdef \varphi - 2^{-\dim} D \varphi$ in \eqref{lpqj}. By our particular choice of $\varphi$, this $\psi$ satisfies the cancellation conditions \eqref{lpcancel} for all polynomials $p$ of degree at most $M$ and all multi-indices $\alpha$ of order at most $M$.
Consequently, we have the following integral estimates:

\begin{lemma}
Choose a large integer $M$.  Suppose that \eqref{lpcancel} holds for all polynomials of degree at most $M$ and all multi-indices $\alpha$ of order at most $M$.  Fix any multi-index $\alpha$ with $\last$ components.  If $f$ is a smooth function on $\R^\dim$,
and $g$ is some non-negative function satisfying for all $v \in \threed$ that
\[ 
\sup_{d(v,v') \leq 1} \sum_{|\beta| \leq k+1} | (\partial_{\beta} f) (v')| \leq C_f ~ g(v), 
\quad C_f >0.
\]
Here $k$ is some integer satisfying $-1 \leq k \leq \frac{1}{2}(M - |\alpha|)$.  Then the  inequality
\begin{equation}
 \left| 2^{-|\alpha| j} {\tilde \nabla}^{\alpha} Q_j f(v) \right| \lesssim 2^{-(k+1)j} C_f ~ g(v),\label{lpsmooth}
\end{equation}
holds uniformly in $f$, $v$, and $j \geq 0$.
\end{lemma}

\begin{proof}
We first consider the case $\alpha = 0$.
We study the integral $I_j(v)$ given by
\begin{align*} 
I_j(v) & \eqdef \int_{\threed} dv' \psi_j(\ext{v} - \ext{v'}) \omega(v') \\
 & = \ang{v}^{-1} \int_{\threed} du \ \psi( \ext{\tau}_v u + 2^{-j-1} |u|^2 e_ \last) \omega(v + 2^{-j} \tau_v u),
\end{align*}
where the equality follows after the change $v' \mapsto v + 2^{-j} \ext{\tau}_v u$.  If we expand the integrand in powers of $2^{-j}$ by means of Taylor's theorem, we have  an asymptotic expansion of this integral, with the coefficient of $2^{-kj}$ equaling
\[ 
\ang{v}^{-1} \frac{1}{k!}\int_{\threed} du \left(  \left. \frac{d^k}{d \epsilon^k} \right|_{\epsilon=0} \psi( \ext{\tau}_v u + \frac{\epsilon}{2} |u|^2 e_ \last) \omega(v + \epsilon \tau_v u) \right). 
\]
This can subsequently be expanded as a sum of terms, each of which is an integral of some derivative of $\psi$ times a polynomial in $u$.  The order of differentiation is at most $k$, and the degree of the polynomial is at most $2k$.  Consequently, if $2k \leq M$, the $k$-th term in the asymptotic series will vanish identically.  Using the integral form of the remainder in Taylor's theorem, it follows that $I_j(v)$ is exactly equal to 
 \[ \ang{v}^{-1} \frac{1}{k!} \int_{\threed} du \int_0^1 d \Btheta ~ (1- \Btheta)^{k} \frac{d^{k+1}}{d \Btheta^{k+1}} \left[ \psi( \ext{\tau}_v u + 2^{-j-1} \Btheta |u|^2 e_ \last) \omega(v + 2^{-j} \Btheta \tau_v u) \right], \]
for any $k \leq \frac{M}{2}$.  
From here, it is elementary to see that
\[ 
\left| I_j(v) \right| \lesssim 2^{-(k+1)j} \ang{v}^{-1} \sup_{d(v,v')  \leq 2^{-j} \leq 1} \sum_{|\beta| \leq k+1}  | (\partial_{\beta} \omega )(v')|, 
\]
uniformly for all $v$ and all $j \geq 0$.  Setting $\omega(v') = \ang{v'} f(v')$ establishes the result for $\alpha = 0$.  When $\alpha \neq 0$, the effect of including an additional $2^{-|\alpha| j} {\tilde \nabla}^\alpha$ is to replace $\psi$ with ${\tilde \nabla}^\alpha \psi$ in the definition of $I_j$.  The proof follows exactly as before, where now one only has \eqref{lpcancel} for derivatives up through order $M - |\alpha|$.
\end{proof}

This lemma 
may now be directly applied to obtain
 the uniform bounds \eqref{uniformQineq} on $|Q_j^k (\omega_l)(v)|$ as  required for Proposition \ref{compareprop} and \eqref{lpsobolev00}, \eqref{lpsobolev0} to be true.

\section{Upper bounds for the trilinear form}
\label{sec:upTRI}

In this section, we establish Theorem \ref{TriLinEst} for the nonlinear term $\Gamma$ as well as 
Lemmas \ref{NonLinEstLOW}, \ref{NonLinEstHIGH}, and \ref{sharpLINEAR}.

\subsection{The main upper bound for hard potentials: $\gamma \ge -2s$}
We begin with the proof of Theorem \ref{TriLinEst}; we more generally prove the estimate with weights in \eqref{generalUPPER}.
We'll write 
$$
f = P_0 f + \sum_{j=1}^\infty Q_j f \eqdef \sum_{j=0}^\infty f_j,
$$ 
and likewise for $h$, then expand the trilinear form: 
\begin{equation}
  \ang{w^{2\ell} \Gamma(g,h),f}  
= 
\sum_{l=1}^\infty \sum_{j=0}^\infty   \ang{w^{2\ell} \Gamma(g,h_{j+l}),f_j} + 
\sum_{l=0}^\infty \sum_{j=0}^\infty \ang{w^{2\ell} \Gamma(g,h_j),f_{j+l}}. 
\label{mainexpand}
\end{equation}
Consider the sum over $j$ of the terms $\ang{w^{2\ell} \Gamma(g,h_{j+l}),f_j}$ for fixed $l$.  
We expand $\Gamma$ 
 by introducing the cutoff around the singularity of $b$ in terms of $\teePLUSop$ and $\teeMINUSop$:
\begin{align}
 \sum_{j=0}^\infty \ang{ w^{2\ell} \Gamma(g,h_{j+l}),f_j} 
 & = 
 \sum_{k=-\infty}^\infty  \sum_{j=0}^\infty 
\left\{ \teePLUSop(g,h_{j+l},f_j) - \teeMINUSop(g,h_{j+l},f_j) \right\} \nonumber 
\\
& = 
\sum_{j=0}^\infty \sum_{k=-\infty}^j \left\{ \teePLUSop(g,h_{j+l},f_j) - \teeMINUSop(g,h_{j+l},f_j) \right\} \label{nearsing}
\\
& \hspace{30pt} 
+ 
\sum_{j=0}^\infty \sum_{k=j+1}^\infty \left\{ \teePLUSop(g,h_{j+l},f_j) - \teeMINUSop(g,h_{j+l},f_j) \right\}. \label{cancelsing}
\end{align}
Throughout the manipulation, the order of summation may be rearranged with impunity since the estimates we employ below will imply that the sum is absolutely convergent when $g$, $h$, $f$ are all Schwartz functions.   Regarding the terms \eqref{nearsing},  
the inequalities \eqref{tminusgh} and \eqref{tplussmallK} guarantee that
\begin{multline*}
\sum_{j=0}^\infty \sum_{k=-\infty}^{j}   \left| (\teePLUSop -\teeMINUSop) (g,h_{j+l},f_j) \right| 
\\
\lesssim \sum_{j=0}^\infty 2^{2sj} \nsm w^{\ell^+ - \ell'} g\nsm_{L^2} 
\nsm w^{\ell + \ell'}  h_{j+l}\nsm_{L^2_{\gamma+2s}} \nsm w^{\ell} f_j\nsm_{L^2_{\gamma+2s}} 
\\
 \lesssim 2^{-sl} \nsm w^{\ell^+ - \ell'} g\nsm_{L^2} \left| \sum_{j=0}^\infty 2^{2s(j+l)} \nsm w^{\ell + \ell'} h_{j+l}\nsm ^2_{L^2_{\gamma+2s}} \right|^{\frac{1}{2}} \left| \sum_{j=0}^\infty 2^{2sj} \nsm w^{\ell} f_j\nsm ^2_{L^2_{\gamma+2s}} \right|^{\frac{1}{2}}
 \\
 \lesssim 2^{-2sl} \nsm w^{\ell^+ - \ell'} g\nsm_{L^2} \nsm  h\nsm_{N^{s,\gamma}_{\ell + \ell'}} 
\nsm f\nsm_{\spaceELLn},
\end{multline*}
where for the first inequality, we have used the trivial facts that $2^{2sk} = 2^{2sj} 2^{2s(k-j)}$ and 
$\sum_{k=-\infty}^{j} 2^{2s(k-j)}\lesssim 1$, the second follows by Cauchy-Schwartz on the sum over $j$, and the third by \eqref{lpsobolev00}. 
This estimate may clearly also be summed over $l \geq 0$.  Also here $\ell = \ell^+ - \ell^-$ with $\ell^\pm \ge 0$ and $0 \le \ell' \le \ell ^-$ as in Proposition \ref{referLATERprop}.

A completely analogous argument may be used to expand $\Gamma$ for the terms in \eqref{mainexpand} of the form $\ang{w^{2\ell}\Gamma(g,h_j),f_{j+l}}$ in terms of $\teePLUSop - \teeSTARop$ from \eqref{defTKLcarl}:
\begin{align}
\sum_{j=0}^\infty \ang{w^{2\ell} \Gamma(g,h_j),f_{j+l}} & = 
\sum_{j=0}^\infty \LopGstar(g,h_j,f_{j+l})
+
\sum_{j=0}^\infty \sum_{k=-\infty}^\infty (\teePLUSop - \teeSTARop)(g,h_j,f_{j+l}) \nonumber \\
& = 
\sum_{j=0}^\infty \LopGstar(g,h_j,f_{j+l})
\label{sum1} 
+ \sum_{j=0}^\infty \sum_{k=-\infty}^j (\teePLUSop - T_*^k)(g,h_j,f_{j+l}) 
\\
& \hspace{30pt} + \sum_{j=0}^\infty \sum_{k=j+1}^\infty (\teePLUSop - T_*^k)(g,h_j,f_{j+l}) \label{sum2}.
\end{align}
In this case the estimates \eqref{tminusgh} and \eqref{tplussmallK} are used to handle the second sum in \eqref{sum1} just as the corresponding terms \eqref{nearsing} were handled.  The only difference is that the roles of $h$ and $f$ are now reversed.  For the first sum in \eqref{sum1}, when $\gamma > - \frac{n}{2}$, we use Proposition \ref{opGstarEST} and the inequality \eqref{opGstarINEQ1} to obtain
$$
\left| \LopGstar(g,h_j,f_{j+l}) \right|
\lesssim
     2^{-sl} \nsm w^\ell g\nsm_{L^2} 
~ 2^{sj} \nsm w^\ell h_j\nsm_{L^2_{\gamma}} 
~ 2^{s(j+l)}\nsm w^\ell f_{j+l}\nsm_{L^2_{\gamma}},
$$
which also used the inequality
$
1 \le 2^{2sj} = 2^{-sl}  2^{s(j+l)} 2^{sj}.
$
Again the Cauchy-Schwartz inequality on the $j$ index  and \eqref{lpsobolev00} 
yield the desired upper bound.   If $\gamma < - \frac{n}{2}$ we must include the second term in \eqref{opGstarINEQ2}. 
For this second term, we use the inequalities
\[ |h_j M^{\delta}|_{H^{s + \epsilon}} \lesssim |h_j|_{L^2_{-m}}^{1-(s+\epsilon)} ||\tilde \nabla| h_j|_{L^2_{-m}}^{s + \epsilon} \]
and likewise for $f_{j+l}$, which hold as long as $0 \leq \epsilon \leq \min\{s,1-s\}$.  Thus we have that, 
modulo the lower-order terms we have already handled,
\begin{align*}
|\LopGstar(g,h_j, f_{j+l})| \lesssim & |g|_{L^2_{-m}} |h_j|_{L^2_{-m}}^{1-(s+\epsilon)}  ||\tilde \nabla| h_j|_{L^2_{-m}}^{s + \epsilon} |f_{j+l}|_{L^2_{-m}}^{1-(s-\epsilon)} ||\tilde \nabla| f_{j+l}|_{L^2_{-m}}^{s - \epsilon}. 
\end{align*}
Now let $t_{i_1,i_2}^{j,l}  \eqdef 2^{-\epsilon l} 2^{(s-i_1)j}  ||\tilde \nabla|^{i_1} h_j|_{L^2_{-m}} 2^{(s-i_2)(j+l)} ||\tilde \nabla|^{i_2} f_{j+l}|_{L^2_{-m}}$ and $S_{i_1,i_2} \eqdef \sum_{j=0}^\infty \sum_{l=1}^\infty  t_{i_1,i_2}^{j,l}$ for indices $i_1, i_2 \in \{0,1\}$.
Algebraic manipulation gives that
\begin{align*}  
|h_j|_{L^2_{-m}}^{1-(s+\epsilon)} & ||\tilde \nabla| h_j|_{L^2_{-m}}^{s + \epsilon} |f_{j+l}|_{L^2_{-m}}^{1-(s-\epsilon)} ||\tilde \nabla| f_{j+l}|_{L^2_{-m}}^{s - \epsilon}
=  (t_{1,1}^{j,l})^{s-\epsilon} (t_{0,0}^{j,l})^{1-(s+\epsilon)} (t_{1,0}^{j,l})^{2 \epsilon} 
\end{align*}
Consequently, by H\"{o}lder, we have
\begin{align*} \sum_{j=0}^\infty \sum_{l=0}^\infty  |\LopGstar(g,h_j,f_{j+l})| \lesssim & 
|g|_{L^2_{-m}} (S_{1,1})^{s-\epsilon} (S_{0,0})^{1-(s+\epsilon)} (S_{1,0})^{2 \epsilon}, 
\end{align*}
modulo the sum we have already estimated.  We finish the consideration of the $\LopGstar$ operators by noting that each $S_{i_1,i_2}$ may be estimated just as before by an application of Cauchy-Schwartz to the sum over $j$.

Recalling the original expansion of $\ang{w^{2\ell} \Gamma(g,h),f}$ it is clear that the only terms that remain to be considered are 
\eqref{cancelsing} and  \eqref{sum2}.  These terms are both treated by the cancellation inequalities.  
The terms \eqref{cancelsing}, for example, are handled by \eqref{cancelf}.  
$$
 \left| \left( \teePLUSop  - \teeMINUSop \right)(g,h_{j+l},f_j) \right|
\lesssim 2^{(2s-i) k} \nsm w^\ell g\nsm_{L^2} \nsm w^\ell h_{j+l}\nsm_{L^2_{\gamma+2s}} 
\nl w^\ell |\tilde{\nabla}|^i f_j \nr_{L^2_{\gamma+2s}},
$$
Because $2s-i < 0$ there is decay of the norm as $k \rightarrow \infty$, we may conclude
$$
 \sum_{k=j+1}^\infty  
 \left| \left( \teePLUSop  - \teeMINUSop \right)(g,h_{j+l},f_j) \right|
 \lesssim 
 2^{(2s-i)j} \nsm w^\ell g\nsm_{L^2} \nsm w^\ell h_{j+l}\nsm_{L^2_{\gamma+2s}} 
\nl w^\ell |\tilde{\nabla}|^i f_j \nr_{L^2_{\gamma+2s}},
$$
Just as before, Cauchy-Schwartz is applied to the sum over $j$.  In this case $2^{(2s-i)j}$ is written as $2^{(s-i)j} 2^{s(j+l)} 2^{-sl}$; the first factor goes with $f$, the second with $h$, and the third remains for the sum over $l$.  Once again \eqref{lpsobolev00} is employed.  


The desired bound for the trilinear term is completed by performing summation of the terms \eqref{sum2}.  The pattern of inequalities is exactly the same as the one just described, this time using \eqref{cancelh}.  In particular, one has that
$$
 \left| \left(\teePLUSop - \teeSTARop \right)(g,h_{j},f_{j+l}) \right|
  \lesssim 
2^{(2s-i)k} \nsm w^\ell g\nsm_{L^2} \nl w^\ell |\tilde{\nabla}|^i h_j \nr_{L^2_{\gamma+2s}} \nsm w^\ell f_{j+l}\nsm_{L^2_{\gamma+2s}}
$$
with  $2s - i < 0$. This leads to the corresponding inequality for the sum over $k$:
$$
 \sum_{k=j+1}^\infty  
 \left| \left(\teePLUSop - \teeSTARop \right)(g,h_{j},f_{j+l}) \right|
 \lesssim 
 2^{(2s-i)j} \nsm w^\ell g\nsm_{L^2} \nl w^\ell |\tilde{\nabla}|^i h_j \nr_{L^2_{\gamma+2s}} \nsm w^\ell f_{j+l}\nsm_{L^2_{\gamma+2s}}.
$$
The same Cauchy-Schwartz estimate is used for the sum over $j$; there  is exponential decay allowing the sum over $l$ to be estimated.  The end result includes  Theorem \ref{TriLinEst} as a special case, and more generally establishes the upper bound 
\begin{equation}
\left| \ang{w^{2 \ell} \Gamma_\beta(g,h),f} \right| 
\lesssim | w^{\ell^+ - \ell'} g|_{L^2} | h|_{N^{s,\gamma}_{\ell + \ell'}} | f|_{N^{s,\gamma}_{\ell}}. 
\label{generalUPPER}
\end{equation}
This inequality holds under either \eqref{kernelP}, or \eqref{kernelPsing} combined with $\gamma +2s> -\frac{n}{2}$.

\subsection{Trilinear upper bounds with soft potentials: $-2s> \gamma > -n$}
\label{sec:upTRIsub}

We will now prove  non-linear estimates in the velocity norms in Lemma \ref{NonLinEstA}.  We have

\begin{lemma}
\label{NonLinEstA}
(Trilinear estimate)
Consider the non-linear term \eqref{gamma0} and \eqref{DerivEstG}.

For any multi-index $\beta$, any $\ell^+, \ell^-, \ell' \ge 0$ with $\ell = \ell^+ - \ell^-$ and $\ell' \le \ell^-$ we have
\begin{equation}
 |\ang{w^{2\ell}\Gamma_{\beta}( g,  h), f} | 
 \lesssim 
  \nsm w^{\ell^+ - \ell'} g\nsm_{L^2} 
  \nsm h\nsm_{N^{s,\gamma}_{\ell + \ell', \ksob} }  \nsm f\nsm_{\spaceELLn}     
 \label{nlineq1}
\end{equation}
Here $\ell + \ell' = \ell^+ - (\ell^- - \ell' )$.
We alternatively use the Sobolev embedding on $g$:
\begin{equation}
 |\ang{w^{2\ell}\Gamma_{\beta}( g,  h), f} | 
 \lesssim 
\nsm  g\nsm_{H^{\ksob}_{{\ell^+} - \ell'}}  
  \nsm h\nsm_{N^{s,\gamma}_{\ell + \ell'} }  \nsm f\nsm_{\spaceELLn}
\label{nlineq}
\end{equation}
These estimates hold for the soft potentials \eqref{kernelPsing} 
and the hard potentials \eqref{kernelP}.
\end{lemma}

These estimates immediately imply 
Lemma \ref{NonLinEstHIGH}, as we explain just now; the same process will deduce Lemma \ref{NonLinEstLOW} given the results of the previous subsection.

\begin{remark}\label{rem:sobolev}
First let us record the following particularly useful variation of the Sobolev embedding theorem:
suppose $g$, $h$, $f$ are functions on ${\mathbb T}^n \times \threed$.  For any velocity spaces $X,Y,Z$ (i.e., function spaces on $\threed$) such as those appearing in Section \ref{sec:FuncSp} and any nonnegative integers $k_1,k_2$ such that $k_1 + k_2 > \frac{n}{2}$, we have 
\[ \int_{{\mathbb T}^n} dx ~ |g|_{X} |h|_{Y} |f|_{Z} \leq C_{k_1,k_2,n} ||g||_{H^{k_1}_x X_v} ||h||_{H^{k_2}_x Y_v} ||f||_{L^2_x Z_v}. \]
This is a consequence of the functional Sobolev embedding theorem, stating that
\[ 
\left( \int_{{\mathbb T}^n} dx \ |g|_{X}^q \right)^{\frac{1}{q}} \leq C_{q,n} ||g||_{H^k_x X_v}, \quad \mbox{ when } \quad \frac{k}{n} + \frac{1}{q} \geq \frac{1}{2}, \quad 1 \leq q < \infty. 
\]
and
\[ \mathop{\mathrm{ess.sup}}_{x \in {\mathbb T}^n} |g|_{X_v} \leq C_n ||g||_{H^k_x X_v}, \quad \mbox{ when } \frac{k}{n} > \frac{1}{2}. \]
Combine that with H\"{o}lder's inequality on ${\mathbb T}^n$, if $\alpha_1 \leq \alpha$ and $\beta_1 \leq \beta$,  to obtain
\begin{multline}
\int_{{\mathbb T}^n} dx ~ |\partial^{\alpha - \alpha_1}_{\beta - \beta_1} g|_{X} | \partial^{\alpha_1}_{\beta_1} h|_{Y} |f|_{Z} 
\\
\lesssim ||\partial_{\beta - \beta_1} g||_{H^{K-|\beta-\beta_1|}_x X_v} ||\partial_{\beta_1} h||_{H^{K-|\beta_1|}_x Y_v} ||f||_{L^2_x Z_v},  \label{sobolev}
\end{multline}
whenever $K$ satisfies $K \geq \max\{ |\alpha - \alpha_1| + |\beta - \beta_1|, |\alpha_1| + |\beta_1| \}$ and $2K \geq |\alpha| + |\beta| + \ksob$.
\end{remark}

\begin{proof}[Proof of Lemma \ref{NonLinEstHIGH}]
Now consider derivatives of the non-linear term as in \eqref{DerivEstG} which include velocity derivatives; a typical term is
\[ \left|\ang{w^{2\ell - 2 | \beta|} \Gamma_{\beta_2} (\partial^{\alpha - \alpha_1}_{\beta - \beta_1}g,\partial^{\alpha_1}_{\beta_1}h), \partial^{\alpha}_{\beta} f} \right|. \]
Note that 
since 
$|\alpha| + |\beta|  \leq K$
we also have
$|\alpha_1| + |\beta_1| + |\alpha - \alpha_1| + |\beta - \beta_1| \leq K$.  To estimate this term, we will apply either \eqref{nlineq1} or \eqref{nlineq}; in the former case, the right-hand side applies an additional $\ksob =  \lfloor \frac{n}{2} +1 \rfloor$ velocity derivatives to $\partial^{\alpha_1}_{\beta_1} h$ (a consequence of Sobolev embedding), and in the latter, the same $\ksob$ velocity derivatives are applied instead to $ \partial^{\alpha-\alpha_1}_{\beta - \beta_1} g$.  We will make the choice of \eqref{nlineq1} versus \eqref{nlineq} so that the total number of derivatives on either $h$ or $g$ does not exceed $K$.

To that end, consider the situation when $|\alpha_1| + |\beta_1| \leq K -  \ksob$.
  Applying \eqref{nlineq1} 
with $\ell^+ = \ell \ge 0$, $\ell^- = |\beta|$  and $\ell' = |\beta - \beta_1|$, so $\ell^--\ell' = |\beta_1|$, gives the estimate
\begin{multline*}
 \left|\ang{w^{2\ell - 2 | \beta|} \Gamma_{\beta_2} (\partial^{\alpha - \alpha_1}_{\beta - \beta_1}g,\partial^{\alpha_1}_{\beta_1}h), \partial^{\alpha}_{\beta} f} \right| 
 \\
 \lesssim 
  \nsm w^{\ell - |\beta - \beta_1|} \partial^{\alpha - \alpha_1}_{\beta - \beta_1}g\nsm_{L^2} 
  \nsm \partial^{\alpha_1}_{\beta_1}h\nsm_{N^{s,\gamma}_{\ell -|\beta_1|, \ksob} }  
  \nsm \partial^{\alpha}_{\beta} f\nsm_{N^{s,\gamma}_{\ell - |\beta|}}.     
\end{multline*}
Note that the $\ell$ right here is not the same $\ell$ as the one in Lemma \ref{NonLinEstA}; instead we use Lemma \ref{NonLinEstA} with $\ell$ replaced by $\ell - |\beta|$.  Now after also integrating over $\mathbb{T}^n$ we can use \eqref{sobolev} to establish that
\begin{multline*}
 \left|\left( w^{2\ell - 2 | \beta|} \Gamma_{\beta_2} (\partial^{\alpha - \alpha_1}_{\beta - \beta_1}g,\partial^{\alpha_1}_{\beta_1}h), \partial^{\alpha}_{\beta} f \right) \right| 
 \\
 \lesssim 
  \| w^{\ell - |\beta - \beta_1|} \partial_{\beta - \beta_1}g\|_{H^{K-|\beta - \beta_1|}_x L^2_v} 
  \| \partial_{\beta_1}h\|_{H^{K - |\beta_1| - \ksob}_x N^{s,\gamma}_{\ell -|\beta_1|, \ksob} }  
  \| \partial^{\alpha}_{\beta} f\|_{ N^{s,\gamma}_{\ell - |\beta|}},     
\end{multline*}
under the constraint that $2K \geq |\alpha| + |\beta| + 2 \ksob$ (since the assumption on $|\alpha_1| + |\beta_1|$ guarantees the remaining inequalities for $K$ are true).  By symmetry, if $|\alpha - \alpha_1| + |\beta - \beta_1| \leq K -  \ksob$, we may instead apply \eqref{nlineq} and estimate the $g$ terms on the right-hand side via Sobolev embedding \eqref{sobolev} as in the previous case.  Thus, as long as $2K \geq |\alpha| + |\beta| + 2 \ksob$, one of the two alternatives will always be applicable.  Since $|\alpha| + |\beta| \leq K$, the constraint $K \geq 2 \ksob$ suffices to ensure Lemma \ref{NonLinEstHIGH} holds.
\end{proof}

\begin{proof}[Proof of Lemma \ref{NonLinEstLOW}]
The inequality in Lemma \ref{NonLinEstLOW} may be established via a similar argument using  \eqref{generalUPPER} rather than \eqref{nlineq}.  Without the use of Sobolev embedding in the velocity variables, the argument above establishes 
\begin{multline*}
 \left|\left( w^{2\ell - 2 | \beta|} \Gamma_{\beta_2} (\partial^{\alpha - \alpha_1}_{\beta - \beta_1}g,\partial^{\alpha_1}_{\beta_1}h), \partial^{\alpha}_{\beta} f \right) \right| 
 \\
 \lesssim 
  \| w^{\ell - |\beta - \beta_1|} \partial_{\beta - \beta_1}g\|_{H^{K-|\beta - \beta_1|}_x L^2_v} 
  \| \partial_{\beta_1}h\|_{H^{K - |\beta_1|}_x N^{s,\gamma}_{\ell -|\beta_1|} }  
  \| \partial^{\alpha}_{\beta} f\|_{N^{s,\gamma}_{\ell - |\beta|}}     
\end{multline*}
for any multiindices $\alpha, \alpha_1, \beta,\beta_1$ provided that $2K \geq |\alpha| + |\beta| + \ksob$.  This establishes Lemma \ref{NonLinEstLOW}  for the hard potentials whenever $K \geq \max \{ |\alpha| + |\beta|, \ksob\}$.
\end{proof}

Now we set about to prove the two non-linear estimates in Lemma \ref{NonLinEstA}.

\begin{proof}[Proof of  Lemma \ref{NonLinEstA}]
Both \eqref{nlineq1} and \eqref{nlineq} are established by the same summation procedure  used to establish Theorem \ref{TriLinEst} in the previous subsection.   Specifically, we expand exactly as in \eqref{mainexpand}.  The only difference is that we replace the estimates from Section \ref{physicalDECrel} with their analogues from Section \ref{sec2:physicalDECrel}. 
We control
\eqref{nearsing} with the estimates in Propositions \ref{prop11} and \ref{referLATERprop2}. 
Then Propositions \ref{starPROP}, \ref{referLATERprop2}, and \ref{opGstarESTderiv} are used to 
handle the terms \eqref{sum1}.
For the cancellations,
\eqref{cancelsing} is handled by  Proposition \ref{cancelFprop} 
and \eqref{sum2} is controlled using Proposition \ref{cancelHprop}. 
To establish \eqref{nlineq1}, we simply use the inequalities from these propositions in Section \ref{sec2:physicalDECrel} which apply the derivatives to $h$; for \eqref{nlineq}, the corresponding estimates with derivatives on $g$ are used.
\end{proof}

This concludes our main non-linear estimates. 

\subsection{The Compact Estimates}
Here we collect some estimates for the linearized collision operator.
The first one is the key to the estimate for $\kPiece$ in \eqref{compactupper}.

\begin{lemma}
\label{CompactEst}
(Compact Estimate)
For any $\ell \in \mathbb{R}$, we have the uniform estimate
\begin{equation}
\notag
\left| \langle  w^{2\ell}  \kPiece g, h \rangle  \right|
 \lesssim 
 \nsm w^{\ell} g\nsm_{L^2_{\gamma+2s - \delta }} \nsm w^{\ell} h\nsm_{L^2_{\gamma+2s - \delta}},
 \quad
\delta = \min\{2s, (n-1)\}.
 \end{equation}
\end{lemma}

Since $\delta>0$ above, Lemma \ref{CompactEst} easily implies \eqref{compactupper}  when $g = h$.  To see this, first apply Cauchy's inequality with $\frac{\eta}{2}$ to the upper bound in Lemma \ref{CompactEst}:
$$
\left| \langle  w^{2\ell}  \kPiece g, g \rangle  \right|
 \le
 \frac{\eta}{2} \nsm w^{\ell} g\nsm_{L^2_{\gamma+2s - \delta }}^2 + C_\eta \nsm w^{\ell} g\nsm_{L^2_{\gamma+2s - \delta}}^2.
$$  
For the term  $C_\eta \nsm w^{\ell} g\nsm_{L^2_{\gamma+2s - \delta}}^2$ above, we split into $|v| \ge R$ and $|v|\le R$.  Choosing $R>0$ sufficiently large so that $C_\eta R^{-\delta} \le \frac{\eta}{2}$ proves \eqref{compactupper} subject only to Lemma \ref{CompactEst}.  Now Lemma \ref{CompactEst}  and other estimates will follow from \eqref{coerc1ineqPREP}  below.

\begin{proposition}
\label{upperBds}
For any function $\phi$ satisfying \eqref{rapidDECAYfcn}, we have the estimate
\begin{equation}
\label{coerc1ineqNORM}
\left| \ang{w^{2\ell}\Gamma_\beta (\phi,h),f} \right|
\lesssim
 \nsm  h \nsm_{\spaceELLn}  \nsm  f \nsm_{\spaceELLn}.
\end{equation}
For the next two estimates, we suppose that $\phi$ further satisfies \eqref{derivESTa}.  Then
\begin{equation}
\label{coerc1ineqPREP}
\left| \ang{w^{2\ell}\Gamma_\beta(g,\phi),f} \right|
 \lesssim 
 \nsm w^\ell g \nsm_{L^2_{\gamma + 2s-(n-1)}}  \nsm w^\ell f \nsm_{L^2_{\gamma + 2s-(n-1)}}.
\end{equation}
For any $m\ge 0$ we also have
\begin{equation}
\label{coerc1ineqPREP2}
\left| \ang{w^{2\ell}\Gamma_\beta (g,f),\phi} \right| \lesssim 
 \nsm  g \nsm_{L^2_{-m}}  \nsm  f \nsm_{L^2_{-m}}.
\end{equation}
Each of these estimates hold for any $\beta$, and any $\ell \in \mathbb{R}$.
\end{proposition}

Notice that these imply several other previously-stated estimates.  In particular, 
Lemma \ref{CompactEst} is an immediate consequence of \eqref{coerc1ineqPREP}  and
Pao's estimate of $\nu_{\kPiece}(v)$  in \eqref{compactpiece}.
Thus Proposition 
 \ref{upperBds} implies Lemma \ref{sharpLINEAR}, since also 
 \eqref{normupper} follows directly from \eqref{coerc1ineqNORM}.
 Other uses of Proposition 
 \ref{upperBds} will be seen below.

\begin{proof}[Proof of Proposition \ref{upperBds}]
To prove \eqref{coerc1ineqNORM}, we expand $\ang{w^{2\ell}\Gamma_\beta(\phi,h),f}$
as in \eqref{mainexpand}, and the proof follows the same lines as the proof of \eqref{nlineq}.
Following that proof,
 we estimate 
\eqref{nearsing} using the inequalities \eqref{tminusg} and \eqref{posUPPer3norm}.
Then 
\eqref{tstarg}, \eqref{posUPPer3norm}, and \eqref{tstarCf}
are 
 used to handle the terms \eqref{sum1} (note that, by the method of proof of these various inequalities, one may assume without loss of generality that $\phi \in H^{\ksob}_\ell$). 
For the cancellations,
\eqref{cancelsing} is handled by  \eqref{cancelf2}  and
\eqref{sum2} is controlled  using \eqref{cancelh2g2}.

To prove the estimate in \eqref{coerc1ineqPREP} we will use the inequality 
\eqref{lpsmooth}.  In particular 
\begin{multline*}
\left| \ang{w^{2\ell}\Gamma_\beta(g,\phi),f} \right| 
=   
\left|
\sum_{j=0}^\infty \sum_{k=-\infty}^\infty  
\left\{ T^{k,\ell}_{+}(g,\phi_j,f) - T^{k,\ell}_{-}(g,\phi_j,f) \right\}
\right| 
\\
\lesssim
 \nsm w^\ell g \nsm_{L^2_{\gamma + 2s-(n-1)}}  \nsm w^\ell f \nsm_{L^2_{\gamma + 2s-(n-1)}}
  \sum_{j=0}^\infty \sum_{k=-\infty}^\infty  \min\{2^{(2s-2)k} , 2^{2sk}  \} 2^{-2j}.
\end{multline*}
We have used
\eqref{tminushRAP} and 
\eqref{tminushRAPplus}
with Proposition \ref{compactCANCELe}. Specifically, in each of those estimates 
$\phi_j$ satisfies 
\eqref{derivESTa} with $C_\phi \lesssim 2^{-2j}$, as in \eqref{lpsmooth}.

The estimate for 
$
\ang{w^{2\ell}\Gamma_\beta(g,f),\phi}
$
in \eqref{coerc1ineqPREP2}
is proved
in 
exactly the same way
using instead
\eqref{tminushRAP},
\eqref{tminushRAPplus2},
and Proposition \ref{compactCANCELe}.
 In particular, 
Proposition \ref{upperBds} follows.  
\end{proof}

This concludes our compact estimates.

\section{The main coercive inequality}
\label{sec:mainCOER}

This section is devoted to the proof of Lemma \ref{estNORM3} when $\ell =0$.  
 Our approach involves  direct pointwise estimates of a Carleman representation in Section \ref{sec:pe}.  However this argument will not be completely sufficient, as explained below.  Thus in Section \ref{redistsec} we prove an estimate dubbed ``Fourier redistribution'' to finish the desired bound.  The essential idea is to appeal to the Fourier transform in the situation where the pointwise bound is not available.   Then in Section \ref{sec:funcN} we will establish functional analytic results on the space $\spacen$.  Finally in Section \ref{ssc:fce} we prove the remainder of the coercive estimates which were stated in Section \ref{mainESTsec}.

\subsection{Pointwise estimates}\label{sec:pe}
For any Schwartz function $f$, consider the quadratic difference expression arising in  $\ang{\nPiece f,f}$ from \eqref{normexpr} with $\ell = 0$.
By virtue of the Carleman-type change of variables, it is possible to express this semi-norm as
\eqref{semiCARLEMAN},
where the kernel can be computed with Proposition \ref{carlemanV2} in Appendix \ref{secAPP:HSr} to be
\begin{equation}
K(v,v')
\eqdef
2^{n-1}\int_{E_{v'}^{v}} d \pi_{v_*'} ~ \frac{M_* M_*'}{|v-v'| |v'-v_*'|^{n-2}} ~ B \left(2v - v' -v_*', \frac{v' - v_*'}{|v' - v_*'|} \right).  \label{kernel}
\end{equation}
The hyperplane
  $
  E^{v}_{v'} \eqdef 
  \set{ v'_*\in \threed}{    \ang{v' - v, v_*' -v} =0 }
  $
  is the integration domain and $d\pi_{v'_*} $ denotes the Lebesgue measure on $E^{v}_{v'}$.  Further $M_* = M(v'_*+ v' - v)$.

Our goal is to estimate this kernel $K$ pointwise from below and compare it to the corresponding kernel for the norm $\nsm  \cdot \nsm_{\spacen}$ from \eqref{normdef}; this, by virtue of \eqref{lpsobolev00},  allows control of our anisotropic Littlewood-Paley square function by 
$
\ang{\nPiece f,f}.
$  
We make this estimate when $|v-v'| \leq 1$ and $| |v|^2 - |v'|^2 | \leq |v-v'|$.  This constraint will require the introduction of a somewhat technical argument, but it is necessary since the required pointwise bound fails to hold uniformly outside this region.

On the hyperplane $E_{v'}^{v}$, we have 
$|2v - v' - v_*'| = |v' - v_*'|$; in particular, then
\[ \ang{\frac{2v - v' - v_*'}{|2v - v' - v_*'|}, \frac{v' - v_*'}{|v'-v_*'|}} = \frac{|v - v_*'|^2- |v - v'|^2}{|v' - v|^2 + |v - v_*'|^2}. \]
By virtue of the lower bound for $b(\cos \theta)$ in \eqref{kernelQ}, 
it follows that
\[ 
B \left(2v - v' -v_*', \frac{v' - v_*'}{|v' - v_*'|} \right) \gtrsim 
\Phi(|v'-v_*' |) ~
\frac{|v'-v_*'|^{n-1+2s}}{|v-v'|^{n-1+2s}} {\bf 1}_{|v-v_*'| > |v-v'|}. 
\]
The indicator function must be included because of the support condition in \eqref{kernelQ}.
Thus the kernel $K(v,v')$ from \eqref{kernel} is bounded below by a uniform constant times
\begin{equation}
|v-v'|^{-n-2s} \int_{E_{v'}^v} d \pi_{v_*'} ~ M_* M_*' \Phi(|v'-v_*'|)  |v'-v_*'|^{1+2s} {\bf 1}_{|v-v_*'| > |v-v'|}. 
\label{lower1}
\end{equation}
Next we consider the magnitude of the projections of $v_* = v'+v_*'-v$ and $v_*'$ in the direction of $v-v'$.  The orthogonality constraint $\ang{v-v',v-v_*'} = 0$ dictates that
\begin{align*}
\ang{v_*, \frac{v-v'}{|v-v'|}} = \ang{v', \frac{v-v'}{|v-v'|}}  & = \frac{-|v-v'|^2 + |v|^2 - |v'|^2}{2 |v-v'|},  
\\
\ang{v_*', \frac{v-v'}{|v-v'|}} = \ang{v , \frac{v-v'}{|v-v'|}}  & = \frac{ |v-v'|^2 + |v|^2 - |v'|^2}{2 |v-v'|}.
\end{align*}
With our assumptions $|v-v'| \leq 1$ and $||v|^2 - |v'|^2| \leq |v-v'|$, both right-hand sides are uniformly bounded by $1$ in magnitude, implying that $|v_*|^2 + |v_*'|^2 \leq 2 |w_*'|^2 + 1$, where 
$w_*'$ is the orthogonal projection of $v_*'$ onto the hyperplane through the origin with normal $v-v'$, e.g. 
$w_*' = v_*' - \frac{v-v'}{|v-v'|}\ang{\frac{v-v'}{|v-v'|},v_*'}$.  
This implies $M_* M_*' \gtrsim e^{-|w_*'|^2/4}$ uniformly.  Let $w'$ and $w$ be the orthogonal projections of $v'$ and $v$  respectively onto this same hyperplane through the origin with normal $v-v'$.  Trivially $|v' -v_*'| \geq | w' - w_*'|$.  Further  $|w' - v'| \leq 1$ since
\[ 
\left|\ang{v',v-v'} \right| = \frac{1}{2}\left||v|^2-|v'|^2- |v-v'|^2 \right| \leq |v-v'|.
\] 
Write $v_*' = w_*' + v - w$, then
we may parametrize the integral in \eqref{lower1} as an integral over $w_*'$ (with unit Jacobian) and thereby bound \eqref{lower1} uniformly from below as
\[ 
K(v, v') \gtrsim
|v-v'|^{-n-2s} \int_{E'} d \pi_{w_*'} ~ e^{- \frac{1}{4} |w_*'|^2}  |w' - w_*'|^{\gamma+2s+1}, 
\]
with
$
E' \eqdef \set{ w_*'}{ \ang{w_*',v-v'}=0, ~ |w' -w_*'| \geq 3} ; 
$
note that $|v' -v_*'|\approx |w' - w_*'|$ on this region because $0 \leq |v' -v_*'| - |w' -w_*'| \leq 2 \leq \frac{2}{3} |w- w_*'|$ since both $|v_*' - w_*'|$ and $|v' - w'|$ are less than one.

If $|w'| \le 4$,  
it is not hard to see that
$$
\int_{E'} d \pi_{w_*'} ~ e^{- \frac{1}{2} |w_*'|^2}  |w' - w_*'|^{\gamma+2s+1} \gtrsim 1.
$$
When 
$|w'| \ge 4$, 
we may restrict $w_*'$ to lie in the disk
$
\frac{1}{2} |w'| \ge  |w_*'| +1,
$
which implies in particular $|w' -w_*'| \approx |w'|$.  We thus have the following:
$$
\int_{E'} d \pi_{w_*'} e^{- \frac{1}{2} |w_*'|^2}  |w' - w_*'|^{\gamma+2s+1} \gtrsim
\ang{w'}^{\gamma+2s+1}\int_{0}^{\frac{1}{2} |w'| - 1}  d\rho ~ \rho e^{- \frac{1}{2} \rho^2} \gtrsim \ang{w'}^{\gamma+2s+1}.
$$
Since $|w' - v'| \leq 1$, 
the final, uniform estimate for \eqref{kernel} becomes:
\[ 
K(v, v') \gtrsim |v-v'|^{-n-2s} \ang{v'}^{\gamma+2s+1}
\ind_{|v-v'| \leq 1}  \ind_{||v|^2 - |v'|^2| \leq |v-v'|}.
\]
On this region $d(v,v') \lesssim |v-v'|$ and $\ang{v} \approx \ang{v'}$, so with \eqref{normexpr} 
we have uniformly 
\begin{equation} 
| f |_B^2
\gtrsim
\int_{\threed} dv \int_{\threed} dv' \frac{(f' - f)^2}{d(v, v')^{n+2s}} 
(\ang{v} \ang{v'})^{\frac{\gamma+2s+1}{2}} 
\ind_{d(v, v') \leq 1}  \ind_{||v|^2 - |v'|^2| \leq |v-v'|}. 
\label{coerciveptwise}
\end{equation}
 To obtain a favorable coercivity estimate from this, it would suffice to show that the expression \eqref{coerciveptwise} is bounded from below by the corresponding piece of \eqref{normdef}  (since the former expression has already been shown to be connected to our exotic Littlewood-Paley projections).    Because of the cutoff restricting 
 $
 \ind_{||v|^2 - |v'|^2| \leq |v-v'|}
 $
 a direct pointwise comparison is not possible to accomplish uniformly at all points.  This is not merely a limitation of the argument leading to \eqref{coerciveptwise}; in fact, a more involved analysis of \eqref{normexpr} and \eqref{kernel} shows that there is exponential
  decay of $K(v,v')$ in $|v-v'|$ when $v$ and $v'$ point in the same direction.  Thus there is an intrinsic obstruction to obtaining the correct coercive inequality by means of a simple, pointwise comparison of these expressions.

\subsection{Fourier redistribution}
\label{redistsec}
To get around this obstruction, we use the following trick. The key idea is already contained in the following proposition:

\begin{proposition}
Suppose $K_1$ and $K_2$ are even, nonnegative, measurable functions on $\threed$ satisfying \label{fourierdist}
\[ \int_{\threed} du ~ K_l(u) |u|^2 < \infty, \qquad l=1,2. \]
Suppose $\phi$ is any smooth, nonnegative function on $\threed$ and that there is some constant $C_\phi$ such that 
$|\nabla^2 \phi(u) | \le C_\phi $ 
for all $u$.
For $l=1,2$, consider the following quadratic forms (defined for arbitrary real-valued Schwartz functions $f$):
\[ \nsm f\nsm_{K_l}^2 \eqdef \int_{\threed} dv \int_{\threed} dv'~ \phi(v) \phi(v') K_l(v-v') (f(v) - f(v'))^2. \]
If there exists a finite, nonnegative constant $C$ such that, for all $\xi \in \threed$
\[ \int_{\threed} du~ K_1(u) |e^{2 \pi i \ang{\xi,u}} - 1|^2 \leq C + \int_{\threed} du~ K_2(u) |e^{2 \pi i \ang{\xi,u}} - 1|^2, \]
then for all Schwartz functions $f$, 
\[ \nsm f\nsm_{K_1}^2 \leq \nsm f\nsm_{K_2}^2 + C' C_{\phi} \int_{\threed} dv~ \phi(v) (f(v))^2, 
\]
where the constant $C'$ satisfies $C' \lesssim 1 + C+ \int_{\threed} du (K_1(u) + K_2(u)) |u|^2$ uniformly in $K_1, K_2$, $\phi$ and $C$.
\end{proposition}

\begin{proof}
We begin with the following identity:
\begin{align*}
 \phi(v) \phi(v') (f(v) - f(v'))^2 = \ & (\phi(v) f(v) - \phi(v') f(v'))^2 
  \ + \phi(v) (f(v))^2 (\phi(v') - \phi(v)) \\
& \ + \phi(v') (f(v'))^2 (\phi(v) - \phi(v')). 
\end{align*}
Multiply both sides by $K_l(v-v')$ and integrate with respect to $v$ and $v'$.  Exploiting symmetry, the result is:
\begin{align*}
 \int_{\threed} dv \int_{\threed} dv' & K_l(v-v') (f(v) - f(v'))^2 \phi(v) \phi(v') \\
& = \int_{\threed} dv \int_{\threed} dv' K_l(v-v') ( \phi(v) f(v) - \phi(v') f(v'))^2 \\
 & \ \ \ \ \ \  + 2 \int_{\threed} dv \phi(v) (f(v))^2 \ {p.v.} \! \! \int_{\threed} dv' K_l(v-v') (\phi(v') - \phi(v)).
\end{align*}
Now Taylor's theorem and the hypotheses on the second derivative of $\phi$ dictate 
\begin{align*}
 | \phi(v') - \phi(v) &  - \ang{v'-v, \nabla \phi(v) } | \leq \frac{1}{2} |v'-v|^2 C_\phi.
\end{align*}
If we define $C(K_l) \eqdef \int du K_l(u) |u|^2$, it follows that the difference
\begin{align*}
 \left| \int_{\threed} dv \int_{\threed} dv' \right. & K_l(v-v') (f(v) - f(v'))^2 \phi(v) \phi(v') \\
&  \left. - \int_{\threed} dv \int_{\threed} dv' K_l(v-v') ( \phi(v) f(v) - \phi(v') f(v'))^2 \right|, 
\end{align*}
is bounded above by $C(K_l) \int_{\threed} dv \phi (v) (f(v))^2$.  If we cutoff $|u| > \epsilon$, then clearly the Plancherel formula can be applied to the second term inside the absolute values above, with $F(v) \eqdef \phi(v) f(v)$, to get 
\begin{align}
 \int_{\threed} dv \int_{\threed} du &~ K_l(u) ( F(v+u) - F(v))^2 \ind_{|u| > \epsilon} \nonumber \\
& = \int_{\threed} d \xi \int_{\threed} du ~ K_l(u) |e^{2 \pi i \ang{\xi,u}}-1|^2 |\widehat{ F}(\xi)|^2 \ind_{|u| > \epsilon}. \label{plancherel}
\end{align}
Clearly the limiting case $\epsilon \rightarrow 0$ will hold as well because $|e^{2 \pi i \ang{\xi,u}}-1|^2$ vanishes to second order in $u$ and $\widehat{F}$ may be assumed to have arbitrarily rapid decay in $|\xi|$.  From here, the remainder is clear.  The hypotheses on $K_1$ and $K_2$ give that the limit of the Plancherel term \eqref{plancherel} is bounded above by the Plancherel term for $K_2$ plus $C$ times the $L^2$-norm of $\hat F$.  This term plus the errors in comparing the Plancherel pieces \eqref{plancherel} to the norms $\nsm \cdot\nsm_{K_l}^2$ give rise to the constant $C'$.
\end{proof}

Next, fix functions $K_1,K_2$ on $\R^{\last}$ given by $K_1(u) \eqdef |u|^{-\dim-2s} \ind_{|u| \leq 1}$ and 
$K_2(u) \eqdef |u|^{-\dim-2s} \ind_{|u| \leq 1} \ind_{|u_\last| \leq \epsilon |u|}$, 
that is, $K_1$ is restricted to the unit ball and $K_2$ is further restricted to 
the set 
$(1 - \epsilon^2) u_\last^2 \leq \epsilon^2 (u_1^2 + \cdots + u_\dim^2)$ with $\epsilon \in (0,1)$.  
Note that $\epsilon = 1/\sqrt{2}$ is the particular choice relevant to the coercive lower bound \eqref{coerciveptwise}.
We define the semi-norm $N_0$ by
\begin{align*}
\nsm f\nsm_{N_0}^2 & \eqdef \int_{\threed} dv \int_{\threed} dv' K_2(\ext{v} - \ext{v'}) (f' - f)^2 (\ang{v'}\ang{v})^{\frac{\gamma+2s+1}{2}}.
\end{align*}
Note that, if $K_2$ is replaced by $K_1$, the resulting expression is the derivative part of our main norm \eqref{normdef}.
By a pointwise comparison of $K_1$ and $K_2$, it is trivially true that $\nsm f\nsm_{N_0} \lesssim \nsm f\nsm_{N^{s,\gamma}}$, but our goal is to prove an inequality in the reverse direction.  To that end, let $\{ \phi \}$ be a smooth partition of unity on $\R^{\last}$ which is locally finite and satisfies uniform bounds for each $\phi$ and their first and second (Euclidean) derivatives.  Suppose furthermore that each $\phi$ is supported on a (Euclidean) ball of radius $\frac{\epsilon}{8}$ for a small $\epsilon >0$.  

Recall the notation from Section \ref{sec:aniLP}.
We restrict these functions to the paraboloid $\ext{v}=(v, \frac{1}{2} |v|^2)$ and insert them into the norms $\nsm \cdot\nsm_{N^{s,\gamma}}$ and $\nsm \cdot\nsm_{N_0}$:
\begin{align}
\int_{\threed} dv \int_{\threed} dv' K_l(\ext{v} - \ext{v'}) (f' - f)^2 \ang{v}^{\frac{\gamma+2s+1}{2}} \ang{v'}^{\frac{\gamma+2s+1}{2}} \phi(\ext{v}) \phi(\ext{v'}), \label{cutoffterms}
\end{align}
for $l=1,2$.  Suppose that $v_0 \in \threed$ satisfies $\phi(\ext{v_0}) \neq 0$ for some fixed $\phi$.  Make the change of variables $v \mapsto v_0 + \tau_{v_0} u$ and likewise for $v'$; including the Jacobian factor $\ang{v_0}^{-1}$ for each integral, the result is an integral over $u$ and $u'$ of the integrand
\begin{align*}
 \ang{v_0}^{-2} & K_l(\ext{v_0 + \tau_{v_0} u}   -  \ext{v_0 + \tau_{v_0} u'})   (f(v_0 + \tau_{v_0} u) - f(v_0 + \tau_{v_0} u'))^2 \\
& \times \phi(\ext{v_0 + \tau_{v_0} u}) \phi(\ext{v_0 + \tau_{v_0} u'}) (\ang{v_0 + \tau_{v_0} u} \ang{v_0 + \tau_{v_0} u'})^{\frac{\gamma+2s+1}{2}}. 
\end{align*}
Now we expand.  The argument of $K_l$, for example, becomes
\begin{align*}
 \ext{v_0 + \tau_{v_0} u} - \ext{v_0 + \tau_{v_0} u'} & = \ext{\tau}_{v_0} (u-u') + \frac{1}{2} (|u|^2 - |u'|^2)e_\last \\
& = \ext{\tau}_{v_0} (u-u') + \ang{u-u',\frac{u+u'}{2}} e_\last.
\end{align*}
Now $|\ext{v_0} - \ext{v_0 + \tau_{v_0} u}| \leq |u| - \frac{1}{2} |u|^2$; therefore $|\ext{v_0} - \ext{v_0 + \tau_{v_0} u}| \geq \frac{1}{2} |u|$ so long as $|u| \leq 1$.
Since the support of $\phi$ is in a ball of radius $\frac{\epsilon}{8}$, it follows that $| \frac{u+u'}{2}  | \leq \frac{\epsilon}{4}$, hence the magnitude of the coefficient of $e_\last$ above is at most $\frac{\epsilon}{4} |u-u'| = \frac{\epsilon}{4} |\ext{\tau}_{v_0} (u-u')|$, so 
$$
|(\ext{v_0 + \tau_{v_0} u }-\ext{v_0 + \tau_{v_0} u'}) - \ext{\tau}_{v_0} (u-u')| \leq \frac{\epsilon}{4}|\ext{\tau}_{v_0} (u-u')|.
$$  
In particular, for any $\epsilon \leq 2$, it must be the case that
\[ \frac{1}{2} |u-u'| \leq |\ext{v_0 + \tau_{v_0} u }-\ext{v_0 + \tau_{v_0} u'} | \leq \frac{3}{2} |u-u'| \]
on the support of the cutoff $\phi$.  In particular, this implies
\begin{align*}
 K_1(\ext{v_0+\tau_{v_0}u } & - \ext{v_0 + \tau_{v_0} u'}) \phi(\ext{v_0+\tau_{v_0}u}) \phi(\ext{v_0 + \tau_{v_0} u'}) \\
& \lesssim |u-u'|^{-n-2s} \ind_{|u-u'| \leq 2}   \phi(\ext{v_0 + \tau_{v_0} u})  \phi(\ext{v_0 + \tau_{v_0} u'}).
\end{align*}
Likewise, notice that 
\[ \left| \ang{\ext{v_0 + \tau_{v_0} u }-\ext{v_0 + \tau_{v_0} u'}, e_\last} - \ang{v_0}^{-1} \ang{v_0, u - u'} \right| \leq \frac{\epsilon}{4} |u-u'|,\]
so the condition $|\ang{v_0,u-u'}| \leq \frac{\epsilon}{4} |u-u'|$ guarantees
that
\[  \left| \ang{\ext{v_0 + \tau_{v_0} u }-\ext{v_0 + \tau_{v_0} u'}, e_\last} \right| \leq \frac{\epsilon}{2} |u-u'| \]
which is, in turn, at most $|\ext{v_0 + \tau_{v_0} u }-\ext{v_0 + \tau_{v_0} u'}|$.
Therefore we also have that
\begin{align*}
 |u-u'|^{-n-2s} & \ind_{|u-u'| \leq \frac{1}{2}} \ind_{\ang{v_0}^{-1} |\ang{v_0,u-u'}| \leq \frac{\epsilon}{4} |u-u'|} \phi(\ext{v_0 + \tau_{v_0} u})  \phi(\ext{v_0 + \tau_{v_0} u'}) \\
& \lesssim K_2(\ext{v_0+\tau_{v_0}u }  - \ext{v_0 + \tau_{v_0} u'}) \phi(\ext{v_0+\tau_{v_0}u}) \phi(\ext{v_0 + \tau_{v_0} u'}).
\end{align*}
To apply  Proposition \ref{fourierdist}, then, it suffices to check the Fourier condition and estimate the derivatives of the cutoff functions.
Clearly zeroth-order through second-order derivatives of 
\[ \ang{v_0}^{-1} \ang{v_0 + \tau_{v_0} u}^{\frac{\gamma+2s+1}{2}} \phi (\ext{v_0 + \tau_{v_0} u}), \]
with respect to $u$ will be uniformly bounded by $\ang{v_0}^{\frac{\gamma+2s-1}{2}}$ by virtue of the corresponding estimates for $\phi$ coupled with the fact that $\tau_{v_0}$ has norm $1$ as a mapping of Euclidean vector spaces and $\ext{\tau}_{v_0}$ is an isometry. 

Modulo the verification of the Fourier condition, then, we have
\begin{align*}
\int_{\threed} dv  \int_{\threed} dv' & K_2(\ext{v} - \ext{v'}) (f' - f)^2 \ang{v}^{\frac{\gamma+2s+1}{2}} \ang{v'}^{\frac{\gamma+2s+1}{2}} \phi(\ext{v}) \phi(\ext{v'}) \\
  + \int_{\R^n} du & ~ (f(v_0 + \tau_{v_0} u ))^2 \tilde \phi(u) \\
& \gtrsim
\int_{\threed} dv \int_{\threed} dv' K_1(\ext{v} - \ext{v'}) (f' - f)^2 \ang{v}^{\frac{\gamma+2s+1}{2}} \ang{v'}^{\frac{\gamma+2s+1}{2}} \phi(\ext{v}) \phi(\ext{v'}).
\end{align*}
Here $\tilde \phi \eqdef \ang{v}^{(\gamma+2s-1)/2} \phi$ (recall that the extra factor of $\ang{v_0}^{-1}$ comes from the change-of-variables we employed).
Thus the quadratic dependence on $\tilde \phi$ gives a factor of $\ang{v_0}$ to the power $\gamma+2s-1$; however an additional factor of $\ang{v_0}$ is obtained when the change-of-variables is reversed (that is, $v_0 + \tau_{v_0} u$ reverts back to $v$).  Thus, summing over the partition will give
\begin{align*} 
\int_{\R^n} dv \int_{\R^n} dv' & \frac{(f-f')^2}{d(v,v')^{n+2s}} \ang{v}^{\gamma+2s+1} {\mathbf 1}_{d(v,v') \leq 1} \sum_\phi \phi(\ext{v}) \phi(\ext{v'}) \\
& \lesssim \nsm f\nsm_{N_0}^2 + \int_{\threed} dv (f(v))^2 \ang{v}^{\gamma+2s}. 
\end{align*}
Now for any $v$, there must be an element of the partition on which $\phi(\ext{v}) \geq \frac{1}{N}$, where $N$ is the maximal number of partition elements which are nonzero at any particular point.  Since the partition was chosen so that there are uniform bounds on the first derivatives, it must be the case then, that there is a nonzero radius $\epsilon$ such that at any point $\ext{v}$, $\phi \geq \frac{1}{2N}$ on the ball centered at $\ext{v}$ with radius $\epsilon$.  Consequently $\sum_\phi \phi(\ext{v}) \phi(\ext{v'})$ is uniformly bounded below on a neighborhood of the diagonal, and we have
\begin{align*} 
\int_{\R^n} dv \int_{\R^n} dv' & \frac{(f-f')^2}{d(v,v')^{n+2s}} \ang{v}^{\gamma+2s+1} {\mathbf 1}_{d(v,v') \leq \epsilon}  \lesssim \nsm f\nsm_{N_0}^2 + \int_{\threed} dv (f(v))^2 \ang{v}^{\gamma+2s}. 
\end{align*}
Notice also that
\begin{align*} 
\int_{\R^n} dv \int_{\R^n} dv' & \frac{(f-f')^2}{d(v,v')^{n+2s}} \ang{v}^{\gamma+2s+1} {\mathbf 1}_{\epsilon \leq d(v,v') \leq 1}  \\
 & \leq \int_{\R^n} dv \int_{\R^n} dv' \frac{2 f^2 + 2 f'^2}{d(v,v')^{n+2s}} \ang{v}^{\gamma+2s+1} {\mathbf 1}_{\epsilon \leq d(v,v') \leq 1} \\
 & \lesssim \int_{\R^n} dv ~ (f(v))^2 \ang{v}^{\gamma+2s}
\end{align*}
since 
$
\int_{\R^n} dv' \frac{{\mathbf 1}_{\epsilon \leq d(v,v') \leq 1}}{d(v,v')^{n+2s}} \ang{v}^{\gamma+2s+1}  \lesssim 
\ang{v}^{\gamma+2s}
$ 
for fixed $\epsilon$ (and likewise with the roles of $v$ and $v'$ reversed).

To complete the comparison, then, it suffices to make the following estimate:
\begin{proposition}
Fix any $\epsilon > 0$, and let $E_1$ and $E_2$ be the sets in $\threed$ given by   $E_1 \eqdef \set{ u \in \threed}{ |u| \leq 2}$ and $E_2 \eqdef \set{u \in \threed}{ |u| \leq \frac{1}{2} \mbox{ and } |u_\dim| \leq \epsilon |u|}$.  Then
\begin{equation}
 \int_{E_1} du ~ |e^{2 \pi i \ang{\xi,u}} - 1|^2 |u|^{-\dim-2s} 
 \lesssim 
 1 + \int_{E_2} du ~ |e^{2 \pi i \ang{\xi,u}} - 1|^2 |u|^{-\dim-2s}, \label{plancherelcheck}
\end{equation}
uniformly for all $\xi \in \threed$.
\end{proposition}
\begin{proof}
Writing both sides in polar coordinates, we see that each side may be realized as an integral over the unit sphere $\sph$ of
\[ \int_{\tilde{E}_l} d \sigma  ~ \Psi(\ang{\xi,\sigma}), \]
where $\tilde{E}_1 = \sph$, $\tilde{E}_2$ is a small band near the equator, and $\Psi(\lambda)$ is 
of the form
\[ 
\Psi(\lambda) \eqdef \int_0^a dt ~ |e^{2 \pi i \lambda t} - 1|^2 t^{-1-2s}, 
\]
for some appropriate value of $a$ (
$a=2$ or $a = \frac{1}{2}$).
From the elementary inequalities
\begin{align*}
 \int_0^{(2 \lambda)^{-1}} dt |e^{2 \pi i \lambda t} - 1|^2 t^{-1-2s}  & \approx \int_0^{(2 \lambda)^{-1}} dt \lambda^2 t^2 t^{-1-2s} \approx \lambda^{2s} \\
\int_{(2 \lambda)^{-1}}^\infty dt |e^{2 \pi i \lambda t} - 1|^2 t^{-1-2s} & \lesssim \int_{(2 \lambda)^{-1}}^\infty dt \ t^{-1-2s} \approx \lambda^{2s} ,
\end{align*}
it follows that the integrands will be comparable to $|\ang{\xi,\sigma}|^{2s}$ when this quantity is bounded below by a fixed constant and less than a constant times $|\ang{\xi,\sigma}|^{2s}$ regardless of whether or not this quantity is bounded below.  For any $\xi$ with $|\xi| \geq 1$, then, at least a positive measure region of $\sph$ will have $|\ang{\xi,\sigma}| \gtrsim |\xi|$ (whether in $\tilde{E}_1$ or $\tilde{E_2}$), so both sides of \eqref{plancherelcheck} will be comparable to $|\xi|^{2s}$, which is sufficient for the inequality \eqref{plancherelcheck} to hold.
\end{proof}

The proof of the coercive inequality is now complete, for we demonstrated that
\[ 
\int_{\threed} dv \int_{\threed} dv_* \int_{\sph} d \sigma ~ B (f'-f)^2 M_*'M_* \gtrsim \nsm f\nsm_{N_0}^2, 
\]
by direct pointwise comparison and that $\nsm f\nsm_{N_0}^2 + \nsm f\nsm_{L^2_{\gamma+2s}}^2 \gtrsim \nsm f\nsm_{N^{s,\gamma}}$ by Fourier redistribution. The combination of these inequalities gives Lemma \ref{estNORM3} when $\ell =0$.

\subsection{Regarding the functional analysis of $N^{s,\gamma}$}
\label{sec:funcN}
  An important consequence of the analysis of the previous section is that we have an alternate characterization of the space $\spacen$ in terms of the usual Sobolev spaces.  In particular, let $\{ \phi_i \}$ be a partition of unity constructed as above by restricting a smooth, locally finite partition of unity on $\R^{\last}$ (such that each $\phi_i$ has support in a ball of unit radius) to the paraboloid $(v, \frac{1}{2} |v|^2)$.  For each $\phi_i$ in the partition, let $v_i$ be some point in its support.  If we define
\[ f_i(u) \eqdef \phi_i(\ext{v_i + \tau_{v_i} u}) f(v_i + \tau_{v_i} u), \]
it follows that we have the comparison
\begin{equation} 
|f|_{\spacen}^2 \approx \sum_{i=1}^\infty \ang{v_i}^{\gamma+2s-1} |f_i|_{H^s}^2, \label{isosobolev}
\end{equation}
where $H^s$ is the usual ($n$-dimensional) $L^2(\threed)$-Sobolev space.  This result is true by virtue of the fact that
\[ 
|f|_{H^s}^2 \approx |f|_{L^2}^2 + \int_{\threed} dv \int_{\threed} dv' \frac{(f(v')- f(v))^2}{|v-v'|^{\dim+2s}} {\mathbf 1}_{|v-v'| \leq 1}, 
\]
which follows itself by an application of the Plancherel theorem as in Proposition \ref{fourierdist} together with the asymptotic estimates for the integrals \eqref{plancherelcheck}.

In particular, if $F_i \eqdef  \phi_i ( \ext{v}) f(v)$, then $|\hat f_i(\xi)| = |\ang{v_i} \hat F_i( \tau_{v_i}^{-1} \xi)|$, so by Plancherel and the change of variables $\xi \mapsto \tau_{v_i} \xi$, 
we have
\[ |f_i|_{H^s}^2 \approx \ang{v_i} \int_{\R^n} d \xi ~ (1 + | \tau_{v_i} \xi|)^{2s} |\hat F_i (\xi)|^2; \]
now $ \ang{v_i}^{-2s} (1 + |\xi|)^{2s} \lesssim (1 + | \tau_{v_i} \xi|)^{2s} \lesssim (1 + |\xi|)^{2s}$, which provides the comparison \eqref{isocomp}; simply observe that
\[ \sum_{i=1}^\infty \ang{v_i}^{\gamma} |F_i|_{H^s}^2 \lesssim |f|_{N^{s,\gamma}}^2 \lesssim \sum_{i=1}^\infty \ang{v_i}^{\gamma+2s} |F_i|_{H^s}^2. \]
Now sum the partition of unity to compare the left- and right- hand sides to $|f|_{H^s_{\gamma}}$ and $|f|_{H^s_{\gamma+2s}}$, respectively.

With the aid of \eqref{isosobolev}, a number of elementary functional analysis properties of $\spacen$ reduce to the situation of the standard Sobolev spaces.  For example, it is a simple exercise to show that Schwartz functions are dense in $N^{s,\gamma}$ by exploiting this same fact for the space $H^s$, approximating $f_i$ individually in $H^s$, and summing over the partition (note that this requires the elements of the partition $\phi_i$ themselves to be Schwartz functions, but this additional restriction is not a problem to satisfy).

\subsection{Further coercive estimates}\label{ssc:fce}
   In this sub-section we will prove the coercive interpolation inequalities in
 \eqref{coerc1ineq}
and
 \eqref{coerc2ineq} from Lemma \ref{estNORM3} and Lemma \ref{DerCoerIneq}.  Actually,  \eqref{coerc2ineq} is a trivial consequence of Lemma \ref{estNORM3} and \eqref{compactupper}
 because 
$
\langle w^{2\ell} Lf, f \rangle = \langle w^{2\ell} \nPiece f, f \rangle + \langle w^{2\ell} \kPiece f, f \rangle.
$  
Thus we will restrict attention to \eqref{coerc1ineq}  and
 Lemma \ref{estNORM3}.

\begin{proof}[Proof of Lemma \ref{estNORM3}]
Firstly,  Lemma \ref{estNORM3} for $\ell = 0$  was proven in Sections \ref{sec:pe} and \ref{redistsec}.  We focus here on estimating the norm piece when $\ell \ne 0$.  We expand
\begin{gather*}
\langle w^{2\ell} \nPiece g, g \rangle 
=  
| g |_{B_{\ell}}^2 
+
\int_{\mathbb{R}^n} dv ~ w^{2\ell}(v)  \nu(v) ~ |g(v)|^2
+ J,
\end{gather*}
where $| g |_{B_{\ell}}$ is defined in \eqref{normexpr}.
 Furthermore, we have 
  the following equivalence 
  $$ 
  | g |_{B_{\ell}}^2+ | w^\ell g |_{L^2_{\gamma + 2s}}^2 \approx | g |_{\spaceELLn}^2.
  $$
  The lower bound $\gtrsim$ of this equivalence follows directly from the proof in Section \ref{redistsec} after the introduction of the additional weight $w^{2 \ell}(v)$ (note that the arguments contained in Section \ref{redistsec} did not depend on the value of $\gamma+2s$ so this extra weight is trivial).    The upper bound, $\lesssim$, follows from the estimates for $\ang{w^{2\ell}\Gamma (M,g),g}$ in \eqref{coerc1ineqNORM}.  The error term $J$ takes the form
\begin{gather*}
  J \eqdef
  \frac{1}{2} \int_{\mathbb{R}^n} dv  \int_{\mathbb{R}^n} dv_* \int_{\sph} d \sigma ~ B~ (g'-g)g\left( w^{2\ell}(v') - w^{2\ell}(v) \right) M_*' M_* .
\end{gather*}
These expressions are derived exactly as in the computations preceding \eqref{normpiece}.

We will show that this error term $J$ is lower order via an expansion of the kernel.  
In particular we {claim} that there exists an $\epsilon>0$ such that
$$
\left| J \right| \lesssim | g |_{B_{\ell}}
|w^\ell g|_{L^2_{\gamma + 2s - \epsilon}}
\le \eta | g |_{B_{\ell}}^2 + \eta' |w^\ell g|_{L^2_{\gamma + 2s}}^2 + C |w^\ell g|_{L^2(B_C)}^2.
$$
This argument follows as in procedure which is explained below Lemma \ref{CompactEst}.  This estimate easily implies Lemma \ref{estNORM3}.   

To prove this {claim}, notice that Cauchy-Schwartz gives us
$$
\left| J \right| \lesssim 
| g |_{B_{\ell}}
\left(
\int_{\mathbb{R}^n} dv  \int_{\mathbb{R}^n} dv_* \int_{\sph} d \sigma B |g|^2
\frac{\left( w^{2\ell}(v') - w^{2\ell}(v) \right)^2}{w^{2\ell}(v)} M_*' M_*
\right)^{1/2}.
$$
We will show in particular that there exists an $\epsilon>0$ such that
\begin{equation}
 \int_{\mathbb{R}^n} dv_* \int_{\sph} d \sigma  B(v - v_*, \sigma) 
\frac{\left( w^{2\ell}(v') - w^{2\ell}(v) \right)^2}{w^{2\ell}(v)}   M_*' M_*
\lesssim 
w^{2\ell}(v)
\ang{v}^{\gamma + 2s - \epsilon},
\label{claimEST1}
\end{equation}
and this will establish the {claim}.

To obtain \eqref{claimEST1}, first a simple Taylor expansion yields
$$
w^{2\ell}(v') - w^{2\ell}(v) 
=
(v' - v)\cdot (\nabla w^{2\ell})(\TaylorP(\tau)), 
\quad
\exists \tau \in [0,1],
$$
where $\TaylorP(\tau) = v + \tau(v' - v)$.  Since $|v' - v| = |v'_* - v_*|$
we have
$$
\ang{\TaylorP(\tau)} \lesssim \ang{v} \ang{v' - v} \lesssim \ang{v} \ang{v'_* - v_*},
$$
and similarly,
$
\ang{\TaylorP(\tau)}^{-1} 
 \lesssim \ang{v}^{-1}   \ang{v'_* - v_*}. 
$
Thus generally
$$
\frac{\left( w^{2\ell}(v') - w^{2\ell}(v) \right)^2}{w^{2\ell}(v)} ~  M_*' M_*
\lesssim
|v' - v|^2
w^{2\ell}(v)\ang{v}^{-2} \sqrt{M_*' M_*}.
$$
Now we split $|v' - v|^2 = |v' - v|^{2s + \delta} |v'_* - v_*|^{2 -2s - \delta}$ for any $\delta \in (0, 2-2s)$.  We can expand
$
|v' - v|^{2s + \delta} = |v - v_*|^{2s + \delta}\left( \sin \frac{\theta}{2} \right)^{2s + \delta},
$
and then we clearly have
$$
\int_{\sph} d \sigma ~ B(v - v_*, \sigma) ~ |v' - v|^{2s + \delta}\lesssim |v - v_*|^{\gamma+2s + \delta}.
$$
This follows directly from \eqref{kernelQ} - \eqref{kernelPsing}. Furthermore, 
$
|v'_* - v_*|^{2 -2s - \delta}
(M_*' M_*)^{1/4}
\lesssim 1.
$
Putting all of this together, we see that \eqref{claimEST1} holds with $\epsilon = 2- \delta >0$.
 \end{proof}
 
With the help of our coercive estimate \eqref{coerc2ineq} with no derivatives, in the following we will prove the main coercive estimate with high derivatives.

\begin{proof}[Proof of \eqref{coerc1ineq}]
We use the formula for $Lg$ from \eqref{LinGam}.  As in \eqref{DerivEstG}, we expand
\begin{multline*}
\partial^{\alpha}_{\beta} Lg =
 L\left( \partial^{\alpha}_{\beta} g \right)
 - 
  \sum_{\beta_1 + \beta_2  = \beta, ~|\beta_1|< |\beta|}  C^{\beta}_{\beta_1, \beta_2} ~ 
\Gamma_{\beta_2} (\partial_{\beta-\beta_1} M,\partial^{\alpha}_{\beta_1} g) 
\\
 - 
  \sum_{\beta_1 + \beta_2  = \beta, ~|\beta_1|< |\beta|}  C^{\beta}_{\beta_1, \beta_2} ~ 
\Gamma_{\beta_2}(\partial^{\alpha}_{\beta_1} g, \partial_{\beta-\beta_1}M).
\end{multline*}
After multiplying by $w^{2\ell - 2|\beta|}\partial^{\alpha}_{\beta} g$, and integrating over $\mathbb{R}^n$ we can estimate the term
$
\langle w^{2\ell - 2|\beta|} L \left( \partial^{\alpha}_{\beta} g \right), \partial^{\alpha}_{\beta} g \rangle
$
 as in \eqref{coerc2ineq}.  For the error term that arises, i.e., $| \partial_\beta^\alpha  g |_{L^2(B_R)}$ for some $R>0$ we use the compact interpolation for any small $\delta >0$:
$$
| \partial_\beta^\alpha  g |_{L^2(B_R)} 
\le 
\eta  | \partial_\beta^\alpha  g |_{H^\delta(B_R)}
+
C   | \partial^\alpha g |_{L^2(B_{C})}.
$$
Here $C>0$ is some large constant, and $\eta>0$ is any small number.  Further  $| \partial_\beta^\alpha  g |_{H^\delta(B_R)} \lesssim \nsm \partial^{\alpha}_{\beta}g\nsm_{N^{s,\gamma}_{\ell - |\beta|}}$ when $\delta < s$; this holds because the non-isotropy of the norm $\spacen$ only comes into play near infinity. More precisely, if $v$ and $v'$ are confined to the Euclidean ball of radius $R$ at the origin, then $d(v,v') \approx |v-v'|$ (with constants depending on $R$), and so on this region the expression for \eqref{normdef} is comparable to the Gagliardo-type semi-norm for the space $H^s(B_R)$.
This gives the estimate for the inner product of $L\left( \partial^{\alpha}_{\beta} g \right)$.

We estimate the rest of the terms using \eqref{coerc1ineqNORM} and \eqref{coerc1ineqPREP}.  We use the extra velocity decay which is left over from these estimates to split into a large unbounded region times a small constant and a bounded region with a large constant.  On the bounded region, we use the compact interpolation as just used in the last case to put all of the velocity derivatives into slightly larger Sobolev norm multiplied by an arbitrarily small constant.  This is all that is needed to finish the estimate.
 \end{proof}

\section{Decoupled space-time estimates and global existence} 
\label{sec:deBEest}

In this last section, we show that the sharp estimates proved in the previous sections can be applied to the modern technology from the linearized cut-off Boltzmann theory to establish global existence.   This works precisely because of the specific structure of the interactions between the velocity variables and the space-time variables.  The methodology that we employ essentially de-couples the  required space-time estimates that are needed from the new fractional and anisotropic derivative estimates which are shown in the previous sections.

The method that we choose to utilize in this section goes back to Guo \cite{MR2000470}.  A key point of this approach is to derive a system of space-time ``macroscopic equations,'' see \eqref{c} through \eqref{adot} below, which have certain elliptic and hyperbolic structures. This structure can be used to prove an instantaneous coercive lower bound for the linear operator $L$, for solutions to the full non-linear equation \eqref{Boltz}, in our new anisotropic norm \eqref{normdef}.
This original method \cite{MR2000470} used high order temporal derivatives, which we could also utilize.  
But as a result of advances in  \cite{MR2095473}, \cite{Jang2009VMB,MR2420519} the need for temporal derivatives was removed.  The key point here is to use both the macroscopic equations \eqref{c} through \eqref{adot} and the conservation  laws \eqref{cl.0} through \eqref{cl.2} to remove the need to estimate time derivatives with an ``interaction functional'' that is comparable to the energy.  We point the reader's attention to the general abstract framework of  \cite{villani-2006}  also
in this direction.

We will initially discuss the coercivity of the linearized collision operator, $L$.  
With the null space \eqref{null}, and the projection \eqref{hydro}, we 
 decompose  $f(t,x,v)$ as
\begin{equation}
f={\bf P}f+\{{\bf I-P}\}f.
\notag
\end{equation}
We next prove a sharp constructive lower bound for the linearized collision operator.  

\begin{theorem}
\label{lowerN}  
There is a constructive constant $\delta_0>0$ such that
\begin{equation*}
\langle L g, g \rangle \ge \delta_0 | \{ {\bf I - P } \} g |_{N^{s,\gamma}}^2.
\end{equation*}
\end{theorem}

This coercive lower bound is proved with our new constructive compact estimates from Lemma \ref{sharpLINEAR} and Lemma \ref{estNORM3} 
when used in conjunction with
the non-sharp but constructive bound from Mouhot \cite{MR2254617} for the non-derivative part of the norm.

\begin{proof}[Proof of Theorem \ref{lowerN}]  
Suppose $g = \{ {\bf I - P } \} g$.     From \eqref{coerc2ineq}, for some small $\eta >0$
$$
\langle L g, g \rangle
\ge 
\eta |  g |_{N^{s,\gamma}}^2
- C |  g |_{L^{2}(B_C)}^2, 
\quad \exists C \ge 0.
$$ 
The positive constant $C$  is explicitly computable.  From \cite{MR2254617}, it is known that under our assumptions 
$$
\langle L g, g \rangle
\ge 
\delta_1  |  g |_{L^{2}_\gamma}^2.
$$ 
Here $\delta_1>0$ is an explicitly computable constant, and $\gamma$ is from \eqref{kernelP} and \eqref{kernelPsing}.  

Lastly, for any $\delta \in (0,1)$ we employ the splitting
$$
\langle L g, g \rangle
=
\delta \langle L g, g \rangle
+
(1-\delta)\langle L g, g \rangle
\ge 
\delta\eta |  g |_{N^{s,\gamma}}^2
- \delta C |  g |_{L^{2}(B_C)}^2
+
(1-\delta)\delta_1  |  g |_{L^{2}_\gamma}^2.
$$ 
Since $C$ is finite notice that  $|  g |_{L^{2}(B_C)}^2 \lesssim |  g |_{L^{2}_\gamma}^2$ for any $\gamma \in\mathbb{R}$.
Thus the lemma follows by choosing $\delta>0$ sufficiently small so that the last two terms are $\ge 0$.
\end{proof}

\subsection{Local Existence}  
Given the estimates that we have proved (in Section \ref{mainESTsec}), the 
local existence results for small data that we will prove in this section are rather standard; see e.g. \cite{MR2679369,MR2259206,MR839310,MR882376,MR2000470,MR2013332,MR1946444}.
Our local existence proof for \eqref{Boltz} is based on a uniform energy estimate for
an iterated sequence of approximate solutions.   The iteration starts at $f^0(t,x,v)\equiv 0$.  We   solve for $f^{\iter+1}(t,x,v)$ such that 
\begin{equation}
\left( \partial _t+v\cdot \nabla _x+\nPiece \right) f^{\iter+1}+\kPiece f^\iter
=
\Gamma (f^{\iter}, f^{\iter+1}), ~ f^{\iter+1}(0,x,v)=f_0(x,v).  \label{approximate}
\end{equation}
It is standard to show the linear equation \eqref{approximate} admits smooth solutions with the same regularity in $\HARDspaceh$ for \eqref{kernelP} (or $H^K_\ell$ in the case of \eqref{kernelPsing}) as a given  smooth small initial data, and also has a gain of $L^2((0,T); \hardNspace )$ for \eqref{kernelP} (or $L^2((0,T); \nspace )$ in the case of \eqref{kernelPsing}).  This does not create difficulties and can be proved with our estimates.  We explain herein how to establish the {\it a priori} estimates necessary to find a local classical solution in the limit as $\iter\to\infty$.

For notational convenience during  the proof we define the ``dissipation rate'' as
$$
\notag
\mathcal{D}(f(t))\eqdef 
\left\{
\begin{array}{cc}
\| f(t)\|_{\hardNspace}^2,
& \text{for the hard potentials: } \eqref{kernelP},
\\
\| f(t)\|_{\nspace}^2,
&  \text{for the soft potentials: }  \eqref{kernelPsing}.
\end{array}
\right.
$$
We will also use the following total norm
\begin{gather}
\mathcal{G}(f(t))
\eqdef
\|f(t)\|^2_{\spaceU}+ \int_0^t d\tau ~\mathcal{D}(f(\tau)).
\label{totalG}
\end{gather} 
Here the unified norm $\|\cdot \|_{\spaceU}$ is defined in \eqref{unifiedHspace}.  We will furthermore abuse notation by writing $\ang{\nPiece g, g}$ as $|g|_{\spacen}$, etc, below.

Our goal will be to obtain a uniform estimate for the iteration on 
a small time interval. The crucial energy estimate is as follows:

\begin{lemma}
\label{uniform}  The sequence \{$f^\iter(t,x,v)\}$ is well-defined. There exists a short time
$
T^{*} = T^{*}(\|f_0\|^2_{\spaceU})>0,
$
such that for $\|f_0\|^2_{\spaceU}$ sufficiently small, there is a uniform constant $C_0>0$ such that
\begin{equation}
\sup_{\iter\ge 0} ~
\sup_{0\le \tau \le T^{*}} ~ 
\mathcal{G}(f^{\iter}(\tau))
\le
2C_0
\|f_0\|^2_{\spaceU}.  \label{uni}
\end{equation}
\end{lemma}

\begin{proof}
We write down the proof in the case of hard potentials \eqref{kernelP} so that $\spaceU=\HARDspaceh$ with $\ell =\HARDvDER=0$.  The general case when $0\le \HARDvDER \le \HARDxDER$, $\ell \ge 0$, and also the soft potential case $\spaceU = H^K_\ell$ with \eqref{kernelPsing} can be proved in the directly analogous way; as in e.g. \eqref{MAINeINEQ}.  The proof proceeds via induction over $k$. Clearly $k=0$ is true.  We
assume that (\ref{uni}) is valid for $k=\iter$.   For a given $f^{\iter}$, there exists a solution $f^{\iter+1}$ to the linear equation \eqref{approximate} with small data.  
We focus here on the proof of (\ref{uni}).

Take the spatial derivatives $\partial ^\alpha $ of (\ref{approximate}) to obtain
\begin{equation}
\left( \partial _t+v\cdot \nabla _x\right)\partial ^\alpha f^{\iter+1}
+\nPiece \left(\partial^\alpha f^{\iter+1}\right)
+\kPiece \left(\partial ^\alpha f^\iter\right)
=
\partial ^\alpha \Gamma \left(f^\iter, f^{\iter+1}\right).  \label{xderi}
\end{equation}
Therefore, applying the trilinear estimate in Lemma \ref{NonLinEstLOW} 
yields 
\begin{multline*}
\frac 12
\frac d{dt}\|\partial^\alpha f^{\iter+1}\|_{L^2_vL^2_x}^2
+\| \partial^\alpha f^{\iter+1} \|_{\spacen}^2
+(\kPiece \left(\partial ^\alpha f^\iter\right),\partial ^\alpha
f^{\iter+1})
\\
=(\partial^\alpha \Gamma \left(f^{\iter}, f^{\iter+1}\right),\partial ^\alpha f^{\iter+1})
\lesssim
\|f^\iter\|_{ \HARDspaceh }
 \|f^{\iter+1}\|_{ \hardNspace }^2.
\end{multline*}
Then integrating the above over $[0,t]$ we obtain
\begin{gather}
\frac 12\|\partial ^\alpha f^{\iter+1}(t)\|_{L^2_vL^2_x}^2
+
\int_0^t~d\tau~ \| \partial^\alpha f^{\iter+1}(\tau) \|_{\spacen }^2
+
\int_0^t ~d\tau~ (\kPiece\left(\partial ^\alpha f^\iter \right) ,\partial ^\alpha f^{\iter+1})  
\notag
\\
\label{xterm} 
\le \frac 12 \|\partial ^\alpha f_0\|_{L^2_vL^2_x}^2
+
C \int_0^t ~d\tau~
\|f^\iter\|_{\HARDspaceh }
 \|f^{\iter+1}\|_{\hardNspace }^2(\tau).
\end{gather}
We notice that from Lemma \ref{CompactEst} applied to \eqref{compactpiece}, for any $\eta>0$ small and $\eta' = 1/2$, 
\begin{gather*}
\left| \int_0^t d\tau (\kPiece (\partial ^\alpha f^\iter),\partial ^\alpha f^{\iter+1}) 
\right|
\le \int_0^t d\tau
\left( \frac 12 \|\partial ^\alpha f^{\iter+1}(\tau)\|_{L^2_{\gamma + 2s}}^2
+
C\|\partial ^\alpha f^{\iter+1}(\tau)\|_{L^2}^2 \right)
\\
+\eta \int_0^t ~ d\tau ~\|\partial ^\alpha f^\iter (\tau)\|_{L^2_{\gamma + 2s}}^2
+C_\eta \int_0^t ~ d\tau ~ \| \partial ^\alpha f^\iter(\tau)\|_{L^2}^2.
\end{gather*}
We incorporate this inequality into \eqref{xterm} and sum over $|\alpha | \le \HARDxDER$ to obtain 
\begin{multline}
\mathcal{ G}(f^{\iter+1}(t)) 
\le
C_0\| f_0 \|_{\spaceU}^2
+
\int_0^t ~ d\tau ~ \left\{ C\| f^{\iter+1}\|^2_{ \spaceU }(\tau)
+
C\eta \| f^\iter\|_{H^{\HARDxDER}_{x}L^2_{v,\gamma + 2s}}^2(\tau) \right\}
\\
+
 C_\eta \left( \int_0^t ~ d\tau ~ \| f^\iter\|^2_{ \spaceU  }(\tau) \right)
+
C\sup_{0\le \tau\le t}\mathcal{ G}(f^{\iter+1}(\tau))\sup_{0\le \tau\le t}\mathcal{ G}^{1/2}(f^\iter(\tau))  
\\
\le 
C_0\| f_0 \|_{ \spaceU }^2+C_\eta  t\left\{ \sup_{0\le \tau\le t}\mathcal{ G}(f^{\iter+1}(\tau))+
\sup_{0\le \tau\le t}\mathcal{ G}(f^\iter(\tau))\right\} 
\nonumber 
\\
+C\eta \sup_{0\le \tau\le t}\mathcal{ G}(f^\iter(\tau))  
+C\sup_{0\le \tau\le t}\mathcal{ G}(f^{\iter+1}(\tau))\sup_{0\le \tau\le t}\mathcal{ G}^{1/2}(f^\iter(\tau)).  
\nonumber
\end{multline}
We are using the total norm from \eqref{totalG}.
By the induction hypothesis \eqref{uni}
$$
\sup_{0\le \tau \le t}\mathcal{ G}(f^\iter(\tau))\le 2C_0\| f_0 \|_{\spaceU}^2.
$$ 
Then we collect terms in the previous inequality to obtain
$$
\left\{ 1-C_\eta T^{*}-C\| f_0 \|_{\spaceU}\right\}\sup_{0\le t\le T^{*}}\mathcal{ G}%
(f^{\iter+1}(t))
\\
\le
 \left\{C_0+C\eta +2C_\eta T^{*}C_0\right\}\| f_0 \|_{ \spaceU }^2. 
$$
By choosing $\eta $ small, then choosing $T^{*}=T^{*}(\| f_0 \|_{ \spaceU })$ small, we have 
\[
\sup_{0\le t\le T^{*}}\mathcal{ G}(f^{\iter+1}(t))\le 2C_0\| f_0 \|_{ \spaceU }^2.
\]
We therefore conclude Lemma \ref{uniform} if $T^{*}$ and $\| f_0 \|_{ \spaceU }^2$ are sufficiently small. 
\end{proof}

With our uniform control over the iteration from \eqref{approximate} proved in Lemma \ref{uniform}, we can now prove local existence in the following theorem.

\begin{theorem}(Local existence)
\label{local}
For any sufficiently small $M_0>0,$ there exists a time  $T^{*} = T^{*}(M_0)>0$ and $%
M_1>0,$ such that if 
\[
\| f_0 \|_{\spaceU}^2\le M_1, 
\]
then there is a unique  solution $f(t,x,v)$ to \eqref{Boltz} on  
$[0,T^{*})\times \domain\times \threed$ such that 
\[
\sup_{0\le t\le T^{*}}\mathcal{ G}(f(t))\le M_0. 
\]
Furthermore $\mathcal{ G}(f(t))$ is continuous over $[0,T^{*}).$ Lastly, we have positivity in the sense that
 if
 $F_0(x,v)=\mu +\mu
^{1/2}f_0\ge 0,$ then 
$
F(t,x,v)=\mu +\mu ^{1/2}f(t,x,v)\ge 0. 
$
\end{theorem}

\begin{proof}By taking $\iter\rightarrow \infty ,$ we have sufficient compactness from Lemma \ref{uniform} to obtain a strong solution $f(t,x,v)$ to the Boltzmann equation \eqref{Boltz} locally in time.

To prove the uniqueness, we suppose that there exists another solution $g$ with the same initial data satisfying 
$\sup_{0\le \tau \le T^{*}}\mathcal{ G}(g(\tau))\le M_0.$ The difference $f-g$
satisfies 
\begin{equation}
\{\partial _t+v\cdot \nabla _x\}\left(f-g\right)+L\left(f-g\right)=\Gamma \left(f-g,f\right)+\Gamma \left(g,f-g\right).
\label{difference}
\end{equation}
We suppose without loss of generality that we are in the case of hard potentials \eqref{kernelP}.  We apply Theorem \ref{TriLinEst} and the embedding  $H^{\ksob}(\domain) \subset L^\infty(\domain) $ to obtain
\begin{multline*}
\left| \left( \left\{\Gamma\left(f-g,f\right)+\Gamma\left(g,f-g\right)\right\},f-g\right) \right|
\lesssim
\| g\|_{L^2_v H^{\ksob}_x}
\|f-g\|_{N^{s,\gamma}} ^2
\\
+
\|f-g\|_{L^2_{v,x}} \|f\|_{H^{\ksob}_x  N^{s,\gamma}}
\|f-g\|_{N^{s,\gamma}}. 
\end{multline*}
For the soft-potentials \eqref{kernelPsing}, we would use instead Lemma \ref{NonLinEstA}.  The Cauchy-Schwartz inequality (applied in the time variable) shows us that 
\begin{multline*}
\int_0^t d\tau 
\|f\|_{H^{\ksob}_x  N^{s,\gamma}}
\|f-g\|_{N^{s,\gamma}} 
\|f-g\|_{L^2_{v,x}}(\tau)
\\
\le
\sqrt{M_0}
\left(
\sup_{0\le \tau \le t}
\|f(\tau)-g(\tau)\|_{L^2_{v,x}}^2 
\int_0^t d\tau 
~ \|f(\tau)-g(\tau)\|_{N^{s,\gamma}} ^2
\right)^{1/2}.
\end{multline*}
We have just used the following fact, which follows from the local existence, that
\[
\sup_{0\le \tau \le t}
\|f(\tau)\|_{L^2_v H^{\ksob}_x}
+
\int_0^t ~ d\tau ~ \| f(\tau)\|_{H^{\ksob}_x N^{s,\gamma}}^2
\le M_0.
\]
And similarly for $g(t)$.
We use \eqref{coerc2ineq} to obtain
\[
(L(f-g),f-g)\ge \delta \|f-g\|_{N^{s,\gamma}} ^2-C\|f-g\|_{L^2(\domain \times B_C)}^2
\]
for some small $\delta > 0$.
We  multiply \eqref{difference} with $f-g$ and integrate over 
$[0,t]\times \domain\times \threed$ to achieve 
\begin{multline*}
\frac{1}{2}\|f(t)-g(t)\|_{L^2_{v,x}} ^2+ \delta \int_0^t ~ d\tau~ \|f(\tau)-g(\tau)\|_{N^{s,\gamma}} ^2
\\
\lesssim 
\sqrt{M_0}
\left(
\sup_{0\le \tau \le t}
\|f(\tau)-g(\tau)\|_{L^2_{v,x}} ^2
+
\int_0^t ~ d\tau~ \|f(\tau)-g(\tau)\|_{N^{s,\gamma}} ^2
\right)
\\
+
\int_0^t~ d\tau~ \|f(\tau)-g(\tau)\|_{L^2(\domain \times B_C)}^2.
\end{multline*}
We deduce $f \equiv g$ and the uniqueness from the Gronwall inequality.

To show the continuity of $\mathcal{ G}(f(t))$ in time, we  sum 
\eqref{xderi}, \eqref{xterm} over $|\alpha| \le \HARDxDER$ and integrate from $t_2$ to $t_1$ (rather than over $[0,t]$).  Then with $f^\iter=f^{\iter+1}=f$ we obtain 
\begin{gather*}
\left| \mathcal{ G}(f(t_1))-\mathcal{ G}(f(t_2)) \right|
=
\left| \frac 12\| f(t_1)\|_{ \spaceU }^2
-
\frac 12\| f(t_2)\|_{ \spaceU }^2
+
\int_{t_2}^{t_1} ~ d\tau ~\mathcal{D}(f(\tau)) \right| 
\\
\lesssim \left\{1+\sup_{{t_2}\le \tau \le {t_1}}\sqrt{\mathcal{ G}(f(\tau ))}
\right\}
\int_{t_2}^{t_1}
~ d\tau ~ 
\mathcal{D}(f(\tau))
\rightarrow 0,
\end{gather*}
as ${t_1}\rightarrow {t_2}$ since 
$
\mathcal{D}(f(\tau))
$ 
is
integrable in time.

We now explain the proof of positivity.   The key idea in this section is not new, and we give a brief outline.  
Previous works which obtain the positivity of strong solutions without cut-off include \cite{MR839310,MR2679369}.  For simplicity, we  use the argument from \cite{MR2679369}, however their initial data is  effectively in $f_0\in H^{M}_{\ell}$ for $M\ge 5$, since $F_0 = \mu + \sqrt{\mu} f_0$, and they study moderate angular singularities $0<s<1/2$.  
The key point is to consider a sequence of solutions $F^\epsilon$ to the Boltzmann equation \eqref{BoltzFULL} with the collision kernels \eqref{kernelQ}, \eqref{kernelP}, and \eqref{kernelPsing} except that $B$ is replaced by $B_\epsilon$ where the angular singularities in $B_\epsilon$ are removed but however $B_\epsilon \to B$ as $\epsilon \downarrow 0$.  We can observe that $F^\epsilon$ is positive using the argument, as in for instance \cite{MR2000470,MR2013332}.
If our initial data is in $H^{M}_{\ell}$, then 
since we have proved the uniqueness, we use the compactness procedure from \cite{MR2679369} to conclude that $F^\epsilon \to F$ as $\epsilon \downarrow 0$ and therefore
$F = \mu + \sqrt{\mu} f \ge 0$ if initially
$F_0 = \mu + \sqrt{\mu} f_0 \ge 0$.
The argument is finished by using the density of $H^{M}_{\ell}$ in the larger space $\spaceU(\domain\times \threed)$, standard approximation arguments, and our uniqueness theorem.  For the high singularities, 
$1/2\le s<1$, the positivity can be established by using  high derivative estimates from this paper  and following the same compactness procedure as in the low singularity case. 
\end{proof}

\subsection{Coercivity estimates for solutions to the non-linear equation}

The next step is to prove a general statement of the linearized $H$-theorem, which manifests itself as a 
coercive inequality.  These types of coercive estimates for the linearized collision operator were originally proved by Guo \cite{MR2000470,MR2013332} in the hard-sphere and cut-off regime.   The next theorem extends this estimate to the full range of inverse power law potentials $p>2$, and more generally to \eqref{kernelQ}, \eqref{kernelP}, and \eqref{kernelPsing}.

\begin{theorem}
\label{positive}  
Given the initial data $f_0 \in \spaceU$,
which satisfies \eqref{conservation} and the assumptions of Theorem \ref{local}. 
Consider the corresponding solution, $f(t,x,v)$, to \eqref{Boltz} which continues to satisfy \eqref{conservation}.
There is a small constant $M_0 >0$ such that if
\begin{gather}
\| f(t) \|^2_{\spaceU} \le M_0,
\label{smallL}
\end{gather}
then, further, there are universal constants $\delta>0$ and $C_2>0$ such that   
$$
\sum_{|\alpha| \le K} \| \{ {\bf I - P } \} \partial^\alpha f \|_{_{N^{s,\gamma}}}^2(t)
\ge 
\delta
\sum_{|\alpha| \le K} \| { \bf  P  } \partial^\alpha f \|_{_{N^{s,\gamma}}}^2(t) - C_2\frac{d\mathcal{I}(t)}{dt},
$$
where $\mathcal{I}(t)$ is the ``interaction functional'' defined precisely in \eqref{mainINTERACTION} below.
\end{theorem}

We prove this theorem by an analysis of the  macroscopic equations and also the local conservation laws.  The system of macroscopic equations comes from first expressing the hydrodynamic part ${\bf P}f$ through the microscopic part $\{{\bf I-P}\}f,$ up to the higher order term $\Gamma (f,f)$ as
\begin{equation}
\{\partial _t+v\cdot \nabla _x\}{\bf P}f=
-\partial _t \{{\bf I-P}\}f
+l(\{{\bf I-P}\}f)+\Gamma (f,f),  \label{macro}
\end{equation}
where 
\begin{equation}
l(\{{\bf I-P}\}f) \eqdef -\{v\cdot \nabla _x+L\}\{{\bf I-P}\}f.  \label{l}
\end{equation}
Notice that we have isolated the time derivative of the microscopic part.

To derive the macroscopic equations for ${\bf P}f$'s coefficients $a^f(t,x)$, $b^f_i(t,x)$ and 
$c^f(t,x)$, we use (\ref{hydro}) to expand the entries of left hand side of (\ref{macro}) as 
\begin{multline*}
\sum_{i=1}^\dim\left\{ v_i\partial_i c|v|^2+\{\partial_t c
+
\partial_i b_i\}v_i^2+
\{\partial_t  b_i+\partial_i a\}v_i\right\} \sqrt{\mu 
}
\\
+\sum_{i=1}^\dim\sum_{j>i}\{\partial_i b_j+\partial_j b_i\}v_iv_j ~ \sqrt{\mu}
+
\partial_t  a
~\sqrt{\mu}, 
\end{multline*}
where  $\partial_i=\partial _{x_i}$ above. 
For fixed ($t,x),$ this is an expansion of the left hand side of 
\eqref{macro} with respect to the following basis, $\{ e_k \}_{k=1}^{3n+1+n(n-1)/2}$, which consists of
\begin{gather}
\left( v_i|v|^2\sqrt{\mu } \right)_{1\le i\le n}, ~ 
\left(v_i^2\sqrt{\mu } \right)_{1\le i\le n},  ~
\left(v_iv_j\sqrt{\mu } \right)_{1\le i<j\le n}, ~
\left(v_i\sqrt{\mu } \right)_{1\le i\le n}, ~
\sqrt{\mu }.  
\label{base}
\end{gather}
From here one obtains the so-called
macroscopic equations 
\begin{eqnarray}
\nabla _xc&=& -\partial_t r_c+ l_c+\Gamma_c  \label{c} \\
\partial_t c+\partial_i b_i &=& -\partial_t r_i+ l_i+\Gamma_i  \label{bi} \\
\partial_i b_j+\partial_j b_i &=& -\partial_t r_{ij}+ l_{ij}+\Gamma_{ij} \quad (i\neq j)  \label{bij} \\
\partial_t b_i+\partial _i a &=& -\partial_t r_{bi}+ l_{bi}+\Gamma_{bi}
\label{ai} \\
\partial_t a &=& -\partial_t r_a+ l_a+\Gamma_a.  \label{adot}
\end{eqnarray}
For notational convenience we define the index set to be
$$
\mathcal{M} \eqdef \left\{c, ~i,~ \left( ij \right)_{i \ne j}, ~bi,~ a\left| ~ i, j = 1,\ldots,\dim\right. \right\}.
$$
This set $\mathcal{M}$ is just the collection of all indices in the macroscopic equations.  
Then for $\macroCOE \in  \mathcal{M}$ we have that each  $l_\macroCOE (t,x)$ are the coefficients of $l(\{{\bf I-P}\}f)$ with respect to the elements of \eqref{base}; similarly each $\Gamma_\macroCOE(t,x)$ and
$r_\macroCOE(t,x)$
are the coefficients of $\Gamma(f,f)$ and $\{{\bf I-P}\}f$ respectively.  Precisely, each element $r_\macroCOE$ can be expressed as 
$$
r_\macroCOE = \sum_{k} C_k^\macroCOE \langle \{{\bf I-P}\}f, e_k \rangle.
$$
All of the constants $C_k^\macroCOE$ above can be computed explicitly
although we do not give their precise form herein.  Each of the terms $l_\macroCOE$ and $\Gamma_\macroCOE$ can be computed similarly.  

The second set of equations we consider are the local conservation laws satisfied by $(a^f,b^f,c^f)$.  To derive these we 
 multiply \eqref{Boltz} by the collision invariants $\nullSpace$ in \eqref{null}
and integrate only in the velocity variables to obtain 
\begin{eqnarray*}
  \partial_t (a^f+\dim c^f)+\nabla_x\cdot b^f &=& 0,
  \\
  \partial_t b^f+\nabla_x (a^f+(\dim+2) c^f) &=& - \nabla_x\cdot \langle
  v\otimes v\sqrt{\mu},\{{\bf I - P}\}f\rangle,
  \\
   \partial_t(\dim a^f+n(n+2) ~ c^f)+(\dim +2)\nabla_x\cdot b^f 
   &=& 
   - \nabla_x\cdot  \langle |v|^2v\sqrt{\mu},\{{\bf I - P}\}f\rangle.
\end{eqnarray*}
Above we have used the moment values of the normalized global Maxwellian $\mu$: 
\begin{eqnarray*}
&&\langle 1, \mu\rangle=1,
\quad
\langle |v_j|^2, \mu\rangle=1,\ \ \langle |v|^2, \mu\rangle=\dim, ~
\langle |v_j|^2|v_i|^2, \mu\rangle=1, \ \ j\neq i,\\
&&\langle |v_j|^4, \mu\rangle=3, \ \ \langle |v|^2|v_j|^2,
\mu\rangle=\dim+2,\ \ \langle |v|^4, \mu\rangle=n(n+2).
\end{eqnarray*}
Comparing the first and third local conservation law results in
\begin{eqnarray}
  \partial_t a^f&=& 
    \frac{1}{2}\nabla_x \cdot \langle
    |v|^2v\sqrt{\mu},\{{\bf I - P}\}f\rangle,
  \label{cl.0}
  \\
  \partial_t b^f+\nabla_x (a^f+(\dim+2)c^f) &=& - \nabla_x\cdot \langle
  v\otimes v\sqrt{\mu},\{{\bf I - P}\}f\rangle,
  \label{cl.1}
  \\
    \partial_t c^f+\frac{1}{\dim}\nabla_x\cdot b^f  &=& -
    \frac{1}{2\dim}\nabla_x \cdot \langle
    |v|^2v\sqrt{\mu},\{{\bf I - P}\}f\rangle.
    \label{cl.2}
\end{eqnarray}
These are the local conservation laws that we will study below.
For the rest of this section, we concentrate on a solution $f$ to the Boltzmann equation \eqref{Boltz}.

\begin{lemma}
\label{average}Let $f(t,x,v)$ be the local solution to the Boltzmann equation \eqref{Boltz}
shown to exist in
Theorem \ref{local}
which satisfies 
\eqref{conservation}. Then we have 
\begin{equation}
 \int_{\domain}~ dx~ a^f(t,x) = \int_{\domain}~ dx~ b^f(t,x) = \int_{\domain} ~ dx~ c^f(t,x)  
 =
 0,
 \notag
\end{equation}
where $a^f$, $b^f=[b_1,b_2,b_3],$ $c^f$
are defined in (\ref{hydro}).
\end{lemma}

The proof of this lemma follows directly from the conservation of mass, momentum and energy \eqref{conservation}, using the cancellation that we just used in deriving the 
 conservation laws \eqref{cl.0}, \eqref{cl.1}, and \eqref{cl.2}.  
 In the following two lemmas, we establish the required estimates on the linear microscopic piece and then we estimate the non-linear higher order term.  With these lemmas, the macroscopic equations and the local conservation laws, we will prove Theorem \ref{positive}.

\begin{lemma}
\label{linear}
For any of the microscopic terms, $l_\macroCOE$, from the macroscopic equations
\[
\sum_{\macroCOE \in \mathcal{M}} \|l_\macroCOE\|_{H^{K-1}_x}
\lesssim
\sum_{|\alpha| \le K}\|\{{\bf I-P}\} \partial^\alpha f\|_{L^2_{\gamma+2s}(\domain\times\threed)}.
\]
\end{lemma}

\begin{proof} 
Recall 
$\{ e_k \}$, the basis in (\ref
{base}).
For fixed $(t,x)$,  it suffices to estimate  the $H^{K-1}_x$ norm of
$
\langle l(\{{\bf I-P}\}f), e_k \rangle.
$
We  use (\ref{l}) to expand out
$$
\langle \partial^\alpha l(\{{\bf I-P}\}f), e_k \rangle
=
-\langle v\cdot \nabla _x (\{{\bf I-P}\}\partial^\alpha f), e_k \rangle
-\langle L (\{{\bf I-P}\}\partial^\alpha f), e_k \rangle.
$$
Now for any $|\alpha |\le K-1$
\begin{multline*}
\| \langle v\cdot \nabla _x (\{{\bf I-P}\}\partial^\alpha f), e_k \rangle\|_{L^2_x}^2 
\lesssim 
\int_{\domain\times \threed} ~ dx dv ~
|e_k(v)| ~ 
|v|^2 ~ |\{{\bf I-P}\}\nabla_x\partial ^\alpha f|^2
\\
\lesssim 
\|\{{\bf I-P}\}\nabla _x\partial^\alpha f\|^2_{L^2_{\gamma+2s}(\domain\times\threed)}.
\end{multline*}
Here we have used the exponential decay of $e_k(v)$.

It remains to estimate the linear operator $L$.  
With the expression from \eqref{LinGam} and \eqref{coerc1ineqPREP2}, we have the following
\begin{multline*}
\| \langle L (\{{\bf I-P}\}\partial^\alpha f), e_k \rangle\|_{L^2_x}^2 
\lesssim 
\left\| ~\nsm \{{\bf I-P}\}\partial^\alpha f \nsm_{L^2_{\gamma+2s}}~  \nsm M \nsm_{L^2_{\gamma+2s}} \right\|^2_{L^2_{x}}
\\
\lesssim 
\left\| \{{\bf I-P}\}\partial^\alpha f   \right\|^2_{L^2_{\gamma+2s}(\domain\times\threed)}.
\end{multline*}
This completes the proof of our estimates for the $l_\macroCOE$.
\end{proof}

We now estimate the coefficients of the higher order term $\Gamma(f,f)$:

\begin{lemma}
\label{high}Let (\ref{smallL}) be valid for some $M_0>0.$ Then  for $\NgE$ we have
\[
\sum_{\macroCOE \in \mathcal{M}}
\|\Gamma_\macroCOE\|_{H^{K-1}_x}
\lesssim
\sqrt{M_0}\sum_{|\alpha |\le K} \| \partial ^\alpha f\|_{L^2_{\gamma+2s}(\domain\times\threed)}.
\]
\end{lemma}

\begin{proof}
As in the proof of Lemma \ref{linear}, to prove the estimate for $\Gamma_\macroCOE$,
it will suffice to estimate the $H^{K-1}_x$ norm of $\langle \Gamma(f,f), e_k\rangle$.  We apply 
\eqref{coerc1ineqPREP2}
from
Proposition \ref{upperBds}
 to see that for any $m\ge 0$
\begin{gather*}
\left\| \langle \Gamma(f,f), e_k\rangle\right\|_{H^{K-1}_x}
\lesssim
\sum_{|\alpha |\le K-1}\sum_{\alpha_1 \le \alpha}
\left\|
 \nsm \partial^{\alpha - \alpha_1} f \nsm_{L^2_{-m}} \nsm \partial^{\alpha_1} f \nsm_{L^2_{-m}} 
\right\|_{L^2_x}.
\end{gather*}
We obtain for $K \ge \ksob$ that
\begin{gather*}
\lesssim
 \| f\|_{L_{-m}^2H^{K}_x}
\sum_{|\alpha |\le K} \| \partial ^\alpha f\|_{L^2_{\gamma+2s}}
\lesssim
\sqrt{M_0}\sum_{|\alpha |\le K} \| \partial ^\alpha f\|_{L^2_{\gamma+2s}}.
\end{gather*}
The last inequalities follow from the embedding as in the remark of \eqref{sobolev}.
\end{proof}

We now prove the crucial positivity of $L$ for small solution $f(t,x,v)$ to the Boltzmann equation \eqref{Boltz}. The
conservation laws \eqref{conservation} will play an important
role. \\

\noindent {\it Proof of Theorem \ref{positive}}. 
We first of all notice from \eqref{hydro} that
$$
 \| { \bf  P  } \partial^\alpha f (t) \|_{N^{s,\gamma}}^2
\lesssim
  \|\partial^\alpha a(t)\|^2_{L^2_x}
  +
  \|\partial^\alpha b (t)\|^2_{L^2_x}
  +
\|  \partial^\alpha c(t)\|^2_{L^2_x}.
$$
Thus it will be sufficient to bound each of the terms on the right side above by 
$ \| \{ {\bf I - P } \} \partial^\alpha f (t)\|_{N^{s,\gamma}}^2$ 
plus the time derivative of the interaction functional, which is defined in \eqref{mainINTERACTION}.  Indeed, our  proof is devoted to establishing the following
\begin{gather}
\notag
\|a (t)\|_{H^K_x}^2
+
\|b(t)\|_{H^K_x}^2
+
\|c(t) \|_{H^K_x}^2
\lesssim
\sum_{|\alpha |\le K}
\|\{{\bf I-P}\} \partial ^\alpha  f(t)\|_{L^2_{\gamma + 2s}}^2
\\
+
M_0\sum_{|\alpha |\le K}\|\partial ^\alpha f(t)\|_{L^2_{\gamma + 2s}}^2
+
\frac{d\mathcal{I}(t)}{dt}.
\label{claimH}
\end{gather}
Clearly the second term on the right above can be neglected 
because of
\begin{gather*}
\sum_{|\alpha |\le K}\|\partial^\alpha f(t)\|_{L^2_{\gamma + 2s}}^2
\lesssim
\sum_{|\alpha |\le K}
\|{\bf P}\partial ^\alpha f(t)\|_{L^2_{\gamma + 2s}}^2
+
\sum_{|\alpha |\le K}\|\{{\bf I-P}\}\partial
^\alpha f(t)\|_{L^2_{\gamma + 2s}}^2
\\
\lesssim
\left\{\|a(t) \|_{H^K_x}
+
\|b(t)\|_{H^K_x}
+
\|c(t) \|_{H^K_x}\right\}^2
+
\sum_{|\alpha |\le K}
\|\{{\bf I-P}\}\partial ^\alpha f(t)\|_{L^2_{\gamma + 2s}}^2.
\end{gather*}
Thus \eqref{claimH} will imply Theorem \ref{positive} when $M_0$ is sufficiently
small.

To prove (\ref{claimH}), we estimate each of $a$, $b$, and $c$ individually with spatial derivatives of order $0<|\alpha |\le K$.  Then at the end of the proof we estimate the pure $L^2_x$ norm of $a$, $b$, and $c$ in a uniform way.
  We first  estimate $a(t,x)$. 
Consider any $|\alpha |\le K-1$.
By taking $\partial_i \partial^\alpha$ of (\ref{ai}) and summing over $i$, we get 
\begin{equation}
-\Delta \partial ^\alpha a
=
\frac{d}{dt} \left( \nabla \cdot \partial ^\alpha b \right)
+
\sum_{i=1}^\dim \left(\partial_t \partial_i\partial ^\alpha r_{bi} -\partial_i\partial ^\alpha \{l_{bi}+\Gamma_{bi}\} \right).  \label{div}
\end{equation}
Multiply with $\partial ^\alpha a$ to (\ref{div}) and integrate over $dx$ to obtain
\begin{gather*}
\|\nabla \partial ^\alpha a\|^2_{L^2_x}
\le 
\frac{d}{dt} \int_{\domain}  ~ dx ~ \left( \nabla \cdot \partial ^\alpha b \right)\partial^\alpha a(t,x)
+
\frac{d}{dt} \int_{\domain}  ~ dx ~ \partial_i\partial ^\alpha r_{bi} ~
\partial^\alpha a(t,x)
\\
- \int_{\domain}  ~ dx ~ \left( \nabla \cdot \partial ^\alpha b \right)
\partial_t \partial^\alpha a(t,x)
-
 \int_{\domain}  ~ dx ~ \partial_i\partial ^\alpha r_{bi} ~
\partial_t \partial^\alpha a(t,x)
\\
+
\|\partial ^\alpha \{l_{bi}+\Gamma_{bi}\}\|_{L^2_x}
\|\nabla \partial ^\alpha a\|_{L^2_x}.
\end{gather*}
Above we implicitly sum over $i=1,\ldots,\dim$. We define the interaction functional
$$
\mathcal{I}_a^\alpha(t) \eqdef 
\int_{\domain}  ~ dx ~ \left( \nabla \cdot \partial ^\alpha b \right)\partial^\alpha a(t,x)
+
\sum_{i=1}^\dim
\int_{\domain}  ~ dx ~ \partial_i\partial ^\alpha r_{bi} ~
\partial^\alpha a(t,x).
$$
We also use the local conservation law \eqref{cl.0}, to see that for any $\eta >0$, we have
\begin{gather*}
 \int_{\domain}  ~ dx ~ \left\{ \left| \left( \nabla \cdot \partial ^\alpha b \right)
\partial_t \partial^\alpha a(t,x) \right|
+ \left| \partial_i\partial ^\alpha r_{bi} ~
\partial_t \partial^\alpha a(t,x) \right| \right\}
\\
\le \eta \| \nabla \cdot \partial ^\alpha b \|^2_{L^2_x} + C_\eta  \| \{ {\bf I - P} \} \nabla \partial^\alpha f \|^2_{L^2_{\gamma + 2s}}.
\end{gather*}
We combine these last few estimates with Lemma \ref{linear} and \ref{high} to see that  
\begin{gather}
\notag
\|\nabla \partial ^\alpha a\|^2_{L^2_x}
-
\eta \| \nabla \cdot \partial ^\alpha b \|^2_{L^2_x} 
\lesssim
C_\eta
\sum_{|\alpha| \le K}\|\{{\bf I-P}\} \partial^\alpha f\|^2_{L^2_{\gamma+2s}}
+
\frac{d \mathcal{I}_a^\alpha  }{dt}
\\
+
M_0\sum_{|\alpha |\le K} \| \partial ^\alpha f\|_{L^2_{\gamma+2s}}^2.
\label{mainAest}
\end{gather}
This will be our main estimate for $a(t,x)$ with derivatives.

Next we estimate $c(t,x)$ from  (\ref{c}), with $|\alpha |\le K-1$.   We notice that
\begin{gather*}
\|\nabla \partial^\alpha c\|^2_{L^2_x}
\le 
C\left\{\|\partial ^\alpha l_c\|^2_{L^2_x}+\|\partial^\alpha \Gamma_c\|^2_{L^2_x} \right\} 
- 
\frac{d}{dt}\int_{\domain} ~ dx ~ \partial ^\alpha r_c(t,x) ~\cdot \nabla_x \partial ^\alpha c(t,x)
\\
-
\int_{\domain} ~ dx ~ \nabla_x \cdot\partial ^\alpha r_c(t,x) ~  \partial ^\alpha \partial_t c(t,x).
\end{gather*}
We now define another interaction functional as
$$
\mathcal{I}_c^\alpha(t) \eqdef 
- \int_{\domain} ~ dx ~ \partial ^\alpha r_c(t,x) ~ \cdot \nabla_x \partial ^\alpha c(t,x).
$$
Next we use the conservation law \eqref{cl.2} to obtain the following estimate
\begin{gather*}
\int_{\domain} ~ dx ~ \left| \nabla_x \cdot \partial ^\alpha r_c(t,x) ~  \partial ^\alpha \partial_t c(t,x) \right|
\le \eta \| \nabla \cdot \partial ^\alpha b \|^2_{L^2_x} + C_\eta  \| \{ {\bf I - P} \} \nabla \partial^\alpha f \|^2_{L^2_{\gamma + 2s}},
\end{gather*}
which holds for any $\eta >0$.
Combining these with Lemmas \ref{linear} and \ref
{high}, we see that
\begin{gather}
\notag
\|\nabla \partial ^\alpha c\|^2_{L^2_x}
-
\eta \| \nabla \cdot \partial ^\alpha b \|^2_{L^2_x} 
\lesssim
C_\eta
\sum_{|\alpha| \le K}\|\{{\bf I-P}\} \partial^\alpha f\|^2_{L^2_{\gamma+2s}}
+
\frac{d \mathcal{I}_c^\alpha  }{dt}
\\
+
M_0\sum_{|\alpha |\le K} \| \partial ^\alpha f\|_{L^2_{\gamma+2s}}^2.
\label{mainCest}
\end{gather}
This will be our main estimate for $c(t,x)$ with derivatives.

The last term to estimate with derivatives is $\nabla \partial ^\alpha b$.  
Suppose $|\alpha |\le K-1$, take $\partial_j$ of (\ref{bi}) and (\ref{bij}) and  sum on $j$.  It was shown in a nontrivial calculation from \cite{MR2000470}, using the elliptic structure of these equations and several symmetries, that
\begin{multline*}
\Delta \partial ^\alpha b_i
=
-\partial_i\partial_i\partial ^\alpha b_i
+
2\partial_i\partial ^\alpha l_i
+
2\partial_i\partial ^\alpha
\Gamma_i 
\\
+
\left\{
\sum_{j\neq i}
-\partial_i \partial^\alpha l_j-\partial_i\partial ^\alpha \Gamma_j
+
 \partial_j\partial ^\alpha l_{ij}
 +
 \partial_j\partial^\alpha \Gamma_{ij}
 -\partial_t \partial_{j} \partial^\alpha r_{ij}
\right\}.
\end{multline*}
We then multiply the whole expression by 
$\partial^\alpha b_i$ and integrate by parts to yield 
\begin{gather}
\label{sb}
\|\nabla \partial ^\alpha b_i\|^2
\le C\left\{\sum_{\macroCOE \in \mathcal{M}} 
\|\partial ^\alpha l_\macroCOE\|^2
+
\|\partial ^\alpha \Gamma_\macroCOE\|^2
\right\} 
+
\sum_{j\neq i} \int_{\domain} ~ dx ~ \partial_{j} \partial^\alpha r_{ij}
\partial_t \partial^\alpha b_i
\\
- 
\frac{d}{dt} \sum_{j\neq i} \int_{\domain} ~ dx ~ \partial_{j} \partial^\alpha r_{ij}\partial^\alpha b_i.
\notag
\end{gather}
We  define the last component of the interaction functional as
$$
\mathcal{I}_b^\alpha(t) \eqdef 
- 
 \sum_{j\neq i} \int_{\domain} ~ dx ~ \partial_{j} \partial^\alpha r_{ij}\partial^\alpha b_i.
$$
Using the conservation law \eqref{cl.1}, we estimate the term with a time derivative as
\begin{gather*}
\sum_{j\neq i} \int_{\domain} ~ dx ~  \left|  \partial_{j} \partial^\alpha r_{ij} 
\partial_t \partial^\alpha b_i(t,x) \right| 
\le \eta \left\{ \| \nabla  \partial ^\alpha a \|^2_{L^2_x} + \| \nabla  \partial ^\alpha c \|^2_{L^2_x}  \right\}
\\
+ C_\eta  \| \{ {\bf I - P} \} \nabla \partial^\alpha f \|^2_{L^2_{\gamma + 2s}},
\end{gather*}
which once again holds for any $\eta >0$.
Combining these  last few estimates with Lemmas \ref{linear} and \ref
{high}, we obtain
\begin{gather}
\notag
\|\nabla \partial ^\alpha b\|^2_{L^2_x}
-
 \eta \left\{ \| \nabla  \partial ^\alpha a \|^2_{L^2_x} + \| \nabla  \partial ^\alpha c \|^2_{L^2_x}  \right\}
\lesssim
C_\eta
\sum_{|\alpha| \le K}\|\{{\bf I-P}\} \partial^\alpha f\|^2_{L^2_{\gamma+2s}}
+
\frac{d \mathcal{I}_b^\alpha  }{dt}
\\
+
M_0\sum_{|\alpha |\le K} \| \partial ^\alpha f\|_{L^2_{\gamma+2s}}^2.
\label{mainBest}
\end{gather}
This is our main estimate for $b(t,x)$ with derivatives.

Now, with $\mathcal{I}_a^\alpha(t)$, $\mathcal{I}_b^\alpha(t)$ and $\mathcal{I}_c^\alpha(t)$ defined just above, we define the total interaction functional as
\begin{gather}
\mathcal{I}(t) 
\eqdef 
\sum_{|\alpha|\le K-1}
\left\{
\mathcal{I}_a^\alpha(t)
+
\mathcal{I}_b^\alpha(t)
+
\mathcal{I}_c^\alpha(t)\right\}.
\label{mainINTERACTION}
\end{gather}
Choosing say $\eta = 1/8$ and collecting 
\eqref{mainAest}, \eqref{mainCest}, \eqref{mainBest}, we have established 
\begin{multline}
\notag
 \| \nabla   a \|^2_{H^{K-1}_x}
 +
\|\nabla  b\|^2_{H^{K-1}_x}
+ \| \nabla   c \|^2_{H^{K-1}_x} 
\lesssim
\sum_{|\alpha| \le K}\|\{{\bf I-P}\} \partial^\alpha f\|^2_{L^2_{\gamma+2s}}
+
\frac{d \mathcal{I}  }{dt}
\\
+
M_0\sum_{|\alpha |\le K} \| \partial ^\alpha f\|_{L^2_{\gamma+2s}}^2.
\end{multline}
To finish \eqref{claimH}, it remains to estimate the terms without derivatives.

With the Poincar\'{e} inequality and Lemma \ref{average}, 
$a$ itself is bounded by 
\begin{gather*}
\|a\|
\lesssim 
\|\nabla a\|+\left|\int_{\domain}  ~ dx ~ a \right| 
=
\|\nabla a\|.
\end{gather*}
This is also bounded by the right side of  \eqref{claimH} by the last estimate above.
The estimates for $b_i(t,x)$ and $c(t,x)$ without derivatives are exactly the same.  This completes the main estimate \eqref{claimH} and the proof.
\qed \\

We are now ready to prove that global in time solutions to  \eqref{Boltz} exist. 

\subsection{Global existence and rapid decay}  With the coercivity estimate for non-linear local solutions from Theorem \ref{positive}, we show that these solutions must be global with the standard continuity argument.   Then we will prove rapid time decay.

A crucial step in this analysis is to prove the following energy inequalities:
\begin{equation}
\frac{d}{dt}\mathcal{E}_{\ell,m}(t)+\mathcal{D}_{\ell,m}(t) \leq C_{\ell, m} \sqrt{\mathcal{E}_{\ell}(t)}\mathcal{D}_{\ell}(t).
\label{MAINeINEQ}
\end{equation}
These  hold
for any $\ell\ge 0$ and $m = 0, 1, \ldots,K$.   We define the ``dissipation rate''   as
\begin{gather}
\notag
\mathcal{D}_{\ell,m} (t)\eqdef 
\sum_{|\beta|\le m}
\sum_{|\alpha | \le K-|\beta|}
\|\partial_\beta^\alpha f(t)\|_{N_{\ell - |\beta|}^{s,\gamma}}^2.
\end{gather}
In the case of Theorem \ref{SG1mainGLOBAL} with the hard potentials \eqref{kernelP}, we have $K = \HARDxDER$  and
 $\mathcal{D}_{\ell}\eqdef \mathcal{D}_{\ell,\HARDvDER}$.  Alternatively, in the case of Theorem \ref{mainGLOBAL} with the soft potentials \eqref{kernelPsing}, we write
  $\mathcal{D}_{\ell}\eqdef \mathcal{D}_{\ell,K} =   \| f(t) \|^2_{\nspace}$.  This unified notation will be useful in the following developments.  
  Furthermore the ``instant energy functional'' 
$
\mathcal{E}_{\ell,m}(t)
$
for a solution is a high-order norm which satisfies
$$
\mathcal{E}_{\ell,m}(t)
\approx
\sum_{|\beta|\le m}
\sum_{|\alpha |   \le K - |\beta|}
\|w^{\ell - |\beta|}\partial^{\alpha}_{\beta} f(t)\|^2_{L^2 (\mathbb{T}^n \times \mathbb{R}^n)}.
$$
Similarly in Theorem \ref{SG1mainGLOBAL} with the hard potentials \eqref{kernelP}, we have $K = \HARDxDER$ 
and
 $\mathcal{E}_{\ell}\eqdef \mathcal{E}_{\ell,\HARDvDER}$.  Alternatively, for Theorem \ref{mainGLOBAL} with the soft potentials \eqref{kernelPsing}, we write
  $\mathcal{E}_{\ell}\eqdef \mathcal{E}_{\ell,K}$.
We prove this energy inequality \eqref{MAINeINEQ} for a local solution via a simultaneous induction on both the order of the weights $\ell$ and on the number of velocity derivatives $m$.

The first inductive step is to prove \eqref{MAINeINEQ} for arbitrary spatial derivatives with $\ell=0$ and $|\beta| =0$. 
 We first fix $M_0\le 1$ such that both
Theorems \ref{local}  and \ref{positive} are valid.    
We now  take the spatial derivatives of $\partial^{\alpha}$ 
of \eqref{Boltz} to obtain
\begin{gather}
\label{initial}
\frac{1}{2}\frac{d}{dt}\|f(t)\|^2_{L^2_vH^K_x} 
+
\sum_{|\alpha |\le K}\left( L\partial^{\alpha}f,\partial^{\alpha}f\right) 
=
\sum_{|\alpha |\le K}
\left( \partial^{\alpha} \Gamma (f,f), \partial^{\alpha} f \right) . 
\end{gather}
With Lemmas \ref{NonLinEstLOW} and \ref{NonLinEstHIGH} we have
$$
\sum_{|\alpha |\le K}
\left( \partial^{\alpha} \Gamma (f,f), \partial^{\alpha} f \right) 
\lesssim
\sqrt{\mathcal{E}_{0}(t)} ~ 
\mathcal{D}_{0}(t).
$$
Now from Theorem \ref{lowerN} and then Theorem \ref{positive} we have
\begin{multline*}
\sum_{|\alpha |\le K}\left( L\partial^{\alpha}f,\partial^{\alpha}f\right) 
\ge \delta_0 \sum_{|\alpha |\le K} \| \{ {\bf I - P } \}\partial^{\alpha} f \|_{N^{s,\gamma}}^2
\\
\ge
\frac{\delta_0}{2} \sum_{|\alpha |\le K} \| \{ {\bf I - P } \}\partial^{\alpha} f \|_{N^{s,\gamma}}^2
+
\frac{\delta_0 \delta}{2} \sum_{|\alpha |\le K} \| {\bf P } \partial^{\alpha} f \|_{N^{s,\gamma}}^2
-\frac{\delta_0C_2 }{2} \frac{d\mathcal{I}(t)}{dt}. 
\end{multline*}
With $\tilde{\delta} \eqdef \min\left\{ \frac{\delta_0}{2} , \frac{\delta_0 \delta}{2} \right\}>0$
and
$C' \eqdef \delta_0C_2 >0$,
we  conclude that
\begin{gather}
\frac{1}{2}\frac{d}{dt}\left\{\|f(t)\|^2_{L^2_vH^K_x} - C'\mathcal{I}(t)\right\} +\tilde{\delta}  \mathcal{D}_{0,0}(t)  
\lesssim
\sqrt{\mathcal{E}_0(t)} \mathcal{D}_0(t).
\nonumber
\end{gather}
Now, by \eqref{mainINTERACTION}, 
for any $C'>0$ we can choose a large constant $C_1>0$  such that 
$$
\|f(t)\|^2_{L^2_vH^K_x} 
\le
\left( C_1 + 1 \right)\|f(t)\|^2_{L^2_vH^K_x} - C'\mathcal{I}(t)
\lesssim
 \|f(t)\|^2_{L^2_vH^K_x}. 
$$
Notice $C_1$  only depends upon the structure of the interaction functional and $C'$, but not on $f(t,x,v)$.
We then define the equivalent instant energy functional by
$$
\mathcal{E}_{0,0}(t)  \eqdef \left( C_1 + 1 \right)\|f(t)\|^2_{L^2_vH^K_x} - C' \mathcal{I}(t).
$$
We multiply 
\eqref{initial} 
by $C_1$ and add it to the previous differential inequality to conclude
\begin{gather}
\frac{d\mathcal{E}_{0,0}(t)}{dt} +\tilde{\delta}   \mathcal{D}_{0,0}(t)  
\lesssim
\sqrt{\mathcal{E}_0(t)} \mathcal{D}_0(t).
\nonumber
\end{gather}
In the last step we have used the positivity of $L \ge 0$.   We have thus established 
\eqref{MAINeINEQ} 
when $\ell = |\beta| =0$.

We turn to the case when $\ell >0$, but still $|\beta| =0$.  We only have pure spatial derivatives. 
With \eqref{coerc2ineq} in  Lemma \ref{DerCoerIneq}, we deduce 
that for a
$C>0$ and $\ell\geq0$
\begin{equation}
\left( w^{2\ell} L\partial^{\alpha}f, \partial^{\alpha}f\right)  
\gtrsim
\frac{1}{2}
\| \partial^{\alpha} f \|_{\spaceELLn}^{2}
-
C\| \partial^{\alpha}f\|_{L^2(B_C)}^{2}.
\label{newlower}
\end{equation}
Take the $\partial^\alpha$ derivative of
\eqref{Boltz}, then take the inner product of both sides with $w^{2\ell}   \partial^{\alpha}f$ and integrate to obtain the following: 
\[
\sum_{|\alpha|\leq K}\left(  \frac{1}{2}\frac{d}{dt}
\|w^\ell \partial^{\alpha} f(t)\|_{L^2}^{2}+\left(  w^{2\ell}L\partial^{\alpha}f,\partial^{\alpha
}f\right)  \right)  
\lesssim
\sqrt{{\mathcal{E}}_{\ell}(t)} {\mathcal{D}}_{\ell}(t).
\]
We have used Lemmas \ref{NonLinEstLOW} and \ref{NonLinEstHIGH} to estimate the non-linear term.   We apply the coercive lower bound \eqref{newlower}.  Then we add \eqref{MAINeINEQ} for the case $\ell = |\beta| =0$ multiplied by a suitably large constant $C_2$ to the result.  This yields
\[
\frac{d}{dt}{\mathcal{E}}_{\ell,0}(t) 
+
{\mathcal{D}}_{\ell,0}(t)
\lesssim
\sqrt{{\mathcal{E}}_{\ell}(t)} {\mathcal{D}}_{\ell}(t),
\]
where
$
{\mathcal{E}}_{\ell,0}(t) 
\eqdef
\frac{1}{2}
\sum_{|\alpha|\leq K}  
\|w^\ell \partial^{\alpha} f(t)\|_{L^2}^{2}
+
C_2{\mathcal{E}}_{0,0}(t).
$
Since this is indeed an instant energy functional, we have
\eqref{MAINeINEQ} 
when $\ell >0$ and  $|\beta| =0$.

The final step is the case when $\ell \ge 0$, but also $|\beta| =m+1>0$. We suppose that 
\eqref{MAINeINEQ} 
holds for any
 $\ell \ge 0$ and any  $|\beta| \le m$.
We take $\partial_{\beta}^{\alpha}$ of \eqref{Boltz} to obtain
\begin{equation}
\left(\partial_{t}+v\cdot\nabla_{x}\right)\partial_{\beta}^{\alpha}f+\partial
_{\beta}L\partial^{\alpha}f=-
\sum_{|\beta_{1}|=1} C^{\beta}_{\beta_{1}} ~ 
\partial_{\beta_{1}}v\cdot\nabla_{x}\partial_{\beta-\beta_{1}}^{\alpha
}f+\partial_{\beta}^{\alpha}\Gamma(f,f). 
\label{bgeBIG}
\end{equation}
We use Cauchy's inequality for $\eta>0$, since $|\beta_{1}|=1$ we have
\begin{multline*}
\left\vert \left(  w^{2\ell - 2|\beta|} \partial_{\beta_{1}}v\cdot\nabla
_{x}\partial_{\beta-\beta_{1}}^{\alpha}f,\partial_{\beta}^{\alpha}f\right)
\right\vert 
\\
\leq
\| w^{\ell - |\beta| - 1/2}  \partial_{\beta}^{\alpha}f(t)\|_{L^2}
\|w^{\ell - |\beta| + 1/2} \nabla_{x}\partial_{\beta-\beta_{1}}^{\alpha}f(t)\|_{L^2}
\\
\le
\eta
||\partial_{\beta}^{\alpha}f(t)||_{N^{s,\gamma}_{\ell - |\beta|}}^{2}
+
C_{\eta}||\nabla_{x}\partial_{\beta-\beta_{1}}^{\alpha}f||_{N^{s,\gamma}_{\ell - |\beta-\beta_1|}}^{2}.
\end{multline*}
Estimates using this particular trick were already seen in \cite{MR2013332}.

Now we multiply \eqref{bgeBIG} with $w^{2\ell -2|\beta|}\partial^{\alpha}f$ and integrate.  We estimate the non-linear term of the result with Lemmas \ref{NonLinEstLOW} and \ref{NonLinEstHIGH}.  
 With \eqref{coerc1ineq} we estimate from below the linear term $\left(  w^{2\ell - 2|\beta|} \partial_{\beta}\{L\partial^{\alpha}f\}, \partial_{\beta}^{\alpha}f\right)$.  With these inequalities above, we
conclude \eqref{MAINeINEQ} for $|\beta| = m+1$, but only after adding  to the inequality a suitably large constant times  \eqref{MAINeINEQ} for $|\beta| = m$ similar to the previous cases.  This establishes \eqref{MAINeINEQ} in general by induction.
From here we can conclude global existence  using the standard continuity argument.  
It remains to establish the time decay rates.  For the soft potentials we
use the  argument from \cite{MR2209761}.
Exponential time decay for the soft potentials, as in \cite{MR2366140}, may also be feasible.

If $\|f_0\|_{\spaceU}^{2}$ is sufficiently small, from 
\eqref{MAINeINEQ} for $m = K$ and $\ell \ge 0$,  we have 
\begin{equation}
\frac{d}{dt}\mathcal{E}_{\ell}(t)+\delta\mathcal{D}_{\ell}(t) \leq 0,
\quad
\exists \delta >0.
\notag
\end{equation}
For the hard potentials, $\gamma + 2s \ge 0$, exponential time decay follows directly from 
$\mathcal{D}_{\ell}(t)\gtrsim \mathcal{E}_{\ell}(t)$. But for the soft potentials $\gamma + 2s <0$,
the problem is that for fixed $\ell,$ the non-derivative part of the dissipation rate $\mathcal{D}_{\ell}(t)$ is clearly
weaker than the instant energy $\mathcal{E}_{\ell}(t)$.  In particular, we only have $\mathcal{D}_{\ell}(t)\gtrsim \mathcal{E}_{\ell-1}(t)$.

We interpolate with stronger
norms  to overcome this difficulty. Fix $\ell\ge 0$ and $m>0$.  Interpolation  between the weight functions $w^{2\ell-2}(v)$ and $w^{2\ell+2m}(v)$ yields
\[
\mathcal{E}_{\ell}(t)
\lesssim
 \mathcal{E}_{\ell-1}^{m/(m+1)}(t)\mathcal{E}_{\ell+m}^{1/(m+1)}(t)
\lesssim
\mathcal{D}_{\ell}^{m/(m+1)}(t)\mathcal{E}_{\ell+m}^{1/(m+1)}(0).
\]
The last inequality follows from 
$
\mathcal{E}_{\ell+m}(t)
\lesssim
\mathcal{E}_{\ell+m}(0).
$
Then for some $C_{\ell,m}>0,$
\[
\frac{d}{dt}\mathcal{E}_{\ell}(t)+C_{\ell,m}{\mathcal{E}}_{\ell+m}^{-1/m}(0)~ \mathcal{E}_{\ell}^{(m+1)/m}(t)\leq0.
\]
It follows that 
$
-m~ d(\mathcal{E}_{\ell}(t))^{-1/m}/dt
\leq
-C_{\ell,m}\left(  {\mathcal{E}}_{\ell+m}(0)\right)  ^{-1/m}.
$ 
Integrate over $[0,t]$:
\[
m\left( \mathcal{E}_{\ell}(0)\right)^{-1/m}-m\left( \mathcal{E}_{\ell}(t)\right)^{-1/m}\leq-\left(  {\mathcal{E}}%
_{\ell+m}(0)\right)  ^{-1/m} ~ C_{\ell,m} ~ t.
\]
Hence
\[
\left( \mathcal{E}_{\ell}(t)\right)^{-1/m}\geq t\frac{C_{\ell,m}}{m}\left(  {\mathcal{E}}_{\ell+m}(0)\right)  ^{-1/m}+\{\mathcal{E}_{\ell}(0)\}^{-1/m}.
\]
Since we can assume $\mathcal{E}_{\ell}(0)\lesssim {\mathcal{E}}_{\ell+m}(0)$, the rapid decay
thus follows. \hfill {\bf Q.E.D.}

\appendix
\section{Carleman's representation and the dual formulation}
\label{secAPP:HSr}

In this Appendix \ref{secAPP:HSr} we develop two Carleman \cite{MR1555365} type representations which are used crucially in our main text.  
We consider the general expression
$$
\tilde{\mathcal{C}}(v_*) = \int_{\mathbb{R}^n}dv ~\Phi(|v-v_*|) 
\int_{\mathbb{S}^{n-1}} d\sigma ~
b\left(\ip{k}{\sigma}\right) ~
H(v,v_*, v^\prime, v^\prime_*),
$$
with $k = \frac{v - v_*}{|v - v_*|}$ and the usual post-collisional velocities  $(v^\prime, v^\prime_*)$ are given by \eqref{sigma}.  The functions $b$ and $\Phi$ are generally given by \eqref{kernelQ}, \eqref{kernelP} and \eqref{kernelPsing}.  For the purposes of deriving the expression in Proposition \ref{carlemanA} it suffices to suppose that both of these functions are smooth.  The general expressions can then be deduced from these formulas by the usual approximation procedures.
We have the following representation formula:

\begin{proposition} 
\label{carlemanA}
Let $H: \mathbb{R}^n \times \mathbb{R}^n \times \mathbb{R}^n \times \mathbb{R}^n \to \mathbb{R}$ be a smooth, rapidly decaying function at infinity.  Then we have
 \begin{gather*}
\tilde{\mathcal{C}}(v_*) 
= 
2^{n-1}
\int_{\mathbb{R}^n}dv '~
\int_{E^{v'}_{v_*}}d\pi_{v} ~
\frac{\Phi(|v-v_*|)  }{|v' - v_*|^{}}
\frac{b\left(\ang{\frac{v-v_*}{|v-v_*|}, \frac{2v' - v - v_*}{|2v' - v - v_*|} } \right)}{|v - v_*|^{n-2}}
H.
  \end{gather*}
  Above 
  $
  H = H(v, v_*, v' , v+v_*- v')
$
and
  $
  E^{v'}_{v_*}
  $
  is the hyperplane
  $$
  E^{v'}_{v_*} \eqdef \left\{ v\in \mathbb{R}^n : \ang{v_*-v', v - v'} =0 \right\}.
  $$
Then $d\pi_{v} $ denotes the Lebesgue measure on this hyperplane.
\end{proposition}

We also illustrate a Carleman-type representation for 
$$
\mathcal{C}(v) = \int_{\mathbb{R}^n}dv_* ~\Phi(|v-v_*|) 
\int_{\mathbb{S}^{n-1}} d\sigma ~
b\left(\ip{k}{ \sigma}\right) ~
H(v, v_*, v^\prime, v^\prime_*),
$$
with the same notation and the same comments as in the last case.


\begin{proposition} 
\label{carlemanV2}
  Let $H: \mathbb{R}^n \times \mathbb{R}^n \times \mathbb{R}^n \times \mathbb{R}^n \to \mathbb{R}$ be a smooth, rapidly decaying function at infinity.  Then we have
 \begin{gather*}
\mathcal{C}(v) 
= 
2^{n-1}
\int_{\mathbb{R}^n}dv' ~
\int_{E^{v}_{v'}}d\pi_{v'_*} ~
\frac{\Phi(|2v - v' - v'_*|) }{|v - v'|}
\frac{b\left(\ang{\frac{2v - v' - v'_*}{|2v - v' - v'_*|}, \frac{v' - v'_*}{|v' - v'_*|} }\right) }{|v' - v'_*|^{n-2}} ~
H.
  \end{gather*}
  Above 
  $
  H = H(v, v'_*+ v' - v, v' , v_*' ),
  $
and
  $
  E^{v}_{v'}
  $
  is the hyperplane
  $$
  E^{v}_{v'} \eqdef \left\{ v'_*\in \mathbb{R}^n : \ang{v' - v, v_*' -v} =0 \right\}.
  $$
Then $d\pi_{v_*'} $ denotes the Lebesgue measure on this hyperplane.
\end{proposition}

Our expressions above may be at some degree of variance from the usual Carleman representation, however they are of the same form and derived in the same way; a clear proof can be found in \cite{09-GPV}.  
With these expressions we will derive a Dual Representation for the non-linear operator \eqref{gamma0}.

\subsection*{Dual Representation}
We initially suppose that $\int_{{\mathbb S}^{n-1}} d \sigma ~ |b( \ang{k, \sigma} )| < \infty$ and that the kernel $b$ has mean zero, i.e., $\int_{{\mathbb S}^{n-1}} d \sigma ~ b(\ang{k,\sigma}) = 0$. 
Then after the pre-post change of variables we can express \eqref{gamma0} as
\begin{gather*}
 \ang{\Gamma(g,h),f}= \int_{\threed} \! \! dv  \int_{\threed} \! \! dv_* \int_{{\mathbb S}^{n-1}} \! \! d \sigma
 ~ \Phi(|v-v_*|) b \left( \ang{k, \sigma} \right) g_* h \left(M_*' f' - M_* f \right) \nonumber 
  \\
  = \int_{\threed} dv ~ \int_{\threed} dv_* ~ \int_{\sph}  d \sigma ~ \Phi(|v-v_*|) b \left( \ang{k, \sigma} \right) g_* h M_*' f'. \nonumber 
 \end{gather*}
This follows from the vanishing of $\int_{{\mathbb S}^{n-1}} b( \ang{k, \sigma} ) d \sigma$.  With Proposition \ref{carlemanA}, this is
 \begin{gather*}
 = 
 2^{n-1} \int_{\threed} \!  \!dv_* \! \int_{\threed} \! \! dv' \!   \int_{E^{v'}_{v_*}}d\pi_{v} ~
\Phi(|v-v_*|)  
\frac{b\left(\ang{\frac{v-v_*}{|v-v_*|}, \frac{2v' - v - v_*}{|2v' - v - v_*|} } \right)}{|v' - v_*|~|v - v_*|^{n-2}} g_* h M_*' f'. 
 \nonumber
\end{gather*}
In the above formulas, we take $M'_* = M(v+v_* - v')$.  From the identity (on $E_{v_*}^{v'}$)
\[ \ang{\frac{v-v_*}{|v-v_*|}, \frac{2v' - v - v_*}{|2v' - v - v_*|} } = \frac{|v' - v_*|^2 - |v - v'|^2}{|v - v'|^2 + |v' - v_*|^2}, \]
we observe that
\begin{align*}
\int_{E_{v_*}^{v'}} d\pi_{v} ~ &  b \left( \ang{\frac{v-v_*}{|v-v_*|}, \frac{2v' - v - v_*}{|2v' - v - v_*|} } \right) \frac{|v'-v_*|^{n-1}}{|v-v_*|^{2n-2}} \\
& = 
\int_{{\mathbb S}^{n-2}} d \sigma 
\int_0^\infty r^{n-2} \ dr \ b \left( \frac{|v'-v_*|^2 - r^2}{|v' - v_*|^2 + r^2} \right) \frac{|v' - v_*|^{n-1}}{(r^2 + |v'-v_*|^2)^{n-1}} = 0,
\end{align*}
by a change of variables since $\int_{-1}^1 dt \ b(t) (1-t^2)^{\frac{n-3}{2}} = 0$ 
(following from the cancellation condition on $\mathbb{S}^{n-1}$) and 
\[ \frac{d}{dr} \left[ \frac{|v'-v_*|^2 - r^2}{|v' - v_*|^2 + r^2} \right] = 
\frac{-4 r |v' - v_*|^2}{(r^2 + |v' - v_*|^2)^2} \]
\[ \left( 1 - \left(\frac{|v'-v_*|^2 - r^2}{|v' - v_*|^2 + r^2}\right)^2 \right)^\frac{n-3}{2} = \frac{(2 r |v'-v_*|)^{n-3}}{(r^2 + |v' - v_*|^2)^{n-3}}. \]  
In particular, with $\tilde{B}$ defined in \eqref{kernelTILDE}, this implies
\[ 
\int_{E_{v_*}^{v'}}  d \pi_v~ \tilde{B} ~
\frac{ \Phi(|v'-v_*|)}{\Phi(|v-v_*|)}  \frac{|v'-v_*|^n}{|v-v_*|^n} 
~ g_* h'  f' ~ M_* = 0. 
\]
We subtract this expression from the Carleman representation just written for $\ang{\Gamma(g,h),f}$, to see that $\ang{\Gamma(g,h),f}$ must also equal \eqref{dualOPdef} with kernel \eqref{kernelTILDE} and \eqref{opGlabel}.  This will be called the ``dual representation.''

The claim is now that this representation holds even when the mean value of the singular kernel $b(\ang{k, \sigma})$ from \eqref{kernelQ}  is not zero.  To see this claim, suppose $b$ integrable but without mean zero.  Then define
\[ 
b_\epsilon(t) = b(t) - \ind_{[1-\epsilon,1]}(t) \int_{-1}^1 dt ~ b(t) ~ (1-t^2)^{\frac{n-3}{2}} \left( \int_{1-\epsilon}^1 dt ~ (1-t^2)^{\frac{n-3}{2}}\right)^{-1}. \
\]
As a function on ${\mathbb S}^{n-1}$, $b_\epsilon$ will clearly have a vanishing integral.  However, given arbitrary $f$, $g$ and $h$ which are Schwartz functions, it is not hard to see that
\[ 
\left| \ang{\Gamma(g,h),f} - \ang{\Gamma_\epsilon(g,h),f} \right| \rightarrow 0,
\quad
\epsilon \rightarrow 0. 
\]
Above $\Gamma_\epsilon$ is the non-linear term \eqref{gamma0} formed with $b_\epsilon(t)$ in place of $b(t)$.  This convergence holds 
because  cancellation guarantees that the integrand vanishes on the set defined by $\ang{k,\sigma} = 1$.  Moreover, an additional cutoff argument shows that the equality also holds provided that $b(t)$ satisfies \eqref{kernelQ}; 
the higher-order cancellation is preserved because $\frac{v'-v_*}{|v-v_*|}$ possesses radial symmetry in $v-v'$.

The ``dual representation'' deserves its name because if one defines
\begin{align*}
T_g f (v) & \eqdef \int_{\threed} dv_* \int_{{\mathbb S}^{n-1}} d \sigma ~ B ~ g_* ~  \left(M_*' f' - M_* f \right), 
\\
T_g^* h (v') & \eqdef  \int_{\threed} dv_* \int_{E_{v_*}^{v'}} d \pi_{v} ~ \tilde{B} ~  g_*
~ \left(  M_*' h   -    \frac{\Phi(v'-v_*) |v' - v_*|^n}{\Phi(v-v_*) |v-v_*|^n} M_* h' \right),
\end{align*} 
then 
\begin{equation}
\label{3dualZ}
\ang{\Gamma(g,h),f}
=
\ang{ T_g f ,h} 
= 
\ang{ f, T_g^* h }. 
\end{equation}
Note that the last inner product above represents an integration over $dv'$ whereas the first two inner products above represent integrations over $dv$.

The advantage of this representation is that $T_g f$ and $T_g^* h$ both depend on $g$ in a fairly elementary way.  
This allows, for example, the trilinear form $\ang{\Gamma(g,h),f}$ to be understood as a superposition of bilinear forms in $h$ and $f$.

\subsection*{Acknowledgments} 
We would like to thank the referee for useful comments which helped us to improve the presentation.

\end{document}